\documentclass[11pt]{article}
\usepackage[utf8]{inputenc}
\usepackage[a4paper]{geometry}
\usepackage{color}
\usepackage{array}
\usepackage{units}
\usepackage{textcomp}
\usepackage{enumitem}
\usepackage{multirow}
\usepackage{amsmath}
\usepackage{amsthm}
\usepackage{amssymb}
\usepackage{graphicx}
\usepackage[unicode=true,pdfusetitle,
 bookmarks=true,bookmarksnumbered=false,bookmarksopen=false,
 breaklinks=false,pdfborder={0 0 0},pdfborderstyle={},backref=false,colorlinks=true]
 {hyperref}
\hypersetup{
 linkcolor=blue,  citecolor=blue, urlcolor=blue}

\makeatletter

\newcommand{\noun}[1]{\textsc{#1}}
\providecommand{\tabularnewline}{\\}

\numberwithin{equation}{section}
\numberwithin{figure}{section}
\numberwithin{table}{section}
\theoremstyle{plain}
\newtheorem{thm}{\protect\theoremname}[section]
\theoremstyle{remark}
\newtheorem*{acknowledgement*}{\protect\acknowledgementname}
\theoremstyle{definition}
\newtheorem{defn}[thm]{\protect\definitionname}
\theoremstyle{remark}
\newtheorem{rem}[thm]{\protect\remarkname}
\theoremstyle{plain}
\newtheorem{lem}[thm]{\protect\lemmaname}
\theoremstyle{plain}
\newtheorem{prop}[thm]{\protect\propositionname}
\theoremstyle{definition}
\newtheorem{example}[thm]{\protect\examplename}
\theoremstyle{plain}
\newtheorem{cor}[thm]{\protect\corollaryname}
\theoremstyle{plain}
\newtheorem*{thm*}{\protect\theoremname}
\theoremstyle{plain}
\newtheorem{conjecture}[thm]{\protect\conjecturename}

\usepackage{graphicx}
\usepackage{setspace}
\usepackage{url}
\usepackage[perpage,symbol]{footmisc}
\usepackage{multicol}
\usepackage{etoolbox}
\usepackage{bbm}
\usepackage{enumitem}
\usepackage{caption}
\usepackage[capbesideposition={left,center}]{floatrow}
\usepackage{etoolbox}
\usepackage{fancyhdr}
\usepackage[abbrev,alphabetic]{amsrefs}

\date{}

\DefineFNsymbols*{lamportnostar}[math]{{(\dagger)}{(\ddagger)}{(\S)}{(\P)}{(\|)}{(\dagger\dagger)}{(\ddagger\ddagger)}}
\setfnsymbol{lamportnostar}
\DeclareMathOperator{\dist}{dist}
\DeclareMathOperator{\diag}{diag}

\DeclareMathOperator{\Ind}{Ind}

\DeclareMathOperator{\supp}{supp}

\DeclareMathOperator{\Hom}{Hom}

\DeclareMathOperator{\Spec}{Spec}
\DeclareMathOperator{\ord}{ord}

\DeclareMathOperator{\tr}{tr}
\DeclareMathOperator{\col}{col}
\DeclareMathOperator{\minmax}{minmax}
\newcommand{\one}{\mathbbm{1}}
\newcommand{\Mod}[1]{\,\left(\textup{mod}\;#1\right)}

\theoremstyle{plain}

\theoremstyle{definition}
\theoremstyle{definition}

\newtheorem{rem}[thm]{Remark}
\newtheorem*{rems*}{Remarks}
\newtheorem*{disc*}{Discussion}
\setitemize[0]{itemsep=0pt,leftmargin=10pt,itemindent=0pt}
\setlist[enumerate]{itemsep=0pt,leftmargin=*}

\allowdisplaybreaks[2]

\newcommand{\B}{\mathcal{B}}

\AtBeginDocument{%
   \def\MR#1{}
}
\newcommand{\lqu}[2]{\raisebox{-1pt}{$#1$\raisebox{2pt}{$\backslash#2$}}}

\makeatother

\providecommand{\acknowledgementname}{Acknowledgement}
\providecommand{\conjecturename}{Conjecture}
\providecommand{\corollaryname}{Corollary}
\providecommand{\definitionname}{Definition}
\providecommand{\examplename}{Example}
\providecommand{\lemmaname}{Lemma}
\providecommand{\propositionname}{Proposition}
\providecommand{\remarkname}{Remark}
\providecommand{\theoremname}{Theorem}

\begin{document}
\title{Ramanujan complexes and Golden Gates in $PU(3)$}
\author{Shai Evra and Ori Parzanchevski}
\maketitle
\begin{abstract}
In a seminal series of papers from the 80's, Lubotzky, Phillips and
Sarnak applied the Ramanujan-Petersson Conjecture for $GL_{2}$ (Deligne's
theorem), to a special family of arithmetic lattices, which act simply-transitively
on the Bruhat-Tits trees associated with $SL_{2}(\mathbb{Q}_{p})$.
As a result, they obtained explicit Ramanujan Cayley graphs from $PSL_{2}\left(\mathbb{F}_{p}\right)$,
as well as optimal topological generators (``Golden Gates'') for
the compact Lie group $PU(2)$. In higher dimension, the naive generalization
of the Ramanujan Conjecture fails, due to the phenomenon of endoscopic
lifts. In this paper we overcome this problem for $PU_{3}$ by constructing
a family of arithmetic lattices which act simply-transitively on the
Bruhat-Tits buildings associated with $SL_{3}(\mathbb{Q}_{p})$ and
$SU_{3}(\mathbb{Q}_{p})$, while at the same time do not admit any
representation which violates the Ramanujan Conjecture. This gives
us Ramanujan complexes from $PSL_{3}(\mathbb{F}_{p})$ and $PSU_{3}(\mathbb{F}_{p})$,
as well as golden gates for $PU(3)$.
\end{abstract}

\section{Introduction}

In the series of papers \cite{Lubotzky1986ExplicitexpandersRamanujan,lubotzky1986hecke,lubotzky1987hecke,LPS88},
Lubotzky, Phillips and Sarnak presented two remarkable explicit constructions:
(i) infinite families of $(p+1)$-regular Ramanujan graphs, for any
prime $p\equiv1\Mod{4}$, and (ii) topological generators of $PU(2)$
with optimal covering properties. In recent years both results were
extended in different directions. In \cite{li2004ramanujan,Lubotzky2005a}
the notion of Ramanujan graphs was generalized to \emph{Ramanujan
complexes} of higher dimension (see §\ref{subsec:Ramanujan-graphs-complexes}
below), and explicit constructions were given in \cite{Lubotzky2005b,sarveniazi2007explicit}.
In \cite{sarnak2015letter,Parzanchevski2018SuperGoldenGates}, motivated
by application to quantum computation, the notion of optimal topological
generators of $PU(2)$ was strengthened to the notion of golden gate
and super golden gate sets (see §\ref{subsec:Golden-Gates} below),
and explicit constructions were presented for $PU(2)$. For more on
Ramanujan complexes and expansion in Lie groups we refer the reader
to the surveys \cite{Lubotzky2013,Breuillard2018Expansionsimplegroups,lubotzky2017high}.

The work of Lubotzky, Phillips and Sarnak (henceforth, LPS) relies
on two classical results from number theory, both pertaining to the
algebraic group of invertible Hamilton quaternions $\mathbb{H}$ over
$\mathbb{Q}$. The first is the Ramanujan conjecture (Deligne's Theorem
for $PGL_{2}/\mathbb{Q}$ combined with the Jacquet-Langlands correspondence
between $GL_{2}$ and $\mathbb{H}^{\times}$), and the second is Jacobi's
four-square theorem, which leads to a construction of a $p$-arithmetic
group which acts simply-transitively on a Bruhat-Tits tree. Surprisingly,
the famous constructions of LPS were never fully generalized to any
other algebraic group over $\mathbb{Q}$. We note that while the work
of \cite{Lubotzky2005a,Lubotzky2005b} does generalize to higher dimensions
the explicit construction of Ramanujan graphs, it does not yield golden
gate sets, as it is achieved over the field of positive characteristic
$\mathbb{F}_{q}(t)$.

In this paper we establish a Ramanujan type result for $PGU_{3}/\mathbb{Q}$,
and construct $p$-arithmetic lattices which act simply-transitively
on the one and two-dimensional buildings associated with $SU_{3}(\mathbb{Q}_{p})$.
This allows us to present golden gate sets for $PU(3)$, as well as
new explicit constructions of $2$-dimensional Ramanujan complexes.
In a forthcoming paper \cite{Ballantine2018ExplicitCayleyRamanujan}
we show how similar results lead to an explicit construction of $(p^{3}+1,p+1)$-biregular
Ramanujan graphs. Our main result in this paper is the following.
\begin{thm}
\label{thm:main}Let $p$ and $q$ be distinct odd primes. Denote
\begin{align}
p' & :=\begin{cases}
p & p\equiv1\Mod{4}\\
p^{2} & p\equiv3\Mod{4},
\end{cases}\qquad G_{q}:=\begin{cases}
PGL_{3}\left(\mathbb{F}_{q}\right) & q\equiv1\Mod{4}\\
PGU_{3}\left(\mathbb{F}_{q}\right) & q\equiv3\Mod{4},
\end{cases}\label{eq:pprime-Gp}\\
Y_{q} & :=\begin{cases}
\mathbb{P}^{2}(\mathbb{F}_{q}) & q\equiv1\Mod{4}\\
\left\{ v\in\mathbb{P}^{2}(\mathbb{F}_{q}[i])\,\middle|\,v^{*}\!\cdot\!v=0\right\}  & q\equiv3\Mod{4}
\end{cases}\nonumber 
\end{align}
where $Y_{q}$ is considered as a $G_{q}$-set, and define 
\begin{align}
S_{p} & :=\Bigg\{ g\in M_{3}\left(\mathbb{Z}\left[i\right]\right)\,\Bigg|\,{gg^{*}=p'I,\;g\text{ is not scalar},\atop g\equiv\left(\begin{smallmatrix}1 & * & *\\
* & 1 & *\\
* & * & 1
\end{smallmatrix}\right)\Mod{2+2i}}\Bigg\},\label{eq:set-Sp}\\
S_{p,q} & :=S_{p}\Mod{q}\overset{{\scriptscriptstyle (\star)}}{\subseteq}G_{q}\qquad\text{and}\qquad G_{p,q}:=\left\langle S_{p,q}\right\rangle ,\nonumber 
\end{align}
where $(\star)$ implies mapping $i$ to $\sqrt{-1}\in\mathbb{F}_{q}$
when $q\equiv1\Mod{4}$. Then:
\begin{enumerate}
\item The set $S_{p}$ is a golden gate set for $PU_{3}(\mathbb{Z})\backslash PU(3)$.
\item If $p\equiv1\Mod{4}$ then the Schreier graph\footnote{Recall that the Schreier graph $Sch(X,S)$ has vertex set $X$, and
the edges $\{(x\rightarrow sx)|x\in X,s\in S\}$.} $Y^{p,q}:=Sch(Y_{q},S_{p,q})$ is the one-skeleton of a two-dimensional
Ramanujan complex.
\item If $p\equiv3\Mod{4}$ then $Y^{p,q}=Sch(Y_{q},S_{p,q})$ is the graph
induced by non-backtracking 2-walks on the left side of a $(p^{3}+1,p+1)$-biregular
Ramanujan graph. (This Ramanujan graph is explicitly constructed in
\cite{Ballantine2018ExplicitCayleyRamanujan}.)
\end{enumerate}
\end{thm}

A conjecture formulated jointly with Brooke Feigon, regarding the
level of representations contained in a local A-packet of $U_{3}$,
allows for an extension of the theorem:
\begin{thm}
\label{thm:main-2} In the notation of Theorem \ref{thm:main}, Conjecture
\ref{conj:level-A-packets} implies:
\begin{enumerate}
\item The set $S_{p}$ is a golden gate set for $PU(3)$.
\item If $p\equiv1\Mod{4}$ then the Cayley graph $X^{p,q}:=Cay(G_{p,q},S_{p,q})$
is the one-skeleton of a two-dimensional Ramanujan complex.
\item If $p\equiv3\Mod{4}$ then $X^{p,q}=Cay(G_{p,q},S_{p,q})$ is the
graph induced by non-backtracking 2-walks on the left side of a $(p^{3}+1,p+1)$-biregular
Ramanujan graph.
\end{enumerate}
\end{thm}

Both Theorems are proved in §\ref{sec:Golden-gates-and}. The group
$G_{p,q}$ is explicitly determined in Table \ref{tab:Gpq}, and the
size of $S_{p}$ in Prop.\ \ref{prop:ST-action}. Some explicit examples
are demonstrated in Example \ref{exa:p5_q3} and Figure \ref{fig:PGU3_q3_p5}.
The first ingredient in the proof of these results is the following
theorem, whose proof takes up Section \ref{sec:Simply-transitive-lattices}:
\begin{thm}
\label{thm:simp-tran}For any odd prime $p$, the subgroup $\Lambda_{p}$
generated by the image of $S_{p}$ in $PGU_{3}(\mathbb{Z}\left[1/p\right])$
acts simply-transitively on the hyperspecial vertices of the Bruhat-Tits
building associated with $PGU_{3}(\mathbb{Q}_{p})$.
\end{thm}

Theorem \ref{thm:simp-tran} allows us to identify certain quotients
of the building with explicit Schreier and Cayley graphs and complexes
(see §\ref{subsec:Subgroups-and-quotients}), and certain gates in
$PU(3)$ as acting on such buildings. The spectral analysis of these
graphs, complexes and gates translates in turn into automorphic representation
theory which is carried out in Section \ref{sec:Representation-theory}.
In section §\ref{subsec:Explicit-spectral-computations} we use the
representation-theoretic interpretation to compute explicitly the
spectra of operators on our complexes and gates. To establish that
they form Ramanujan complexes and golden gates, one has to show that
the automorphic representations which appear in their analysis are
either one-dimensional or tempered. For lattices in $PU(3)$ coming
from division algebras, such a result was established in \cite{Rogawski1990Automorphicrepresentationsunitary,harris2001geometry}
and was used in \cite{ballantine2000ramanujan,Clozel2002Automorphicformsand,ballantine2015explicit},
but it is actually false for general arithmetic lattices in $PU(3)$.
This is the failure of the \emph{Naive Ramanujan Conjecture}, observed
in \cite{Howe1979} (see also the survey \cite{Sarnak2005NotesgeneralizedRamanujan}).
Since our lattices do not come from a division algebra, we need to
establish a new Ramanujan-type theorem in these settings. For this
we build upon the work of Rogawski \cite{Rogawski1990Automorphicrepresentationsunitary}
and of \cite{Ngo2010Lelemmefondamental,Shin2011Galoisrepresentationsarising}
(which themselves rely on the work of many, most notably Harris-Taylor,
Clozel, Kottwitz and ultimately Deligne). Let us point out that this
should not to be taken for granted -- in the forthcoming \cite{Ballantine2018ExplicitCayleyRamanujan}
we construct another family of simply-transitive lattices in $PU(3)$,
whose quotient complexes and graphs are not Ramanujan!

In §\ref{sec:Ramanujan-type-theorems} we first prove the more general
Theorem \ref{thm:ram-gen}, which still does not apply to our lattice
$\Lambda_{p}$, and then (in §\ref{subsec:Ramanujan-Lambdap}) the
more refined Theorem \ref{thm:ram-K2} which is specially tailored
for our purposes.
\begin{thm}
\label{thm:ram-gen}Let $G=PGU_{3}\left(E/\mathbb{Q}\right)$ be a
unitary algebraic group, compact at $\infty$, and let $\pi$ be an
automorphic representation of $G$ which is Iwahori-spherical at some
prime which ramifies in $E$. Then either $\pi$ is one-dimensional,
or every unramified local factor of $\pi$ is tempered (regarding
the ramified local factors, see Remark \ref{rem:ramified-tempered}).
\end{thm}

We remark that we call the local factor $\pi_{p}$ of $\pi=\otimes_{\nu}\pi_{\nu}$
\emph{unramified} if $p$ is unramified in $E$ and $\pi_{p}$ is
a spherical representation - see §\ref{sec:Ramanujan-type-theorems}
for more details.
\begin{thm}
\label{thm:ram-K2}In the algebraic group $G=PGU_{3}(\mathbb{Q}[i]/\mathbb{Q})$,
the set 
\begin{equation}
\boldsymbol{K}:=\Big\{ g\in G(\hat{\mathbb{Z}})\,\Big|\,g\equiv\left(\begin{smallmatrix}1 & * & *\\
* & 1 & *\\
* & * & 1
\end{smallmatrix}\right)\Mod{2+2i}\Big\}\label{eq:BigK}
\end{equation}
 forms a compact open subgroup of $G(\hat{\mathbb{Z}})\cong\prod\nolimits _{\ell}G(\mathbb{Z}_{\ell})$
satisfying:
\begin{enumerate}
\item For any odd prime $p$, $G(\mathbb{Q})\cap\boldsymbol{K}\negmedspace\cdot\!G\left(\mathbb{R}\mathbb{Q}_{p}\right)=\Lambda_{p}$
(as defined in Theorem \ref{thm:simp-tran}).
\item \label{enu:ram-K2-sch}For an odd prime $q$ denote by $\mathcal{I}_{q}$
a Iwahori subgroup of $G(\mathbb{Z}_{q})$, and let $\boldsymbol{K}\{q\}=\boldsymbol{K}\cap(\mathcal{I}_{q}\prod_{\ell\neq q}G(\mathbb{Z}_{\ell}))$.
If $\pi$ is an automorphic representation of $G$ with trivial $\infty$-factor
and a $\boldsymbol{K}\{q\}$-fixed vector, then either $\pi$ is
one-dimensional, or every unramified local factor of $\pi$ is tempered.
\item \label{enu:ram-K2-cayley}Let $N\in\mathbb{N}$ be odd and $\boldsymbol{K}(N)=\left\{ g\in\boldsymbol{K}\,\middle|\,g\equiv I\Mod{N}\right\} $.
If $\pi$ is an automorphic representation of $G$ with a $\boldsymbol{K}(N)$-fixed
vector, then Conjecture \ref{conj:level-A-packets} implies that either
$\pi$ is one-dimensional, or every unramified local factor of $\pi$
is tempered.\footnote{For $N=3,5$ this holds unconditionally, hence so does Theorem \ref{thm:main-2}
with $q=3,5$ (see Prop.\ \ref{prop:cases35})}
\end{enumerate}
\end{thm}

Theorem \ref{thm:ram-K2}, which is proved in §\ref{subsec:Ramanujan-Lambdap},
allows us to deduce in §\ref{sec:Golden-gates-and} that the graphs
and complexes which arise from $S_{p}$ above are indeed Ramanujan,
and that the spectrum of $S_{p}$ acting on $PU(3)$ is contained
in the corresponding Ramanujan spectrum (see Cor.\ \ref{cor:rep-to-Ram-gates}
and Prop.\ \ref{prop:Spl-spec}).

Finally, in applications to quantum computations, it is essential
to use gates of finite order. This leads to the notion of \emph{super
golden gates}, and such gates for $PU(2)$ were constructed in \cite{Parzanchevski2018SuperGoldenGates},
using lattices which act simply-transitively on the \emph{edges }of
Bruhat-Tits trees. In §\ref{subsec:super-golden-gates} we construct
such a set for $PU(3)$, conditional on the validity of Theorem \ref{thm:ram-gen}
at the ramified local factors (see Remark \ref{rem:ramified-tempered}).

Let us give an outline of the construction for the case of $X^{p,q}$
with $p,q\equiv1\pmod4$: In §\ref{sec:Simply-transitive-lattices}
we show that the elements in $S_{p}$ take a certain vertex $v_{0}$
in the building $\mathcal{B}$ of $PGL_{3}(\mathbb{Q}_{p})$ once
to each of its neighbors, and that $\left\langle S_{p}\right\rangle $
acts simply-transitively on $\mathcal{B}$.\footnote{A warning is in order: for trees, a symmetric set which takes a vertex
once to each neighbor always generates a group which acts simply-transitively.
In higher dimensions this is far from true!} Denoting $\Lambda_{p}=\left\langle S_{p}\right\rangle $ (which one
can think of as a subgroup of $PGL_{3}(\mathbb{Q}[i])$, for simplicity),
this allows us to identify $\mathcal{B}$ with the Cayley graph $\mathrm{Cay}(\Lambda_{p},S_{p})$.
We also find a nonsymmetric subset $S_{p}'\subseteq S_{p}$ (see (\ref{eq:Spprime}))
such that $S_{p}=S_{p}'\sqcup S_{p}'^{-1}$, and which takes $v_{0}$
to all neighbors of color one: this gives an identification of $\mathcal{B}$
with its color structure with $\mathrm{Cay}(\Lambda_{p},S_{p}')$.
We now take $\Lambda_{p}\left(q\right)=\{g\in\Lambda_{p}\,|\,g\equiv I\pmod q\}=\ker\left(\Lambda_{p}\rightarrow PGL_{3}(\mathbb{F}_{q}[i])\right)$,
and observe that the Cayley graph $\mathrm{Cay}(\nicefrac{\Lambda_{p}}{\Lambda_{p}(q)},S_{p})$
is the same as the quotient of $\mathcal{B}$ by $\Lambda_{p}(q)$.
Furthermore, taking $i\in\mathbb{Q}[i]$ to $\sqrt{-1}\in\mathbb{F}_{q}$
yields a map $\Lambda_{p}\rightarrow PGL_{3}(\mathbb{F}_{q})$ whose
kernel is $\Lambda_{p}(q)$, and it turns out that $\nicefrac{\Lambda_{p}}{\Lambda_{p}(q)}$
is isomorphic to either $PGL_{3}(\mathbb{F}_{q})$ or $PSL_{3}(\mathbb{F}_{q})$
(see Table \ref{tab:Gpq}). On the other hand, we study $\Lambda_{p}(q)\backslash\mathcal{B}$
using representation theory: in §\ref{sec:Representation-theory}
we show that the adjacency operators induced by $S_{p}$ and $S_{p}'$
correspond to certain elements in the group algebra of $PGL_{3}(\mathbb{Q}_{p})$,
acting on the $PGL_{3}(\mathbb{Z}_{p})$-fixed vectors in the unitary
representation $L^{2}(\Lambda_{p}(q)\backslash PGL_{3}(\mathbb{Q}_{p}))$;
this allows us to compute their spectra, assuming the Ramanujan property.
In §\ref{subsec:Automorphic-representations} we find yet another
perspective, identifying $\Lambda_{p}(q)\backslash\mathcal{B}$ with
a certain automorphic space (a subspace of the adelic representation
$L^{2}(PGU_{3}(\mathbb{Q})\backslash PGU_{3}(\mathbb{A}))$), where
the adjacency operators act on the local $p$-component of each irreducible
subrepresentation. This prepares the ground for §\ref{sec:Ramanujan-type-theorems}
and §\ref{subsec:Ramanujan-Lambdap}, where we show that these $p$-components
are tempered, obtaining that $\Lambda_{p}(q)\backslash\mathcal{B}$
is indeed a Ramanujan complex. A roadmap for the analysis (for $p,q\equiv1\left(4\right)$)
is:
\begin{align*}
X^{p,q} & =Cay(P\!\raisebox{1pt}{\scalebox{0.9}{\text{\ensuremath{\begin{smallmatrix}G\\
 S 
\end{smallmatrix}}}}}\!L_{3}(\mathbb{F}_{q}),S_{p})\cong Cay(\Lambda_{p}(q)\backslash\Lambda_{p},S_{p})\cong\Lambda_{p}(q)\backslash Cay(\Lambda_{p},S_{p})\cong\Lambda_{p}(q)\backslash\mathcal{B}\\
 & \cong\Lambda_{p}(q)\backslash PGL_{3}(\mathbb{Q}_{p})/PGL_{3}(\mathbb{Z}_{p})\cong PGU_{3}(\mathbb{Q})\backslash P\!\raisebox{1pt}{\scalebox{0.9}{\text{\ensuremath{\begin{smallmatrix}G\\
 S 
\end{smallmatrix}}}}}\!U_{3}(\mathbb{A})/\boldsymbol{K}^{p}(q)PGU_{3}(\mathbb{R}),
\end{align*}
where $\boldsymbol{K}^{p}(q)=\left\{ g\in PGU_{3}(\hat{\mathbb{Z}})\big[\tfrac{1}{p}\big]\,\middle|\,g\equiv\left(\begin{smallmatrix}1 & * & *\\
* & 1 & *\\
* & * & 1
\end{smallmatrix}\right)\pmod{2+2i},\ g\equiv I\pmod q\right\} $.
\begin{acknowledgement*}
The authors are grateful to Peter Sarnak and Alex Lubotzky for their
guidance, and thank Frank Calegari, Ana Caraiani, Pierre Deligne,
Brooke Feigon, Yuval Flicker, Simon Marshall\textcolor{red}{{} }and
the anonymous referees for helpful comments and discussions. S.E.\ was
supported by ERC grant 692854 and NSF grant DMS-1638352, and O.P.\ by
ISF grant 2990/21.
\end{acknowledgement*}

\section{Preliminaries}

In this section we present the definitions of Ramanujan graphs (not
necessarily regular ones), Ramanujan complexes, and golden gate sets
in $PU(d)$.

\subsection{\label{subsec:Ramanujan-graphs-complexes}Ramanujan graphs and complexes}

For a graph $X=(V,E)$, denote by $A_{X}\colon L^{2}(V)\rightarrow L^{2}(V)$
its adjacency operator, and by $\Spec(X)$ its spectrum. If $X$ is
finite and connected with Perron-Frobenius eigenvalue $\mathfrak{pf}\left(X\right)$,
we call \emph{$\pm\mathfrak{pf}\left(X\right)$ }the \emph{trivial
eigenvalues} of $A_{X}$, and denote by $\Spec_{0}\left(X\right)$
the non-trivial spectrum. The \emph{Alon-Boppana Theorem }asserts,
roughly, that when trying to minimize the spectrum of a graph, one
cannot do better than its universal cover:
\begin{thm}[Alon-Boppana, cf. \cite{greenberg1995spectrum,grigorchuk1999asymptotic,cioabua2006eigenvalues}]
If $\{X_{n}\}$ is a family of quotient graphs of a common tree $\widetilde{X}$,
and $\left\{ \mathrm{girth}\left(X_{n}\right)\right\} $ is unbounded,
then $\Spec(\widetilde{X})\subseteq\overline{\bigcup_{n}\Spec_{0}(X_{n})}$.
\end{thm}

This motivated various authors to define:
\begin{defn}
\label{def:Ram-graph}A graph $X$ with universal cover $\widetilde{X}$
is \emph{Ramanujan }if $\Spec_{0}(X)\subseteq\Spec(\tilde{X})$.
\end{defn}

A $k$-regular graph is Ramanujan if every $\lambda\in\Spec_{0}(X)$
satisfies $\left|\lambda\right|\leq2\sqrt{k-1}$ (see \cite{LPS88}).
We are also interested in the case of \emph{bigraphs}, namely biregular
bipartite graphs. Let $X=(V_{L}\sqcup V_{R},E)$ be a bipartite $(K,k)$-biregular
graph with $K\geq k$ (and $\left|V_{R}\right|=\frac{K}{k}\cdot|V_{L}|$).
The Perron-Frobenius eigenvalue of $X$ is $\mathfrak{pf}\left(X\right)=\sqrt{Kk}$.
Writing $K'=K-1$ and $k'=k-1$, the universal cover of $X$ is the
infinite $(K,k)$-biregular tree $T_{K,k}$, whose $L^{2}$-spectrum
is shown in \cite{Godsil1988Walkgeneratingfunctions} to be
\begin{equation}
\Spec(T_{K,k})=\left[-\sqrt{K'}-\sqrt{k'},-\sqrt{K'}+\sqrt{k'}\right]\cup\big\{0\big\}\cup\left[\sqrt{K'}-\sqrt{k'},\sqrt{K'}+\sqrt{k'}\right].\label{eq:spec-bireg-tree}
\end{equation}
This implies that $X$ is a Ramanujan graph if and only if every eigenvalue
$\lambda$ of $A_{X}$ satisfies 
\begin{equation}
\lambda^{2}=Kk,\quad\lambda=0,\quad\mbox{ or }\quad|\lambda^{2}-K-k+2|\leq2\sqrt{\left(K-1\right)\left(k-1\right)}.\label{eq:ram-spec-bireg}
\end{equation}

\begin{rem}
Let us remark that in the recent breakthrough paper \cite{marcus2013interlacing}
the definition of Ramanujan bigraphs is weaker, requiring only $\Spec_{0}\left(X\right)\subseteq[-\sqrt{K'}-\sqrt{k'},\sqrt{K'}+\sqrt{k'}]$,
and allowing eigenvalues between $0$ and $\sqrt{K'}-\sqrt{k'}$.
The methods of \cite{marcus2013interlacing} do not seem to give Ramanujan
bigraphs in the stronger sense of (\ref{eq:ram-spec-bireg}). The
stronger definition is ``justified'' by the Alon-Boppana theorem,
the behavior of random bigraphs \cite{brito2018spectral}, and the
arithmetic constructions in \cite{ballantine2015explicit} and this
paper.
\end{rem}

We move on to the notion of a \emph{Ramanujan complex}. Let $F$ be
a nonarchimedean local field and $G=G\left(F\right)$ an almost-simple
algebraic group over $F$. In this paper we are mostly interested
in the groups $PGL_{3}(\mathbb{Q}_{p})$ and 
\[
PGU_{3}(\mathbb{Q}_{p})=\left\{ g\in GL_{3}\left(\mathbb{Q}_{p}[\text{\scalebox{0.8}{\ensuremath{\sqrt{-1}}}}]\right)\,\middle|\,gg^{*}=\lambda I,\;\lambda\in\mathbb{Q}_{p}^{\times}\right\} /\mathbb{Q}_{p}[\text{\scalebox{0.8}{\ensuremath{\sqrt{-1}}}}]^{\times},
\]
and the two are isomorphic when $p\equiv1\Mod{4}$ (see §\ref{subsec:The-p-arithmetic-groups}).
The \emph{Bruhat-Tits building} $\mathcal{B}=\mathcal{B}\left(G\right)$
associated with $G$ is a pure contractible simplicial $G$-complex,
whose dimension $d$ equals the $F$-rank of $G$ \cite{Tits1979Reductivegroupsover},
and $G$ acts transitively on its $d$-dimensional cells. The point-wise
stabilizer of such a cell is called an \emph{Iwahori subgroup} (see
Lemma \ref{lem:Iwahori-Kottwitz}).

An adjacency operator $T$ on (all, or some) cells of $\mathcal{B}$
is called \emph{geometric} if it commutes with the action of $G$.
If $\Gamma$ is a lattice in $G$, and $X=\Gamma\backslash\mathcal{B}$
is the corresponding quotient complex, any geometric operator $T$
induces an adjacency operator on $X$, which we denote by $T|_{X}$.
Denote by $G'=G'\left(F\right)$ the simple group associated with
$G$ (in the examples above, $PSL_{3}(\mathbb{Q}_{p})$ and $PSU_{3}(\mathbb{Q}_{p})$,
respectively). An eigenvalue $\lambda$ of a geometric operator $T|_{X}$
is called \emph{trivial} if the lift of its eigenfunction to $\mathcal{B}$
is $G'$-invariant.
\begin{defn}
\label{def:Ram-complex}The complex $X=\Gamma\backslash\mathcal{B}$
is a \emph{Ramanujan complex }if for any geometric operator $T$ on
$\mathcal{B}$ the nontrivial spectrum of $T|_{X}$ is contained in
the spectrum of $T|_{L^{2}\left(\mathcal{B}\right)}$.
\end{defn}

\begin{rem}
The original definition of Ramanujan complexes given in the papers
\cite{li2004ramanujan,Lubotzky2005a} is weaker -- it considers only
geometric operators on vertices. We refer to \cite{kang2016riemann,first2016ramanujan,Lubetzky2017RandomWalks,Ballantine2018ExplicitCayleyRamanujan}
for discussions regarding this point. Another subtle point to be aware
of is that a graph is also a one-dimensional simplicial complex, so
we get a new definition for a Ramanujan graph. For a regular graph,
Definitions \ref{def:Ram-graph} and \ref{def:Ram-complex} coincide,
but not so in the biregular case: in \cite{Ballantine2018ExplicitCayleyRamanujan}
we show that Definition \ref{def:Ram-complex} is equivalent to (\ref{eq:ram-spec-bireg})
together with $A_{X}\big|_{V_{L}}$ being injective.
\end{rem}

In the case of the group $G=PGU_{3}(\mathbb{Q}_{p})$, the building
$\mathcal{B}=\mathcal{B}\left(G\right)$ is
\[
\mathcal{B}\left(PGU_{3}(\mathbb{Q}_{p})\right)=\begin{cases}
\text{2-dimensional building of \ensuremath{PGL_{3}(\mathbb{Q}_{p})}} & p\equiv1\Mod{4}\\
\text{\ensuremath{\left(p^{3}+1,p+1\right)}-biregular tree} & p\equiv3\Mod{4}\\
\text{\ensuremath{3}-regular tree} & p=2
\end{cases}
\]
(a more detailed description appears in §\ref{subsec:The-p-arithmetic-groups}).
The building $\mathcal{B}$ of $PGL_{3}(\mathbb{Q}_{p})$, its combinatorics
and various properties of its Ramanujan quotients were studied in
\cite{kang2010zeta,kang2014zeta,Golubev2013triangle}. Its vertices
$\mathcal{B}^{0}$ correspond to the cosets of $K:=PGL_{3}(\mathbb{Z}_{p})$
in $PGL_{3}(\mathbb{Q}_{p})$; there is an edge between $gK$ and
$g'K$ iff $p^{m+1}g\mathbb{Z}_{p}^{3}<g'\mathbb{Z}_{p}^{3}<p^{m}g\mathbb{Z}_{p}^{3}$
for some $m\in\mathbb{Z}$, and $\mathcal{B}$ is the clique complex
obtained from these edges. The vertices $\mathcal{B}^{0}$ are ``colored''
by 
\begin{equation}
\col\left(gK\right):=\ord_{p}\det g\in\mathbb{Z}/3\mathbb{Z},\label{eq:color-def}
\end{equation}
so that every triangle in $\mathcal{B}$ contains one vertex of each
color. A quotient complex $X=\Gamma\backslash\mathcal{B}$ is called
\emph{tri-partite }if $\col$ is well defined on it, which is equivalent
to $\ord_{p}\det\Gamma\subseteq3\mathbb{Z}$. The directed edges of
$\mathcal{B}$ are colored by $\left\{ 1,2\right\} $, according to
$\col\left(u\rightarrow v\right):=\col v-\col u\Mod{3}$, yielding
two ``colored adjacency'' operators, $A_{1}$ and $A_{2}$, on vertices:
\begin{equation}
\left(A_{i}f\right)(v)=\sum_{u\sim v,\ \col\left(u\rightarrow v\right)=i}f(u).\label{eq:Hecke_Ai_def}
\end{equation}
These operators are geometric, as unlike vertex colors, the edge colors
are $G$-invariant. Every vertex has $p^{2}+p+1$ outgoing edges of
each color, and every edge is contained in $p+1$ triangles.
\begin{thm}[\cite{kang2010zeta}]
\label{thm:old-ram-def}A quotient $X$ of $\mathcal{B}=\B(PGL_{3}(\mathbb{Q}_{p}))$
is Ramanujan (as in Def.\ \ref{def:Ram-complex}) iff the nontrivial
spectrum of $A_{1}|_{X}$ is contained in the spectrum of $A_{1}|_{L^{2}(\mathcal{B})}$.
\end{thm}

This means that $X$ is Ramanujan iff every eigenvalue $\lambda$
of $A_{1}|_{X}$ satisfies either 
\begin{align*}
\left|\lambda\right| & =p^{2}+p+1\ \text{(trivial),}\qquad\text{or}\\
\lambda & \in\Spec\left(A_{1}|_{L^{2}\left(\mathcal{B}^{0}\right)}\right)=\left\{ p\left(\alpha+\beta+\overline{\alpha\beta}\right)\,\middle|\,\alpha,\beta\in\mathbb{C},\ |\alpha|=|\beta|=1\right\} .
\end{align*}
In particular, in this case every nontrivial $A_{1}$-eigenvalue satisfies
$\left|\lambda\right|\leq3p$. Figure \ref{fig:PGU3_q3_p5} shows
the $A_{1}$-spectrum of the buildings of $PGL_{3}(\mathbb{Q}_{5})$
and $PGL_{3}(\mathbb{Q}_{13})$, and the nontrivial spectrum of their
Ramanujan quotients $X^{5,3}$ and $X^{13,5}$ from Theorem \ref{thm:main-2}.
These are the Cayley complexes\footnote{By a Cayley complex we mean a clique complex of a Cayley graph.}
of $G_{5,3}=PSU_{3}(\mathbb{F}_{3})$ with the 31 generators listed
in Example \ref{exa:p5_q3}, and of $G_{13,5}=PSL_{3}(\mathbb{F}_{5})$
with $183$ generators.

\begin{figure}[h]
\begin{centering}
\begin{minipage}[c][1\totalheight][t]{0.49\columnwidth}%
\begin{center}
\hfill{}\includegraphics[scale=0.65]{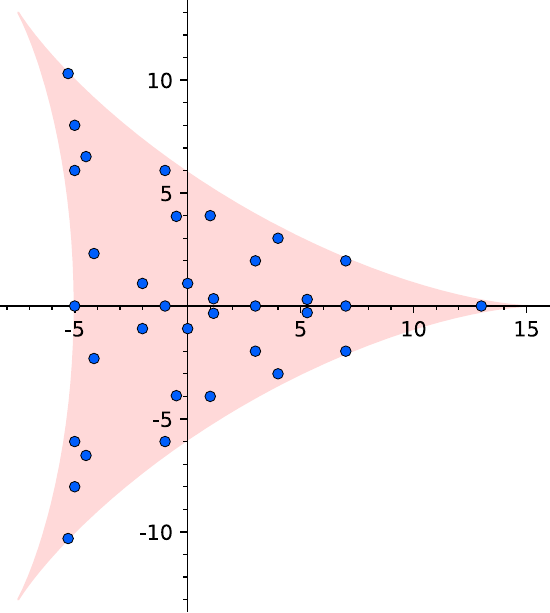}\hfill{}
\par\end{center}%
\end{minipage}%
\begin{minipage}[c][1\totalheight][t]{0.49\columnwidth}%
\begin{center}
\hfill{}\includegraphics[scale=0.65]{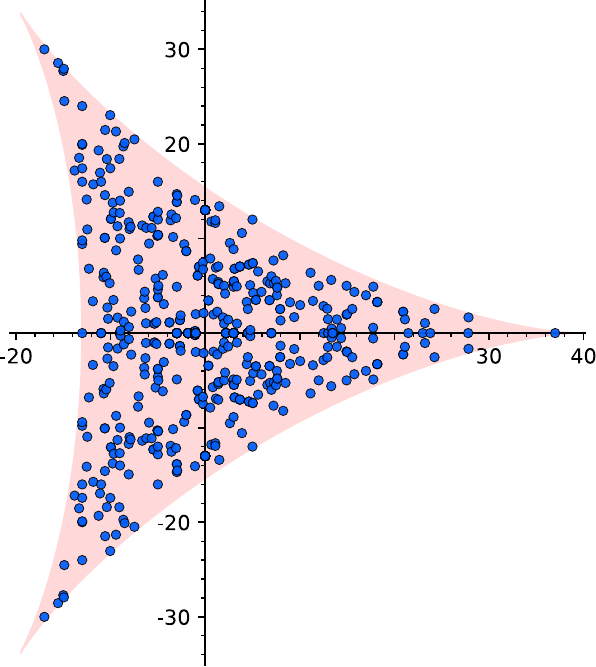}\hfill{}
\par\end{center}%
\end{minipage}
\par\end{centering}
\caption{\label{fig:PGU3_q3_p5}The nontrivial $A_{1}$-spectrum of the Cayley
complexes $X^{5,3}=Cay(PSU_{3}(\mathbb{F}_{3}),S_{5})$ (left) and
$X^{13,5}=Cay(PSL_{3}(\mathbb{F}_{5}),S_{13})$ (right). Shaded: the
$L^{2}$-spectrum of $A_{1}$ on the buildings of $PGL_{3}\left(\mathbb{Q}_{5}\right)$
and $PGL_{3}\left(\mathbb{Q}_{13}\right)$.}
\end{figure}

\medskip{}

The geometric operators which concern us in this paper are the following:
\begin{defn}
\label{def:Al_A1l_def}For two vertices $v,w$ in $\mathcal{B}=\mathcal{B}\left(PGU_{3}(\mathbb{Q}_{p})\right)$,
denote by $\dist\left(v,w\right)$ the graph-distance between $v$
and $w$ in the one-skeleton of $\mathcal{B}$. For $p\equiv1\Mod{4}$,
and $i=1,2$, denote by $\dist_{i}\left(v,w\right)$ the length of
the shortest \emph{$i$-path} (path of constant color $i$) from $v$
to $w$. For $f\in L^{2}\left(\mathcal{B}^{0}\right)$, we define
the following spherical sum operators:
\[
\left(A^{\left(\ell\right)}f\right)\left(v\right)=\sum\nolimits _{\dist\left(v,w\right)=\left\{ \negmedspace\begin{smallmatrix}\ell & p\equiv1\left(4\right)\\
2\ell & p\equiv3\left(4\right)
\end{smallmatrix}\right.}f\left(w\right),\quad\left(A_{1}^{\left(\ell\right)}f\right)\left(v\right)=\sum_{\dist_{1}\left(v,w\right)=\ell}f\left(w\right).
\]
Note that when $p\equiv1\left(4\right)$ we have $A_{1}=A_{1}^{(1)}$
and $A_{1}+A_{2}=A^{(1)}$, and when $p\equiv3\left(4\right)$, $A^{(1)}$
sums over all non-backtracing 2-walks (walks of length 2). The degree
and Ramanujan spectrum of these operators are computed in Prop.\ \ref{prop:Spl-spec}.
\end{defn}

\subsection{\label{subsec:Golden-Gates}Golden Gates}

In \cite{sarnak2015letter,Parzanchevski2018SuperGoldenGates} the
notions of golden gate and super golden gate sets were defined for
$PU(2)=PU_{2}(\mathbb{R})$. These correspond to quantum computation
on a single qubit, and in this section we generalize them to $PU(d)$
for $d\geq2$. The case of $PU(3)$, which is addressed in this paper,
corresponds to computation on a single \emph{qutrit}.

Recall that the compact Lie group $G(\mathbb{R})=PU(d)$ possesses
a Haar measure $\mu=\mu_{G(\mathbb{R})}$ (normalized by $\mu\left(G(\mathbb{R})\right)=1$),
and the bi-invariant metric $d\left(x,y\right)=\sqrt{1-\tr(x^{*}y)/d}$.
For $\varepsilon>0$ and $x\in G(\mathbb{R})$, denote by $B(x,\varepsilon)$
the ball around $x$ of \emph{volume} $\varepsilon$. For a finite
set $X\subset G(\mathbb{R})$, denote $B(X,\varepsilon)=\cup_{x\in X}B(x,\varepsilon)$,
so that $G(\mathbb{R})=B(X,\varepsilon)$ implies $\varepsilon\geq\smash{\tfrac{1}{|X|}}$
by volume considerations. Finally, for $S\subseteq G(\mathbb{R})$
and $\ell\in\mathbb{N}$, let $S^{\left(\ell\right)}$ denote the
elements of $G(\mathbb{R})$ formed by words of length $\ell$ in
$S$, and no less; That is, 
\begin{equation}
S^{\left(\ell\right)}:=S^{\ell}\,\backslash\cup_{t<\ell}S^{t}\qquad(\text{where}\;\;S^{t}=\left\{ s_{1}\cdot\ldots\cdot s_{t}\,\middle|\,s_{i}\in S\right\} ).\label{eq:S(l)}
\end{equation}
Denote by $\left\langle S\right\rangle _{sg}=\bigsqcup_{0\leq\ell}S^{\left(\ell\right)}$
the semigroup generated by $S$.
\begin{defn}
\label{def:golden-gate}A finite subset $S\subset G(\mathbb{R})$
is said to be a \emph{golden gate set }if it satisfies the following
conditions:
\begin{enumerate}
\item Almost-optimal almost-covering (a.o.a.c.): the covering rate of $S^{\left(\ell\right)}$
in $G(\mathbb{R})$ is the same as that of optimally selected $|S^{\left(\ell\right)}|$
points, up to a polylogarithmic factor. Namely, there is a polynomial
$p(x)$, such that
\[
\mu\left(G(\mathbb{R})\setminus B\left(S^{(\ell)},\varepsilon_{\ell}\right)\right)\overset{{\scriptscriptstyle \ell\rightarrow\infty}}{\longrightarrow}0,\qquad\text{when }\varepsilon_{\ell}:=\frac{p\left(\log|S^{(\ell)}|\right)}{|S^{(\ell)}|}.
\]
\item Navigation: The word problem in $\left\langle S\right\rangle _{sg}$
has an efficient solution. Namely, an element in $\left\langle S\right\rangle _{sg}$,
given as a matrix, can be written explicitly as a word in $S$ of
the shortest length possible, in polynomial time. In terms of quantum
computation, this translates to efficient compiling of a circuit as
a product of fundamental gates.
\item Approximation: There exists $N$ (depending only on $G$) such that
any element of $G(\mathbb{R})$ can be decomposed as a product of
at most $N$ \emph{approximable }ones. We say $g\in G(\mathbb{R})$
is approximable (with respect to $S$) if given $\ell$ and $\varepsilon$,
there is an efficient algorithm which finds an element in $B\left(g,\varepsilon\right)\cap S^{\left(\ell\right)}$
if one exists.
\item Growth: The size of $S^{\left(\ell\right)}$ (which is the number
of distinct circuits of length $\ell$) is exponential in $\ell$.
\end{enumerate}
If $C$ is a finite subgroup of $G(\mathbb{R})$, we say that $S$
is a golden gate set on $C\backslash G(\mathbb{R})$ if every element
in $G(\mathbb{R})$ is well approximated by $S$ up to an element
of $C$. Namely, (1) holds with $B(S^{(\ell)},\varepsilon_{\ell})$
replaced by $C\cdot B\left(S^{(\ell)},\varepsilon_{\ell}\right)$.
Finally, we say that a golden gate set $S$ is a \emph{super golden
gate }set if every element in $S$ is of finite order.
\end{defn}

The Ramanujan property (Theorems \ref{thm:ram-gen} and \ref{thm:ram-K2})
is a crucial ingredient in establishing a.o.a.c. To see where each
property is established for our gates, consult the proof of theorem
\ref{thm:main-2} in §\ref{sec:Golden-gates-and}.
\begin{rem}
\label{rem:AOAC-approx}
\begin{enumerate}
\item Almost-covering refers to the possibility that $B(S^{(\ell)},\varepsilon_{\ell})\subsetneq G(\mathbb{R})$,
or in other words, almost all of $G(\mathbb{R})$ is covered by the
balls around $S^{(\ell)}$ but possibly not all of it. However, \cite[Cor.\ 3.2]{Parzanchevski2018SuperGoldenGates}
shows that golden gates satisfy $G(\mathbb{R})=B(S^{(2\ell)},2\varepsilon_{\ell})$,
for large $\ell$, and when $|S^{(\ell)}|$ grows asymptotically as
$N^{\ell}$ for fixed $1<N\leq|S|$, then $\varepsilon_{\ell}\leq\varepsilon_{2\ell}^{1/2+\delta}$,
for any small $\delta>0$ and any large enough $\ell$ (w.r.t.\ $\delta$),
which implies $G(\mathbb{R})=B(S^{(\ell)},2\varepsilon_{\ell}^{1/2+\delta})$.
Note that $\varepsilon_{\ell}$ is the optimal almost-covering rate,
up to the polylogarithmic factor, while $\varepsilon_{\ell}^{1/2}$
is not the optimal covering rate (see \cite{Sarnak2015Appendixto2015}
for an analogue phenomenon in the setting of the finite group $SL_{2}(\mathbb{Z}/q\mathbb{Z})$).
\item \label{enu:approximability}In $PU(2)$, elements of the form 
\[
\left(\begin{array}{cc}
e^{i\alpha}\\
 & e^{-i\alpha}
\end{array}\right),\left(\begin{array}{cc}
\cos\alpha & \sin\alpha\\
-\sin\alpha & \cos\alpha
\end{array}\right),\left(\begin{array}{cc}
\cos\alpha & i\cdot\sin\alpha\\
i\cdot\sin\alpha & \cos\alpha
\end{array}\right)
\]
are approximable with respect to the golden gates of \cite{lubotzky1986hecke,Parzanchevski2018SuperGoldenGates},
by an algorithm of Ross and Selinger \cite{ross2015optimal}. Furthermore,
any element in $PU(2)$ is a product of three such elements. In contrast,
deciding whether $B\left(g,\varepsilon\right)\cap S^{\left(\ell\right)}$
is nonempty for a general $g\in PU(2)$ is closely related to an NP-complete
problem \cite{Parzanchevski2018SuperGoldenGates}. The same algorithm
shows that elements of $PU(3)$ of the form 
\[
\sigma^{i}\left(\begin{smallmatrix}e^{i\alpha}\\
 & e^{-i\alpha}\\
 &  & 1
\end{smallmatrix}\right)\sigma^{-i},\sigma\left(\begin{smallmatrix}\cos\alpha & \sin\alpha\\
-\sin\alpha & \cos\alpha\\
 &  & 1
\end{smallmatrix}\right)\sigma^{-i},\sigma\left(\begin{smallmatrix}\cos\alpha & i\cdot\sin\alpha\\
i\cdot\sin\alpha & \cos\alpha\\
 &  & 1
\end{smallmatrix}\right)\sigma^{-i}
\]
where $\sigma=\left(\begin{smallmatrix} & 1\\
 &  & 1\\
1
\end{smallmatrix}\right)$ and $i=0,1,2$, are approximable. Every element in $PU(3)$ can be
written as a product of nine such elements, using for example the
decomposition in \cite{byrd1998differential}. It is an important
question to determine which other elements in $PU(3)$ are approximable,
as this would lead to lower complexity of quantum circuits.
\end{enumerate}
\end{rem}

\section{\label{sec:Simply-transitive-lattices}Simply-transitive lattices
in $PU(3)$}

In this section we introduce our arithmetic lattices, and show that
they act simply-transitively on the hyperspecial vertices of the corresponding
Bruhat-Tits building, thus proving Theorem \ref{thm:simp-tran}.

\subsection{\label{subsec:The-p-arithmetic-groups}Arithmetic unitary groups}

Let $E=\mathbb{Q}[\text{\scalebox{0.85}{\ensuremath{\sqrt{-d}\,}}}]$
$\left(d\in\mathbb{\mathbb{N}}\right)$ be an imaginary quadratic
extension of $\mathbb{Q}$, and $\mathcal{O}_{E}$\textcolor{red}{{}
}the ring of integers of $E$. Denote by $\mathbf{GU}_{\boldsymbol{3}}^{\boldsymbol{E}}$
the algebraic group scheme of \emph{unitary similitudes }in three
variables with respect to $E/\mathbb{Q}$. Namely, for any commutative
ring $R$, 
\begin{equation}
GU_{3}^{E}\left(R\right)=\left\{ g\in GL_{3}\left(\mathcal{O}_{E}\otimes R\right)\,\middle|\,gg^{*}=\lambda I,\;\lambda\in R^{\times}\right\} ,\label{eq:GU-def}
\end{equation}
where $\left(g^{*}\right)_{ij}=\overline{g_{ji}}$. More generally,
one can replace $E/\mathbb{Q}$ by any CM-extension $E/F$ (namely,
$F$ is a totally real number field and $E$ a completely imaginary
quadratic extension), and also $gg^{*}=\lambda I$ by $gHg^{*}=\lambda H$,
where $H$ is any non-degenerate Hermitian matrix in $M_{3}\left(E\right)$.
The unitary group scheme $\mathbf{U}_{\boldsymbol{3}}^{\boldsymbol{E}}$
is defined as in (\ref{eq:GU-def}) but with $gg^{*}=I$. From this
point on we assume that $F=\mathbb{Q}$ and $H=I$, but the picture
is similar for any CM-extension $E/F$ and any $H$ which is positive-definite
in every complex embedding of $E$.

The group scheme $\mathbf{PGU}_{\boldsymbol{3}}^{\boldsymbol{E}}$
of \emph{projective unitary similitudes }is the quotient of $\mathbf{GU}_{\boldsymbol{3}}^{\boldsymbol{E}}$
by the scalars $\mathbf{E}^{\times}$ (where $\mathbf{E}^{\times}(R)=(\mathcal{O}_{E}\otimes R)^{\times}$).
When the Picard group of $\mathcal{O}_{E}\otimes R$ is trivial (which
is true in the cases we consider), we have
\[
PGU_{3}^{E}\left(R\right)=\nicefrac{{\displaystyle GU_{3}^{E}\left(R\right)}}{{\displaystyle (\mathcal{O}_{E}\otimes R)^{\times}}},
\]
but we remark that $PGU_{3}^{E}\left(R\right)$ can be larger in general.
In this paper we mostly concentrate on the case $E=\mathbb{Q}[i]$,
where we write $R[i]=R\otimes_{\mathbb{Z}}\mathbb{Z}[i]$ and abbreviate
\[
PGU_{3}(R)=PGU_{3}^{\mathbb{Q}[i]}(R)=\left\{ g\in GL_{3}\left(R\left[i\right]\right)\,\middle|\,gg^{*}=\lambda I,\;\lambda\in R^{\times}\right\} /R[i]^{\times}.
\]
Similarly, for any commutative ring $R$ we define
\[
PU_{3}(R)=PU_{3}^{\mathbb{Q}[i]}(R)=\left\{ g\in GL_{3}\left(R\left[i\right]\right)\,\middle|\,gg^{*}=I\right\} /U_{1}(R),
\]
and remark again that the scheme-theoretic interpretation of $\mathbf{PU}_{3}$
may be larger.

Recall that a prime $p\in\mathbb{Z}$ may decompose as $p=\pi\overline{\pi}$
in $\mathcal{O}_{E}$, and that 
\[
p\text{ is called }\begin{cases}
\text{split} & \text{if }p=\pi\overline{\pi}\text{ and \ensuremath{\pi\nsim\overline{\pi}} (not associates)}\\
\text{inert} & \text{if }p\text{ is prime in }\mathcal{O}_{E}\\
\text{ramified} & \text{if }p=\pi\overline{\pi}\text{ and \ensuremath{\pi\sim\overline{\pi}}}.
\end{cases}
\]
For $E=\mathbb{Q}[i]$, $p$ is split, inert, or ramified, according
to whether $p\Mod{4}$ equals $1,3$ or $2$, respectively. For a
valuation $\nu$ of $\mathbb{Q}$, we denote by $G_{\nu}$ the group
$PGU_{3}^{E}(\mathbb{Q}_{\nu})$; we now briefly describe $G_{\nu}$
and its Bruhat-Tits building for nonarchimedean $\nu$, following
\cite[§1.15, §2.10]{Tits1979Reductivegroupsover} (see also \cite{Bruhat1987Schemasengroupes,Kim2001Involutionsclassicalgroups,Prasad2002finitegroupactions,Abramenko2002Latticechainmodels}).
The buildings of $PGU,PU,PSU,U$ and $SU$ (but not $GU$) are all
the same.
\begin{itemize}
\item $G_{\infty}=PGU_{3}^{E}(\mathbb{R})=PU(3)$ is the usual projective
unitary group.
\item If $p$ is split in $E$, there is an embedding $\mathcal{O}_{E}\hookrightarrow\mathbb{Z}_{p}$
which we denote by $\alpha\mapsto\widetilde{\alpha}$. This embedding
extends to an isomorphism of $U_{3}(\mathbb{Q}_{p})$ with $GL_{3}(\mathbb{Q}_{p})$,
which induces an isomorphism $G_{p}=PGU_{3}(\mathbb{Q}_{p})\cong PGL_{3}(\mathbb{Q}_{p})$
(and $PGU_{3}\left(\mathbb{Z}_{p}\right)\cong PGL_{3}\left(\mathbb{Z}_{p}\right)$).
The building of $G_{p}$ is two dimensional, and all of its vertices
are hyperspecial.
\item If $p$ is inert in $E$, then $G_{p}$ is of $\mathbb{Q}_{p}$-rank
one and its building is a $(p^{3}+1,p+1)$-biregular tree; the hyperspecial
vertices are the ones of degree $p^{3}+1$, and they correspond to
the cosets of $PGU_{3}\left(\mathbb{Z}_{p}\right)$ in $G_{p}$.
\item If $p$ is ramified in $E$, then $G_{p}$ is a rank one group and
its building is a $(p+1)$-regular tree with no hyperspecial vertices.
\end{itemize}
\begin{rem}
\label{rem:B-fixed-section}For $p$ which does not split in $E$,
$E_{p}=\mathbb{Q}_{p}\otimes E$ is a field, and the building of $G_{p}$
can be related to that of $\tilde{G}=PGL_{3}(E_{p})$, as follows.
Consider the Galois involution $g^{\#}=(g^{*})^{-1}$, so that $G_{p}=\{g\in\tilde{G}\,|\,g^{\#}=g\}$.\footnote{In the case of a general Hermitian form $H$, one should take $g^{\#}=H(g^{*})^{-1}H^{-1}$
correspondingly.} Denote by $\tilde{\mathcal{B}}$ the (two-dimensional) building of
$\tilde{G}=PGL_{3}(E_{p})$, whose vertices correspond to $\tilde{G}/\tilde{K}$,
where $\tilde{K}=PGL_{3}(\mathcal{O}_{E_{p}})$. The involution $\#$
induces a simplicial symmetry on $\tilde{\mathcal{B}}$ (defined on
vertices by $(g\tilde{K})^{\#}=g^{\#}\tilde{K}$), and we denote by
$\mathcal{B}^{\#}$ its \emph{fixed section}: the vertices of $\mathcal{B}^{\#}$
are either $\#$-fixed $\tilde{\mathcal{B}}$-vertices, or midpoints
of $\tilde{\mathcal{B}}$-edges reflected by $\#$; the edges of $\mathcal{B}^{\#}$
are the medians of $\tilde{\mathcal{B}}$-triangles reflected by $\#$
(by (\ref{eq:color-def}) $\col v^{\#}=-\col v$, so no $\tilde{\mathcal{B}}$-triangle
is fixed point-wise by $\#$). By \cite[§2.6.1]{Tits1979Reductivegroupsover},
if $p$ is odd or inert in $E$ then $\mathcal{B}$ can be identified
with $\mathcal{B}^{\#}$, and if $p=2$ and is ramified in $E$, then
$\mathcal{B}$ is contained in $\mathcal{B}^{\#}$ as a metric space.
Also, when $p$ is inert, the hyperspecial vertices $\mathcal{B}^{sp}$
correspond to the $\#$-fixed $\tilde{\mathcal{B}}$-vertices.
\end{rem}

Fixing a prime $p$, we consider the $p$-arithmetic group 
\[
\Gamma_{p}=PGU_{3}^{E}(\mathbb{Z}[1/p])\leq PGU_{3}^{E}(\mathbb{R})\times PGU_{3}^{E}(\mathbb{Q}_{p})=G_{\infty}\times G_{p}.
\]
By a classical Theorem of Borel and Harish-Chandra, $\Gamma_{p}$
is a cocompact lattice in $G_{\infty}\times G_{p}$ \cite[Thm.\ 5.7(2)]{Platonov1994Algebraicgroupsand},
and since $G_{\infty}$ is compact, $\Gamma_{p}$ is also a cocompact
lattice in $G_{p}$. For any finite index subgroup $\Lambda$ of $\Gamma_{p}$,
the quotient of the Bruhat-Tits building $\B=\B(G_{p})$ by the action
of $\Lambda$ is a finite complex, which we denote by 
\[
X_{\Lambda}=\Lambda\backslash\mathcal{B}.
\]
For $N$ an integer coprime to $p$ we denote $\Lambda(N)=\left\{ g\in\Lambda\,\middle|\,g\equiv I\Mod{N}\right\} $.
We call $\Lambda$ a \emph{congruence subgroup of level $N$} if it
contains $\Gamma_{p}(N)$, the so-called \emph{principal congruence
subgroup of level $N$}.

We finish this section with a useful Lemma concerning $PGU_{3}=PGU_{3}^{\mathbb{Q}[i]}$.
\begin{lem}
\label{lem:PGU=00003DPU}For $R=\mathbb{R},\mathbb{F}_{p},\mathbb{Z},\mathbb{Q}_{p},\mathbb{Z}_{p},\hat{\mathbb{Z}}$
or $\mathbb{A}$ (the adeles of $\mathbb{Q}$), $PGU_{3}\left(R\right)=PU_{3}\left(R\right)$.
\end{lem}

\begin{proof}
Observe the norm group $\mathcal{N}(R)=N_{R[i]/R}(R[i]^{\times})$,
and the group 
\[
\mathcal{S}\left(R\right)=\left\{ \lambda\in R^{\times}\,\middle|\,\exists g\in M_{3}\left(R[i]\right):g^{*}g=\lambda I\right\} \leq R^{\times}
\]
of ``$3$-similtude factors''. The obvious map induces an isomorphism
$\nicefrac{PGU_{3}\left(R\right)}{PU_{3}\left(R\right)}\cong\nicefrac{\mathcal{S}\left(R\right)}{\mathcal{N}\left(R\right)}$,
so we must show that $\mathcal{N}\left(R\right)=\mathcal{S}\left(R\right)$.
 If $R=\mathbb{R}$, then indeed $\mathcal{S}(\mathbb{R})=\mathbb{R}_{>0}=N_{\mathbb{C}/\mathbb{R}}\left(\mathbb{C}^{\times}\right)$.\footnote{In contrast, $\left(\begin{smallmatrix}1 & 2\\
2 & 1
\end{smallmatrix}\right)^{*}\left(\begin{smallmatrix}1\\
 & -1
\end{smallmatrix}\right)\left(\begin{smallmatrix}1 & 2\\
2 & 1
\end{smallmatrix}\right)=\left(\begin{smallmatrix}-3\\
 & 3
\end{smallmatrix}\right)$ implies that $PGU_{2}\left(\mathbb{R},\left(\begin{smallmatrix}1\\
 & -1
\end{smallmatrix}\right)\right)\neq PU_{2}\left(\mathbb{R},\left(\begin{smallmatrix}1\\
 & -1
\end{smallmatrix}\right)\right)$.} For $R=\mathbb{F}_{p},\mathbb{Z}_{p},\mathbb{Q}_{p}$ when $p\equiv1\Mod{4}$,
$\sqrt{-1},2\in R^{\times}$ and $N\big(\frac{1+\alpha}{2}+\frac{1-\alpha}{2\sqrt{-1}}\cdot i\big)=\alpha$
implies $\mathcal{N}\left(R\right)=R^{\times}$ (forcing $\mathcal{N}\left(R\right)=\mathcal{S}\left(R\right)$).
For $p\equiv3\Mod{4}$, $\mathcal{N}(\mathbb{F}_{p})=\mathbb{F}_{p}^{\times}$,
and applying Hensel's lemma over $\mathbb{Z}_{p}[i]$ yields $\mathcal{N}\left(\mathbb{Z}_{p}\right)=\mathbb{Z}_{p}^{\times}$
as well. It follows that $\mathcal{N}(\mathbb{Q}_{p})=p^{2\mathbb{Z}}\mathbb{Z}_{p}^{\times}$,
and indeed for $\lambda\in\mathcal{S}(\mathbb{Q}_{p})$ we see that
$g^{*}g=\lambda I$ implies $2\mathrm{ord}_{p}\det g=3\mathrm{ord}_{p}(\lambda)$,
so that $\ord_{p}(\lambda)$ is even and $\lambda\in\mathcal{N}(\mathbb{Q}_{p})$.
The cases $\mathbb{F}_{2},\mathbb{Z}$ are trivial. If $R=\mathbb{Z}_{2},\mathbb{Q}_{2}$,
then $1,-1$ are representatives for $\mathbb{Q}_{2}^{\times}/\mathcal{N}\left(\mathbb{Q}_{2}\right)$
and also for $\mathbb{Z}_{2}^{\times}/\mathcal{N}\left(\mathbb{Z}_{2}\right)$.
We have $-1\notin\mathcal{S}\left(\mathbb{Z}_{2}\right)$, by a brute
force verification that no $g\in M_{3}\left(\mathbb{Z}[i]\right)$
satisfies $g^{*}g\equiv-I\Mod{2+2i}$. Furthermore, if $g\in GU_{3}\left(\mathbb{Q}_{2}\right)$
and $g^{*}g=-I$ then $g$ lies in the standard maximal compact $PGU_{3}\left(\mathbb{Z}_{2}\right)$
of $PGU_{3}\left(\mathbb{Q}_{2}\right)$, hence can be scaled to $GU_{3}\left(\mathbb{Z}_{2}\right)$,
so that $-1\notin\mathcal{S}\left(\mathbb{Q}_{2}\right)$. Finally,
from $\mathbb{R},\mathbb{Q}_{p},\mathbb{Z}_{p}$ we can deduce the
same for $\hat{\mathbb{Z}},\mathbb{A}$.
\end{proof}

\subsection{\label{subsec:Simply-transitive-action}Simply-transitive action}

In this section we fix an odd prime $p$ and show that the group $\Lambda_{p}$
generated by the image of $S_{p}$ in $PGU_{3}(\mathbb{Z}\left[1/p\right])$
(see (\ref{eq:set-Sp})) acts simply-transitively on $\mathcal{B}^{sp}$,
the hyperspecial vertices of $\mathcal{B}$. These vertices correspond
to the cosets of $K_{p}:=PGU_{3}(\mathbb{Z}_{p})$ in $G_{p}=PGU_{3}(\mathbb{Q}_{p})$,
and we denote the trivial coset by $v_{0}$.

Let us start with a few definitions. We fix a prime factor $\pi\in\mathbb{Z}\left[i\right]$
of $p$, taking $\pi=p$ when $p\equiv3\Mod{4}$, and for $p\equiv1\left(4\right)$
assuming w.l.o.g.\ that $\widetilde{\pi}\in\mathbb{Z}_{p}$ is a
uniformizer (recall that $\widetilde{\smash{a+bi}}=a+b\sqrt{-1}\in\mathbb{Z}_{p}$).
For a matrix $A$ we denote 
\[
\ord_{\alpha}A:=\min_{i,j}\left(\ord_{\alpha}A_{ij}\right),
\]
whenever $\ord_{\alpha}A_{ij}$ makes sense, and we define the \emph{level},
\emph{$\pi$-height, and $p$-height }functions:
\begin{align*}
\ell\left(g\right) & :=\ord_{p}gg^{*}-\ord_{\pi}g-\ord_{\overline{\pi}}g\negthickspace\negthickspace\negthickspace\negthickspace\negthickspace\negthickspace\negthickspace &  & \negthickspace\negthickspace\negthickspace\negthickspace g\in GU_{3}\left(\mathbb{Z}\left[1/p\right]\right)\\
h_{\pi}\left(g\right) & :=\ord_{\pi}\det g-3\ord_{\pi}g &  & \negthickspace\negthickspace\negthickspace\negthickspace g\in GU_{3}\left(\mathbb{Z}\left[1/p\right]\right)\\
h_{p}\left(g\right) & :=\ord_{p}\det g-3\ord_{p}g &  & \negthickspace\negthickspace\negthickspace\negthickspace g\in GL_{3}(\mathbb{Q}_{p}).
\end{align*}
Observe that $\ell,h_{\pi},h_{p}$ are well defined on $PGU_{3}(\mathbb{Z}\left[1/p\right])$
and $PGL_{3}(\mathbb{Q}_{p})$ as well, and that for $p\equiv1\Mod{4}$,
the assumption that $\widetilde{\pi}$ is a uniformizer implies that
the isomorphism $PGU_{3}(\mathbb{Q}_{p})\cong PGL_{3}(\mathbb{Q}_{p})$
satisfies $h_{\pi}\left(g\right)=h_{p}\left(\widetilde{g}\right)$.
\begin{prop}
\label{prop:distances}For $g\in GU_{3}(\mathbb{Z}\left[1/p\right])$,
$\ell\left(g\right)=\dist\left(v_{0},gv_{0}\right)$ and $h_{\pi}\left(g\right)=\dist_{1}\left(v_{0},gv_{0}\right)$.
\end{prop}

\begin{proof}
Let us first show that $x\in PGL_{3}(\mathbb{Q}_{p})$ satisfies $h_{p}\left(x\right)=\dist_{1}\left(v_{0},xv_{0}\right)$.
Observe that if $k,k'\in K_{p}$ then $\ord_{p}\left(kxk'\right)\geq\ord_{p}x$
by the nonarchimedean triangle inequality, and applying the same argument
with $k^{-1},k'^{-1}$ gives an equality. As $\det k\in\mathbb{Z}_{p}^{\times}$,
it follows that $h_{p}$ is $K_{p}$-bi-invariant. Using the Cartan
decomposition $PGL_{3}(\mathbb{Q}_{p})=\bigsqcup\nolimits _{0\leq r\leq s}K_{p}\diag\left(1,p^{r},p^{s}\right)K_{p}$,
we can write $x=kak'$ with $k,k'\in K_{p}$ and $a=\diag\left(1,p^{r},p^{s}\right)$.
Then $h_{p}\left(x\right)=h_{p}\left(a\right)=r+s$, and on the other
hand $\dist_{1}\left(v_{0},xv_{0}\right)=\dist_{1}(k^{-1}v_{0},ak'v_{0})=\dist_{1}\left(v_{0},av_{0}\right)$.
Applying $s$ times $\diag(1,1,p)$ and $r$ times $\diag(1,p,1)$
gives a $1$-path of length $r+s$ from $v_{0}$ to $av_{0}$, and
by determinant considerations there is no shorter one, so that $\dist_{1}\left(v_{0},xv_{0}\right)=r+s=h_{p}\left(x\right)$
as claimed. We now restrict to $p\equiv1\left(4\right)$, and immediately
obtain $h_{\pi}\left(g\right)=h_{p}\left(\widetilde{g}\right)=\dist_{1}\left(v_{0},\widetilde{g}v_{0}\right)=\dist_{1}\left(v_{0},gv_{0}\right)$.
Next, as $\ord_{p}\alpha\overline{\alpha}=\ord_{\pi}\alpha+\ord_{\pi}\overline{\alpha}$
for $\alpha\in\mathbb{Z}\left[i\right]$, we have $\ord_{\pi}\det g+\ord_{\pi}\det g^{*}=\ord_{p}\det gg^{*}=3\left(\ell(g)+\ord_{\pi}g+\ord_{\pi}g^{*}\right)$,
which simplifies to 
\begin{equation}
3\ell\left(g\right)=h_{\pi}\left(g\right)+h_{\pi}(g^{*}).\label{eq:hpi_hpibar_l}
\end{equation}
The Cartan involution $g\mapsto g^{\star}:=\left(g^{-1}\right)^{t}$\textcolor{red}{{}
}of $GL_{3}(\mathbb{Q}_{p})$ induces an isometry of $\mathcal{B}$
which inverts edge color, and fixes $v_{0}$, so that 
\[
h_{\pi}(g^{*})=\dist_{1}(v_{0},g^{*}v_{0})=\dist_{1}\left(gv_{0},v_{0}\right)=\dist_{2}\left(v_{0},gv_{0}\right)=\dist_{1}(v_{0},g^{\star}v_{0})=h_{\pi}(g^{\star}).
\]
Writing again $g\in K_{p}\diag\left(1,p^{r},p^{s}\right)K_{p}$ ($0\leq r\leq s$),
we observe that $\dist\left(v_{0},gv_{0}\right)=\dist\left(v_{0},\diag\left(1,p^{r},p^{s}\right)v_{0}\right)=s$
(by applying $r$ times $\diag(1,p,p)$ and $s-r$ times $\diag(1,1,p)$),
and that $g^{\star}\in K_{p}\diag\left(1,p^{s-r},p^{s}\right)K_{p}$,
implying $h_{\pi}(g^{\star})=2s-r=3\dist\left(v_{0},gv_{0}\right)-h_{\pi}(g)$,
and giving in total $\ell\left(g\right)=\dist\left(v_{0},gv_{0}\right)$.

For $p\equiv3\left(4\right)$, recall that $\mathcal{B}^{sp}$ correspond
to $\#$-fixed vertices in $\tilde{\mathcal{B}}$ (in the notations
of Remark \ref{rem:B-fixed-section}). Hyperspecial vertices of $\mathcal{B}$-distance
two correspond to non-neighboring vertices in a parallelogram in $\tilde{\mathcal{B}}$,
so that in $\mathcal{\widetilde{B}}$ they are of distance two, and
of $1$-distance three. More generally, $\dist^{\mathcal{B}}\left(v,w\right)=\dist^{\mathcal{\widetilde{B}}}\left(v,w\right)=\tfrac{2}{3}\dist_{1}^{\mathcal{\widetilde{B}}}\left(v,w\right)$,
since the involution $\#$ of $\widetilde{\mathcal{B}}$ also fixes
the geodes between two fixed points. Now, we observe that $gg^{*}=p^{\ell(g)+2\ord_{p}g}I$
implies $2\ord_{p}\det g=3\ell\left(g\right)+6\ord_{p}g$, and in
total we have
\[
\dist^{\mathcal{B}}\left(v_{0},gv_{0}\right)=\tfrac{2}{3}\dist_{1}^{\mathcal{\widetilde{B}}}\left(v_{0},gv_{0}\right)\overset{{\scriptscriptstyle (*)}}{=}\tfrac{2}{3}h_{p}\left(g\right)=\tfrac{2}{3}\left(\ord_{p}\det g-3\ord_{p}g\right)=\ell\left(g\right),
\]
where $\overset{{\scriptscriptstyle (*)}}{=}$ is by the same arguments
as for $x\in PGL_{3}(\mathbb{Q}_{p})$, but applied to $PGL_{3}(\mathbb{Q}_{p}[i])$.
\end{proof}
Observe that $GU_{3}\left(\mathbb{Z}\right)$ is a finite group of
order 384, comprised of all monomial matrices with entries in $\mathbb{Z}\left[i\right]^{\times}$,
and for an odd prime $p$ denote
\[
\Omega_{p}:=\left\{ g\in GU_{3}\left(\mathbb{Z}\left[1/p\right]\right)\,\middle|\,g\equiv\left(\begin{smallmatrix}1 & * & *\\
* & 1 & *\\
* & * & 1
\end{smallmatrix}\right)\Mod{2+2i}\right\} .
\]

\begin{prop}
\label{prop:split-GU}$GU_{3}(\mathbb{Z}\left[1/p\right])=\Omega_{p}\rtimes GU_{3}\left(\mathbb{Z}\right)$,
and in particular $\Omega_{p}$ is a group.
\end{prop}

\begin{proof}
We first show that there is a bijection (of sets) 
\begin{equation}
\left(g,k\right)\mapsto gk:\Omega_{p}\times GU_{3}\left(\mathbb{Z}\right)\overset{\cong}{\longrightarrow}GU_{3}\left(\mathbb{Z}\left[1/p\right]\right)\label{eq:gk-bijec}
\end{equation}
Let $g\in GU_{3}(\mathbb{Z}\left[1/p\right])$ such that $gg^{*}\left(=g^{*}g\right)=p^{\ell}I$,
and assume first that $g\in M_{3}\left(\mathbb{Z}\left[i\right]\right)$.
If $p\equiv3\Mod{4}$, then $\ell$ is even by $2\ord_{p}\det g=3\ell$.
Thus, for any odd $p$, every row and every column of $g$ is composed
of three Gaussian integers $\alpha,\beta,\gamma$ with $|\alpha|^{2}+|\beta|^{2}+|\gamma|^{2}\equiv p^{\ell}\equiv1\Mod{4}$.
This implies that there exists a unique $\sigma\in S_{3}$ such that
$|g_{j,\sigma\left(j\right)}|^{2}\equiv1\left(4\right)$ ($j=1,2,3$)
and $|g_{j,j'}|^{2}\equiv0\left(2\right)$ for $j'\neq\sigma\left(j\right)$.
Observe that $\left|\alpha\right|^{2}\equiv1\left(4\right)$ implies
that $\alpha$ is congruent modulo $2+2i$ to exactly one of $1,i,-1,-i$,
hence there exists a unique $u\in\mathbb{Z}\left[i\right]^{\times}$
such that $u\alpha\equiv1\Mod{2+2i}$. Since $GU_{3}\left(\mathbb{Z}\right)$
are precisely the monomial matrices with nonzero entries in $\mathbb{Z}\left[i\right]^{\times}$,
we see that there exists a unique $k\in GU_{3}\left(\mathbb{Z}\right)$
so that $gk\in\Omega_{p}$. For a general $g\in GU_{3}(\mathbb{Z}\left[1/p\right])$,
we have $p^{2t}g\in M_{3}\left(\mathbb{Z}\left[i\right]\right)$ for
large enough $t$, hence $p^{2t}gk\in\Omega_{p}$ for a unique $k\in GU_{3}\left(\mathbb{Z}\right)$.
Since $p^{\pm2}\equiv1\Mod{2+2i}$, it is also the unique $k$ for
which $gk\in\Omega_{p}$, proving (\ref{eq:gk-bijec}).

Next we show that $\Omega_{p}$ is a group: for $g\in\Omega_{p}$
with $gg^{*}=p^{\ell}I$ one has $(g^{-1})_{j,j}=(p^{-\ell}g^{*})_{j,j}\equiv\overline{g_{j,j}}\Mod{2+2i}$
since $p^{\ell}\equiv1\left(4\right)$, so that $g^{-1}\in\Omega_{p}$.
Now let $\alpha,\beta$ be two off-diagonal elements in the same row
or column of $g$. Denoting $\varpi=1+i$, we already know that $|\alpha|^{2}\equiv|\beta|^{2}\equiv0\left(2\right)$
which implies $\alpha\equiv\beta\equiv0\left(\varpi\right)$. In addition,
$\alpha\equiv\beta\left(2\right)$, for otherwise $\alpha\equiv\beta+\varpi\left(2\right)$
would lead to $p^{\ell}\equiv1+|\alpha|^{2}+|\beta|^{2}\equiv3\left(4\right)$.
Thus, we have that either
\begin{equation}
g\equiv\left(\begin{smallmatrix}1 & 0 & 0\\
0 & 1 & 0\\
0 & 0 & 1
\end{smallmatrix}\right)\Mod{2},\quad\text{or}\quad g\equiv\left(\begin{smallmatrix}1 & \varpi & \varpi\\
\varpi & 1 & \varpi\\
\varpi & \varpi & 1
\end{smallmatrix}\right)\Mod{2}.\label{eq:Omega_p_mod2}
\end{equation}
In particular, the first row of $g$ is congruent to $\left(1,x\varpi,y\varpi^{2}+x\varpi\right)\Mod{\varpi^{3}=2+2i}$
for some $x,y\in\mathbb{Z}\left[i\right]$, and similarly $g'\in\Omega_{p}$
has first column $\equiv\left(1,x'\varpi,y'\varpi^{2}+x'\varpi\right)$,
so that 
\[
\left(gg'\right)_{1,1}\equiv1+xx'\varpi^{2}+yy'\varpi^{4}+\left(yx'+xy'\right)\varpi^{3}+xx'\varpi^{2}\equiv1\Mod{\varpi^{3}}.
\]
Similar reasoning for $\left(gg'\right)_{2,2}$ and $\left(gg'\right)_{3,3}$
gives $gg'\in\Omega_{p}$.

As $\Omega_{p}$ is a group, it normalizes itself, so that by (\ref{eq:gk-bijec})
it is enough to show that $GU_{3}\left(\mathbb{Z}\right)$ normalizes
$\Omega_{p}$ to obtain $\Omega_{p}\trianglelefteq GU_{3}(\mathbb{Z}\left[1/p\right])$,
and thus $GU_{3}(\mathbb{Z}\left[1/p\right])=\Omega_{p}\rtimes GU_{3}\left(\mathbb{Z}\right)$.
But conjugation of $g\in\Omega_{p}$ by a permutation matrix permutes
its diagonal elements, and for a diagonal matrix $k$, $g$ and $kgk^{-1}$
have the same diagonal, hence we are done.
\end{proof}
Recall the definition of $S_{p}$ in (\ref{eq:set-Sp}), and for $p\equiv1\Mod{4}$
define
\begin{equation}
S'_{p}:=\left\{ g\in S_{p}\,\middle|\,\ord_{\pi}\det g=1\right\} .\label{eq:Spprime}
\end{equation}

\begin{prop}
\label{prop:ST-action}Assume that $GU_{3}(\mathbb{Z}\left[1/p\right])$
acts transitively on $\mathcal{B}^{sp}$ (this is shown in Prop.\ \ref{prop:U3-class-number}
below). If (with some abuse of notation) we consider $S_{p},S_{p}',\Omega_{p},\Lambda_{p}=\left\langle S_{p}\right\rangle $
as subsets of $PGU_{3}(\mathbb{Z}\left[1/p\right])$, then:
\begin{enumerate}
\item \label{enu:ST-action}$\Omega_{p}$ acts simply-transitively on $\mathcal{B}^{sp}$.
\item \label{enu:Sp-Spt}$S_{p}=\left\{ g\in\Omega_{p}\,\middle|\,\ell\left(g\right)=\ord_{p}p'\right\} $
(see (\ref{eq:pprime-Gp})), and $S_{p}'=\left\{ g\in S_{p}\,\middle|\,h_{\pi}(g)=1\right\} $.
\item \label{enu:Sp-neighbors}$S_{p}v_{0}$ are the hyperspecial vertices
closest to $v_{0}$; $S_{p}'v_{0}$ are those of color one among them,
and
\[
\left|S'_{p}\right|=p^{2}+p+1,\qquad\qquad\left|S_{p}\right|=\begin{cases}
2\left(p^{2}+p+1\right) & p\equiv1\left(4\right)\\
p^{4}+p & p\equiv3\left(4\right)
\end{cases}
\]
\item \label{enu:Spl-sphere}For any $v\in\mathcal{B}^{sp}$, writing $v=gv_{0}$
($g\in\Omega_{p}$) and using the notation (\ref{eq:S(l)}) we have
\[
gS_{p}^{'\left(\ell\right)}v_{0}=\left\{ w\,\middle|\,\dist_{1}\left(v,w\right)=\ell\right\} ,\quad gS_{p}^{\left(\ell\right)}v_{0}=\begin{cases}
\left\{ w\,\middle|\,\dist\left(v,w\right)=\ell\right\}  & p\equiv1\left(4\right)\\
\left\{ w\,\middle|\,\dist\left(v,w\right)=2\ell\right\}  & p\equiv3\left(4\right)
\end{cases}
\]
\item \label{enu:sg-and-group-gen}$\Lambda_{p}=\left\langle S_{p}\right\rangle _{(sg)}=\left\langle S'_{p}\right\rangle _{(sg)}=\Omega_{p}$,
where $\left\langle \,\cdot\,\right\rangle _{\left(sg\right)}$ means
generation both as a group and as a semigroup.
\end{enumerate}
\end{prop}

\begin{proof}
\emph{(1)} Since $GU_{3}\left(\mathbb{Z}\right)$ fixes $v_{0}$,
it follows from Prop.\ \ref{prop:split-GU} that $\Omega_{p}$ acts
transitively on $\mathcal{B}^{sp}$. On the other hand, if $g\in\Omega_{p}$
fixes $v_{0}$, then 
\[
g\in GU_{3}(\mathbb{Z}\big[\tfrac{1}{p}\big])\cap\mathrm{Stab}_{GU_{3}(\mathbb{Q}_{p})}(v_{0})=GU_{3}(\mathbb{Z}\big[\tfrac{1}{p}\big])\cap\mathbb{Q}_{p}[i]^{\times}GU_{3}(\mathbb{Z}_{p})=\mathbb{Z}\big[i,\tfrac{1}{p}\big]^{\times}GU_{3}(\mathbb{Z}).
\]
Writing $g=\lambda k$ with $\lambda\in\mathbb{Z}\left[i,1/p\right]^{\times}$
and $k\in GU_{3}\left(\mathbb{Z}\right)$, we can take $u\in\mathbb{Z}\left[i\right]^{\times}$
such that $u\equiv\lambda^{-1}\Mod{2+2i}$. We then have $\lambda uI\in\Omega_{p}$,
and from $u^{-1}k=\left(\lambda u\right)^{-1}g\in GU_{3}\left(\mathbb{Z}\right)\cap\Omega_{p}=1$
follows that $g$ is scalar.

\emph{(2)} If $s\in S_{p}$ then $\ord_{\pi}s=0$: otherwise, $\frac{s}{\pi}\left(\frac{s}{\pi}\right)^{*}=\frac{p'I}{\pi\overline{\pi}}=I$
would give $\frac{s}{\pi}\in U_{3}\left(\mathbb{Z}\right)\cap\pi^{-1}\Omega_{p}$,
which contains only scalars. Thus, $S_{p}'=\left\{ g\in S_{p}\,\middle|\,h_{\pi}(g)=1\right\} $
and similarly $\ord_{\overline{\pi}}s=0$, so that $\ell(s)=\ord_{p}p'$.
On the other hand, assume that $g\in\Omega_{p}$ and $\ell\left(g\right)=\ord_{p}p'$.
As we are in $PGU_{3}$, we can multiply $g$ by some $\lambda\in\mathbb{Z}\left[i,1/p\right]^{\times}$
to obtain $\ord_{\pi}\lambda g=\ord_{\overline{\pi}}\lambda g=0$
and $\lambda g\in\Omega_{p}$, hence $g=\lambda g\in S_{p}$, and
in total $S_{p}=\left\{ g\in\Omega_{p}\,\middle|\,\ell\left(g\right)=\ord_{p}p'\right\} $.

\emph{(3)} As $\ord_{p}p'$ is the graph-distance from $v_{0}$ to
its closest hyperspecial vertices, the first claim follows by \emph{(\ref{enu:ST-action}),
(\ref{enu:Sp-Spt})} and Prop.\ \ref{prop:distances}. Furthermore,
$GU_{3}(\mathbb{Z}\left[1/p\right])\rightarrow PGU_{3}(\mathbb{Z}\left[1/p\right])$
is injective on $S_{p}$ (since $\ord_{\pi}|_{S_{p}}\equiv\ord_{\overline{\pi}}|_{S_{p}}\equiv0$,
$\mathbb{Z}\left[i,1/p\right]^{\times}=\left\langle i,\pi,\overline{\pi}\right\rangle $,
and $-1,\pm i\not\equiv1\Mod{2+2i}$), so that $|S_{p}|$ equals the
number of closest hyperspecial vertices. The statements for $S_{p}'$
are proved similarly, using the $\pi$-height in Prop.\ \ref{prop:distances},
and $\dist_{1}$ replacing the graph-distance.

\emph{(4)} Let $p\equiv1\left(4\right)$. Since $G_{p}$ acts by isometries,
it follows from \emph{(\ref{enu:Sp-neighbors})} that $gS_{p}v_{0}$
are the neighbors of $gv_{0}$. By induction, $w\in gS_{p}^{t}v_{0}$
iff there is a path of length $t$ from $v$ to $w$, and thus $w\in gS_{p}^{\left(\ell\right)}v_{0}$
if the shortest such path is of length $\ell$. The analogue statements
for $S_{p}'$ and $p\equiv3\left(4\right)$ are proved similarly.

\emph{(5)}\textbf{ }It follows from \emph{(\ref{enu:Sp-neighbors})
}and the connectivity of $\mathcal{B}$ that $\left\langle S_{p}\right\rangle _{sg}$
acts transitively on $\mathcal{B}^{sp}$. We have $\left\langle S_{p}\right\rangle _{sg}\subseteq\Omega_{p}$
since $\Omega_{p}$ is a group, but as $\Omega_{p}$ acts simply-transitively
on $\mathcal{B}^{sp}$, we cannot have $\left\langle S_{p}\right\rangle _{sg}\subsetneq\Omega_{p}$,
and we must conclude that $\Lambda_{p}=\left\langle S_{p}\right\rangle _{(sg)}=\Omega_{p}$.
Finally, by \emph{(\ref{enu:Sp-neighbors})}, $\smash{S_{p}^{'}S_{p}^{'}}v_{0}$
contains all the color-two neighbors of the color-one neighbors of
$v_{0}$; as these include the color-two neighbors of $v_{0}$, and
$\left\langle S_{p}\right\rangle $ acts simply-transitively, this
implies that $S_{p}\subseteq S_{p}'S_{p}'$, so that $\left\langle S_{p}'\right\rangle _{(sg)}=\Lambda_{p}$.
\end{proof}
Theorem \ref{thm:simp-tran} follows from \emph{(\ref{enu:ST-action})}
and \emph{(\ref{enu:sg-and-group-gen})} above, and the only missing
piece is:
\begin{prop}
\label{prop:U3-class-number}The group $U_{3}\left(\mathbb{Z}[1/p]\right)$
(and thus also $GU_{3}(\mathbb{Z}[1/p])$) acts transitively on the
hyperspecial vertices of the Bruhat-Tits building $\mathcal{B}$ associated
with $PGU_{3}(\mathbb{Q}_{p})$.
\end{prop}

\begin{proof}
First, we claim that for any $d\geq3$ 
\begin{equation}
U_{d}(\mathbb{Q})U_{d}(\mathbb{R}\mathbb{Q}_{p}\hat{\mathbb{Z}})=U_{d}\left(\mathbb{A}\right).\label{eq:Ud-SA}
\end{equation}
As $SU_{d}(\mathbb{Q}_{p})$ is non-compact for $d\geq3$, $SU_{d}\left(\mathbb{A}\right)\subseteq U_{d}(\mathbb{Q})U_{d}(\mathbb{R}\mathbb{Q}_{p}\hat{\mathbb{Z}})$
follows from strong approximation (cf.\ \cite[Thm.\ 7.12]{Platonov1994Algebraicgroupsand}),
and by the exact sequence $1\rightarrow SU_{d}\left(\mathbb{A}\right)\rightarrow U_{d}\left(\mathbb{A}\right)\overset{{\scriptscriptstyle \det}}{\longrightarrow}U_{1}\left(\mathbb{A}\right)\rightarrow1$,
(\ref{eq:Ud-SA}) would follow from $\det(U_{d}(\mathbb{Q})U_{d}(\mathbb{R}\mathbb{Q}_{p}\hat{\mathbb{Z}}))=U_{1}\left(\mathbb{A}\right)$.
Embedding $U_{1}\hookrightarrow U_{d}$ by $\alpha\mapsto\diag\left(\alpha,1,\ldots,1\right)$,
we see that it is enough to show that 
\begin{equation}
U_{1}(\mathbb{Q})U_{1}(\mathbb{R}\hat{\mathbb{Z}})=U_{1}(\mathbb{A}),\label{eq:U1-classnumber1}
\end{equation}
namely, that for $(\alpha_{\ell})_{\ell}\in U_{1}(\mathbb{A})$ there
exist $q\in U_{1}(\mathbb{Q})$ such that $q\alpha_{\ell}\in U_{1}\left(\mathbb{Z}_{\ell}\right)$
for all $\ell$. For $\ell\not\equiv1\Mod{4}$, $\mathbb{Q}_{\ell}[i]$
is a local field whose valuation $\nu$ is invariant under conjugation,
so that $2\nu\left(\alpha_{\ell}\right)=\nu\left(\alpha_{\ell}\right)+\nu\left(\overline{\alpha_{\ell}}\right)=\nu\left(\alpha_{\ell}\overline{\alpha_{\ell}}\right)=\nu\left(1\right)=0$
implies $\alpha_{\ell}\in\mathcal{O}_{\mathbb{Q}_{\ell}\left[i\right]}^{\times}=\mathbb{Z}_{\ell}\left[i\right]^{\times}$,
hence $U_{1}(\mathbb{Q}_{\ell})=U_{1}(\mathbb{Z}_{\ell})$ to begin
with; For $\ell\equiv1\Mod{4}$, writing $\ell=\pi\overline{\pi}$
in $\mathbb{Z}[i]$, and embedding $\mathbb{Z}[i]\hookrightarrow\mathbb{Q}_{\ell}$
we have w.l.o.g.\ $\nu_{\ell}\left(\pi\right)=1$, $\nu_{\ell}\left(\overline{\pi}\right)=0$,
so that $q=\pi/\overline{\pi}$ satisfies $q\in U_{1}(\mathbb{Q})$
and $\nu_{\ell}\left(q\right)=1$. Thus, $q^{m}\alpha_{\ell}\in U_{1}\left(\mathbb{Z}_{\ell}\right)$
for some $m$ and furthermore, $q\in U_{1}\left(\mathbb{Z}_{\ell'}\right)$
for $\ell'\neq\ell$. This shows (\ref{eq:U1-classnumber1}), and
(\ref{eq:Ud-SA}) follows. Denoting $\hat{\mathbb{Z}}^{p}=\prod_{\ell\neq p}\mathbb{Z}_{\ell}$,
from (\ref{eq:Ud-SA}) we have $PU_{3}(\mathbb{Q})PU_{3}(\mathbb{R}\hat{\mathbb{Z}}^{p})=PU_{3}\left(\mathbb{A}^{p}\right)$,
and together with $PU_{3}(\mathbb{Q})\cap PU_{3}(\mathbb{R}\hat{\mathbb{Z}}^{p})=PU_{3}(\mathbb{Z}\left[1/p\right])$
we obtain 
\begin{gather*}
PU_{3}(\mathbb{Q})\backslash PU_{3}(\mathbb{A})/PU_{3}\big(\mathbb{R}\hat{\mathbb{Z}}\big)=PU_{3}(\mathbb{Q})\backslash\left(\nicefrac{PU_{3}(\mathbb{A}^{p})}{PU_{3}(\mathbb{R}\hat{\mathbb{Z}}^{p})}\times\nicefrac{PU_{3}(\mathbb{Q}_{p})}{PU_{3}(\mathbb{Z}_{p})}\right)\\
\cong PU_{3}(\mathbb{Q})\backslash\left(\nicefrac{PU_{3}(\mathbb{Q})}{PU_{3}(\mathbb{Z}\left[1/p\right])}\times\mathcal{B}^{sp}\right)\cong PU_{3}(\mathbb{Z}\left[1/p\right])\backslash\mathcal{B}^{sp}.
\end{gather*}
It is left to show that 
\begin{equation}
U_{3}(\mathbb{Q})U_{3}(\mathbb{R}\hat{\mathbb{Z}})=U_{3}(\mathbb{A}),\label{eq:clnum1}
\end{equation}
namely that $U_{3}(\mathbb{Q})$ has class number one. We proceed
in analogy with Siegel's treatment of special orthogonal groups in
\cite{Siegel1963Lecturesanalyticaltheory} (see \cite[§8]{Tamagawa1966Adeles}
and \cite[§2.5]{Hanke2013Quadraticformsautomorphic}), replacing
Siegel's space of symmetric forms by Hermitian ones. Let $h$ be the
class number of $U_{3}$, and write $U_{3}(\mathbb{A})=\bigsqcup_{i=1}^{h}U_{3}(\mathbb{Q})g_{i}U_{3}(\mathbb{R}\hat{\mathbb{Z}})$
with $g_{1}=I$. As $SU_{3}$ has Tamagawa number 1 \cite{Mars1969TamagawaNumber2A},
the exact sequence $1\rightarrow SU_{3}\rightarrow U_{3}\rightarrow SO_{2}\rightarrow1$
shows that the Tamagawa number of $U_{3}$ is $2$. Thus, if $\mu$
is the Tamagawa measure of $U_{3}$ over $\mathbb{Q}$ then 
\begin{equation}
\begin{aligned}2 & =\mu(U_{3}(\mathbb{Q})\backslash U_{3}(\mathbb{A}))=\sum\nolimits _{i=1}^{h}\mu(U_{3}(\mathbb{Q})\backslash U_{3}(\mathbb{Q})g_{i}U_{3}(\mathbb{R}\hat{\mathbb{Z}}))\\
 & =\sum\nolimits _{i=1}^{h}\frac{\mu(U_{3}(\mathbb{R}\hat{\mathbb{Z}}))}{|U_{3}(\mathbb{Q})^{g_{i}}\cap U_{3}(\mathbb{R}\hat{\mathbb{Z}})|}=\sum\nolimits _{i=1}^{h}\frac{\mu(U_{3}(\mathbb{R}))\prod_{\ell}\mu(U_{3}(\mathbb{Z}_{\ell}))}{|U_{3}(\mathbb{Q})^{g_{i}}\cap U_{3}(\mathbb{R}\hat{\mathbb{Z}})|}.
\end{aligned}
\label{eq:Tamagawa-Mass}
\end{equation}
For each prime $\ell\leq\infty$, the volume $\mu_{\ell}=\mu(U_{3}(\mathbb{Z}_{\ell}))$
is the $\ell$-adic density of the solutions to the equation $A^{*}A=I$
in $\mathbb{Z}_{\ell}$ (where $\mathbb{Z}_{\infty}=\mathbb{R}$).
For finite $\ell$ it is given by 
\begin{equation}
\mu_{\ell}=\lim_{r\rightarrow\infty}\frac{|U_{3}(\mathbb{Z}/\ell^{r})|}{\ell^{r\dim U_{3}}}=\lim_{r\rightarrow\infty}\frac{|U_{3}(\mathbb{Z}/\ell^{r})|}{\ell^{9r}}.\label{eq:density-as-limit}
\end{equation}
For $\ell\neq2$ this limit stabilizes at $r=1$, and we are reduced
to computing orders of finite Chevalley groups: 
\[
|U_{3}(\mathbb{F}_{\ell})|=\begin{cases}
|GL_{3}(\mathbb{F}_{\ell})|=(\ell^{3}-1)(\ell^{3}-\ell)(\ell^{3}-\ell^{2}) & \ell\equiv1\;(mod\;4)\\
|U_{3}(\mathbb{F}_{\ell})|=(\ell^{3}+1)(\ell^{3}-\ell)(\ell^{3}+\ell^{2}) & \ell\equiv3\;(mod\;4).
\end{cases}
\]
For $\ell=2$, the limit (\ref{eq:density-as-limit}) stabilizes at
$r=2$, and gives $\mu_{2}=\frac{3}{2}$. Together, 
\[
\begin{aligned}\prod\nolimits _{\ell\neq\infty}\mu_{\ell} & =\tfrac{3}{2}\cdot\prod\nolimits _{\ell\neq2,\infty}\left(1-\left(\tfrac{-1}{\ell}\right)\ell^{-1}\right)\cdot\left(1-\ell^{-2}\right)\cdot\left(1-\left(\tfrac{-1}{\ell}\right)\ell^{-3}\right)\\
 & =\frac{3}{2}\cdot\frac{1}{L\left(1,\left(\tfrac{-1}{\cdot}\right)\right)}\cdot\frac{4/3}{\zeta\left(2\right)}\cdot\frac{1}{L\left(3,\left(\tfrac{-1}{\cdot}\right)\right)}=\frac{1536}{\pi^{6}}.
\end{aligned}
\]
For $\ell=\infty$, let $\Omega$ be a neighborhood of $I$ in $M_{3}^{Her}\left(\mathbb{C}\right)$,
the space of hermitian $3\times3$ matrices. The real density of $A^{*}A=I$
is given by 
\[
\mu_{\infty}=\lim_{\Omega\searrow\{I\}}\frac{vol(\{A\in M_{3}(\mathbb{C})\;|\;A^{*}A\in\Omega\})}{vol(\Omega)},
\]
where $vol$ is the additive Lebesgue measure. We take $\Omega_{\varepsilon}=\{A\;|\;|A_{ij}-\delta_{ij}|\leq\varepsilon\}$,
so that $vol(\Omega_{\varepsilon})=(2\varepsilon)^{3}(\pi\varepsilon^{2})^{3}$,
and parameterize $M_{3}(\mathbb{C})$ by $A=QR$ with $Q\in U_{3}(\mathbb{R})$
and $R$ triangular with non-negative diagonal. We parameterize $U\left(n\right)$
recursively, by observing that 
\[
\nicefrac{U\left(n+1\right)}{U\left(n\right)\times U\left(1\right)}\simeq\mathbb{C}P^{n},
\]
and there is a section to this fibration (up to a set of measure zero):
\[
U\left(n+1\right)=\left(\begin{array}{c|c}
\Large{\mbox{\ensuremath{U\left(n\right)}}}\\
\hline  & {\scriptstyle U\left(1\right)}
\end{array}\right)\cdot\left(\begin{array}{c|c}
\Large{\mbox{\ensuremath{I+\left(\cos\left(\zeta_{n}\right)-1\right)\mbox{\ensuremath{v}}_{n}v_{n}^{*}}}} & \sin\left(\zeta_{n}\right)v_{n}\\
\hline -\sin\left(\zeta_{n}\right)v_{n}^{*} & {\scriptstyle \cos\left(\zeta_{n}\right)}
\end{array}\right);
\]
here $\zeta_{n}\in\left[0,\frac{\pi}{2}\right]$ and $v_{n}\in\mathbb{C}^{n}$
with $\left\Vert v_{n}\right\Vert =1$, so that $\left[\sin\left(\zeta_{n}\right)v_{n}:\cos\left(\zeta_{n}\right)\right]$
parametrizes $\mathbb{C}P^{n}$. We use standard spherical coordinates
for $\mathbb{C}^{n}\cong\mathbb{R}^{2n}$, namely
\[
v_{1}=\left(\cos\varphi_{11}+i\sin\varphi_{11}\right),\qquad v_{2}=\left(\begin{array}{c}
\cos\varphi_{21}+i\sin\varphi_{21}\cos\varphi_{22}\\
\sin\varphi_{21}\sin\varphi_{22}(\cos\varphi_{23}+i\sin\varphi_{23})
\end{array}\right).
\]
Taking
\[
R=\left(\begin{matrix}a_{1} & n_{12}\lambda_{12} & a_{2}n_{13}\lambda_{13}\\
0 & a_{2} & n_{23}\lambda_{23}-n_{12}n_{13}\overline{\lambda_{12}}\lambda_{13}\\
0 & 0 & a_{3}
\end{matrix}\right)
\]
with $a_{i},n_{ij}>0$ (losing again a set of measure zero) and $\lambda_{ij}\in S^{1}$,
we obtain 
\[
A^{*}A=R^{*}R=\left(\begin{array}{ccc}
a_{1}^{2} & a_{1}n_{12}\lambda_{12} & a_{1}a_{2}n_{13}\lambda_{13}\\
a_{1}n_{12}\overline{\lambda_{12}} & a_{2}^{2}+n_{12}^{2} & a_{2}n_{23}\lambda_{23}\\
a_{1}a_{2}n_{13}\overline{\lambda_{13}} & a_{2}n_{23}\overline{\lambda_{23}} & {a_{3}^{2}+a_{2}^{2}n_{13}^{2}+n_{23}^{2}+n_{12}^{2}n_{13}^{2}\atop -2n_{12}n_{13}n_{23}\Re\left(\lambda_{12}-\lambda_{13}+\lambda_{23}\right)}
\end{array}\right),
\]
and $A^{*}A\in\Omega_{\varepsilon}$ iff $|a_{1}^{2}-1|,|a_{1}n_{12}|,|a_{1}a_{2}n_{13}|,|a_{2}^{2}+n_{12}^{2}-1|,\ldots\leq\varepsilon$.
The Jacobian of our parametrization, with respect to the Lebesgue
measure on $M_{3}(\mathbb{C})\cong\mathbb{R}^{18}$, has determinant
\[
|\det\mathbf{J}|=a_{1}^{5}a_{2}^{5}a_{3}n_{12}n_{13}n_{23}\sin(\zeta_{1})\cos(\zeta_{1})\sin(\varphi_{22})\sin^{3}(\zeta_{2})\cos(\zeta_{2})\sin^{2}(\varphi_{21}),
\]
and direct computation (see code on the author's website) gives 
\[
vol(\{A\in M_{3}(\mathbb{C})\;|\;A^{*}A\in\Omega_{\varepsilon}\})=\int_{\{A^{*}A\in\Omega_{\varepsilon}\}}|\mathbf{\det J}|=4\pi^{9}\varepsilon^{9}.
\]
Thus, $\mu_{\infty}=\frac{4\pi^{9}\varepsilon^{9}}{(2\varepsilon)^{3}(\pi\varepsilon^{2})^{3}}=\frac{\pi^{6}}{2},$
and in total $\prod_{\ell\leq\infty}\mu_{\ell}=\frac{1536}{\pi^{6}}\cdot\frac{\pi^{6}}{2}=768$.
Returning to (\ref{eq:Tamagawa-Mass}), we now have 
\[
2=\sum_{i=1}^{h}\frac{768}{|U_{3}(\mathbb{Q})^{g_{i}}\cap U_{3}(\mathbb{R}\hat{\mathbb{Z}})|},
\]
and by $|U_{3}(\mathbb{Q})^{g_{1}}\cap U_{3}(\mathbb{R}\hat{\mathbb{Z}})|=\left|U_{3}(\mathbb{Z})\right|=384$
we conclude that $h=1$.
\end{proof}

\subsection{\label{subsec:Navigation-in-Lambdap}Navigation in $\Lambda_{p}$}

Another consequence of the analysis carried out in §\ref{subsec:Simply-transitive-action}
is a solution to the word problem in $\Lambda_{p}$, either as a group
or a semigroup, with generating set $S_{p}'$ or $S_{p}$ (it follows
from Prop.\ \ref{prop:ST-action} that $S_{p}=S_{p}^{-1}=S_{p}'\sqcup S_{p}^{'^{-1}}$
for $p\equiv1\left(4\right)$).

If $g\in\Lambda_{p}$, then $g\in S_{p}^{'(h)}$ for some $h\geq0$
by Prop.\ \ref{prop:ST-action}\emph{(\ref{enu:sg-and-group-gen})},
and $h=\dist_{1}\left(v_{0},gv_{0}\right)=h_{\pi}\left(g\right)$
by Prop.\ \ref{prop:ST-action}\emph{(\ref{enu:Spl-sphere})} and
Prop.\ \ref{prop:distances}. Choose a 1-path $v_{0}\rightarrow v_{1}\rightarrow\ldots\rightarrow v_{h}=gv_{0}$,
and take $s_{1}$ to be the $S_{p}^{'}$-element such that $s_{1}v_{0}=v_{1}$
(see Prop.\ \ref{prop:ST-action}\emph{(\ref{enu:Sp-neighbors})}).
Now $(s_{1}^{*}v_{i})_{i=1}^{h}$ is a 1-path of length $h-1$ from
$v_{0}$ to $s_{1}^{*}gv_{0}$, and it follows that $h_{\pi}\left(s_{1}^{*}g\right)=h-1$.
Continuing in this manner, there are $s_{1},\ldots,s_{h}\in S_{p}'$
such that $h_{\pi}\left(s_{h}^{*}\ldots s_{1}^{*}g\right)=0$, implying
$s_{h}^{*}\ldots s_{1}^{*}gv_{0}=v_{0}$ and thus $g=s_{1}\ldots s_{h}$
by Theorem \ref{thm:simp-tran}. Navigation with $S_{p}$ is accomplished
similarly, using the level of $g$ in place of the $\pi$-height.

The naive way to find $s_{1}$ as above is to compute $h_{\pi}\left(s^{*}g\right)$
for each $s\in S_{p}'$ (or $\ell(s^{*}g)$ for $s\in S_{p}$); However,
this requires $O(p^{2})$ or $O(p^{4})$ iterations, and can be avoided,
as follows: for $p\equiv1\Mod{4}$, computing $h$ digits of $\smash{\sqrt{-1}}\in\mathbb{Q}_{p}$,
we can write the isomorphism $g\mapsto\widetilde{g}:PGU_{3}(\mathbb{Q}_{p})\smash{\overset{{\scriptscriptstyle \cong}}{\longrightarrow}}PGL_{3}(\mathbb{Q}_{p})$
explicitly, up to the precision required for computations in the $h$-ball
around $v_{0}$. Then, starting with the Iwasawa decomposition of
$GL_{3}(\mathbb{Q}_{p})$ and performing elementary column operations,
any $g\in PGL_{3}(\mathbb{Q}_{p})$ can be written uniquely as $g=b_{g}k_{g}$
for $k_{g}\in\mathbb{Q}_{p}^{\times}GL_{3}\left(\mathbb{Z}_{p}\right)$
and $b_{g}\in M_{3}\left(\mathbb{Z}\right)$ of the form $b=\left(\begin{smallmatrix}p^{m_{1}} & x & y\\
 & p^{m_{2}} & z\\
 &  & p^{m_{3}}
\end{smallmatrix}\right)$ with $0\leq x,y<p^{m_{1}}$, $0\leq z<p^{m_{2}}$, and $\gcd\left(b_{ij}\right)=1$.
Let $j$ denote the maximal $i$ for which $m_{i}\neq0$ in $b_{g}$
above. Computing once $b_{s}$ for each $s\in S_{p}^{'}$, we can
take $s_{1}$ to be the unique $s\in S_{p}^{'}$ such that the $j$-th
rows of $b_{g}$ and $b_{s}$ agree modulo $p$. We leave the other
cases to the reader, remarking only that for $p\equiv3\Mod{4}$, one
should use the Iwasawa decomposition of $GL_{3}(\mathbb{Q}_{p}\left[i\right])$,
instead of $GL_{3}(\mathbb{Q}_{p})$.

\subsection{\label{subsec:Subgroups-and-quotients}Subgroups and quotients}

In this section we explore the subgroups and quotients of $\Lambda_{p}$
which give our explicit constructions of Ramanujan complexes. We use
the definitions from (\ref{eq:set-Sp}), and recall that $\Lambda_{p}(q)=\{g\in\Lambda_{p}\,|\,g\equiv I\pmod q\}$.
\begin{prop}
\label{prop:Xpq_quotient}Recall that we denoted by $X^{p,q}$ the
Cayley graph $Cay\left(G_{p,q},S_{p,q}\right)$.
\begin{enumerate}
\item If $p\equiv1\Mod{4}$ then $X^{p,q}$ coincides with the one-skeleton
of $X_{\Lambda_{p}(q)}$.
\item If $p\equiv3\Mod{4}$ then $X^{p,q}$ is the graph induced by non-backtracking
$2$-walks on the hyperspecial side of $X_{\Lambda_{p}(q)}$.
\end{enumerate}
\end{prop}

\begin{proof}
If $p\equiv1\Mod{4}$, then $\mathcal{B}^{(1)}$, the one-skeleton
of the building of $PU_{3}(\mathbb{Q}_{p})\cong PGL_{3}(\mathbb{Q}_{p})$,
coincides with $Cay\left(\Lambda_{p},S_{p}\right)$ by Theorem \ref{thm:simp-tran}
and Prop.\ \ref{prop:ST-action}(\ref{enu:Spl-sphere}) (with $\ell=1$),
hence
\begin{align*}
X_{\Lambda_{p}(q)}^{\left(1\right)} & =\Lambda_{p}(q)\backslash\mathcal{B}^{(1)}=\Lambda_{p}(q)\backslash Cay\left(\Lambda_{p},S_{p}\right)\\
 & =Cay\left(\Lambda_{p}(q)\backslash\Lambda_{p},S_{p,q}\right)=Cay\left(G_{p,q},S_{p,q}\right)=X^{p,q}.
\end{align*}
When $p\equiv3\Mod{4}$ Theorem \ref{thm:simp-tran} and Prop.\ \ref{prop:ST-action}(\ref{enu:Spl-sphere})
show that $Cay\left(\Lambda_{p},S_{p}\right)$ is the graph with vertices
$\mathcal{B}^{sp}$, and edges corresponding to non-backtracking 2-walks
in the tree $\mathcal{B}$. Therefore, the graph with vertices $X_{\Lambda_{p}(q)}^{sp}$
and edges corresponding to non-backtracking 2-walks in $X_{\Lambda_{p}(q)}$
coincides with $\Lambda_{p}(q)\backslash Cay\left(\Lambda_{p},S_{p}\right)=X^{p,q}$
as before.
\end{proof}
For any odd prime power $q$ and $a\in\mathbb{F}_{q}^{\times}$, we
denote the cubic residue symbol:
\[
\left[\frac{a}{q}\right]_{3}=\begin{cases}
\ 1 & a=b^{3}\mbox{ for some }b\in\mathbb{F}_{q}^{\times}\\
-1 & \text{otherwise}
\end{cases}.
\]

\begin{prop}
\label{prop:Gpq}For distinct odd primes $p$ and $q$, Table \ref{tab:Gpq}
determines the group $G_{p,q}=\Lambda_{p}\Mod{q}$. For $p\equiv1\left(4\right)$,
the Cayley complex $X_{\Lambda_{p}(q)}$ is tri-partite in the cases
labeled by $PGL$ and $PU$, and non-tri-partite in the $PSL$ and
$PSU$ cases.
\end{prop}

\begin{table}[h]
\centering{}%
\begin{tabular}{|c|c|c|c|c|c|c|}
\hline 
\multirow{2}{*}{$G_{p,q}$} & \multicolumn{2}{c|}{$q\equiv1\Mod{12}$} & \multirow{2}{*}{$q\equiv5\left(12\right)\negmedspace$} & \multirow{2}{*}{$\negmedspace q\equiv3,7\left(12\right)\negmedspace$} & \multicolumn{2}{c|}{$q\equiv11\left(12\right)$}\tabularnewline
\cline{2-3} \cline{3-3} \cline{6-7} \cline{7-7} 
 & $\big[\tfrac{p\pi}{q}\big]_{3}=1\vphantom{\Big|}$ & $\big[\tfrac{p\pi}{q}\big]_{3}=-1$ &  &  & $\big[\tfrac{p\pi}{q^{_{2}}}\big]_{3}=1$ & $\big[\tfrac{p\pi}{q^{_{2}}}\big]_{3}=-1$\tabularnewline
\hline 
\hline 
$p\equiv1\left(4\right)$ & $\negthickspace PSL_{3}\left(\mathbb{F}_{q}\right)\negthickspace$ & $\negthickspace PGL_{3}\left(\mathbb{F}_{q}\right)\negthickspace$ & $\negthickspace PSL_{3}\left(\mathbb{F}_{q}\right)\negthickspace$ & $\negthickspace PSU_{3}\left(\mathbb{F}_{q}\right)\negthickspace\negmedspace$ & $\negthickspace PSU_{3}\left(\mathbb{F}_{q}\right)\negthickspace$ & $\negthickspace PU_{3}\left(\mathbb{F}_{q}\right)\negthickspace$\tabularnewline
\hline 
$p\equiv3\left(4\right)$ & \multicolumn{3}{c|}{$PSL_{3}\left(\mathbb{F}_{q}\right)$} & \multicolumn{3}{c|}{$PSU_{3}\left(\mathbb{F}_{q}\right)$}\tabularnewline
\hline 
\end{tabular}\caption{\label{tab:Gpq}The possibilities for the group $G_{p,q}$, where
$\pi$ is a $\mathbb{Z}\left[i\right]$-factor of $p$, interpreted
in $\mathbb{F}_{q}$ or $\mathbb{F}_{q^{2}}$ according to $q\Mod{4}$.}
\end{table}

\begin{proof}
The group $G_{p,q}$ is the image of $\Lambda_{p}=\left\langle S_{p}\right\rangle $
under the mapping of $\Gamma_{p}$ into $PGU_{3}(\mathbb{F}_{q})\cong PU_{3}(\mathbb{F}_{q})$
(Lemma \ref{lem:PGU=00003DPU}), and $PU_{3}\left(\mathbb{F}_{q}\right)\cong PGL_{3}\left(\mathbb{F}_{q}\right)$,
$PSU_{3}\left(\mathbb{F}_{q}\right)\cong PSL_{3}\left(\mathbb{F}_{q}\right)$
when $q\equiv1\left(4\right)$. As $\Lambda_{p}$ is a congruence
subgroup (of level $4$), by strong approximation we have $PSU_{3}(\mathbb{F}_{q})\subseteq G_{p,q}$.
The sequence $1\rightarrow PSU_{3}(\mathbb{F}_{q})\rightarrow PU_{3}(\mathbb{F}_{q})\overset{{\scriptscriptstyle \det}}{\longrightarrow}U_{1}(\mathbb{F}_{q})/U_{1}(\mathbb{F}_{q})^{3}\rightarrow1$
is exact, and when $q\equiv3,5,7\left(12\right)$ one finds that $U_{1}(\mathbb{F}_{q})^{3}=U_{1}(\mathbb{F}_{q})$,
hence $PSU_{3}\left(\mathbb{F}_{q}\right)=PU_{3}\left(\mathbb{F}_{q}\right)$.

We assume now $q\equiv1$ or $11\pmod{12}$, which implies $\left[PU_{3}(\mathbb{F}_{q}):PSU_{3}(\mathbb{F}_{q})\right]=3$,
hence $G_{p,q}=PSU_{3}(\mathbb{F}_{q})$ iff $\det\left(\frac{s}{\pi}\right)\in U_{1}(\mathbb{F}_{q})^{3}$
for every $s\in S_{p}$ (here $\frac{s}{\pi}$ is a homothety-representative
for $s$ in $PU_{3}$). We have $i\in U_{1}(\mathbb{F}_{q})^{3}$,
so that if $p\equiv3\left(4\right)$, then $\det s\sim p^{3}$ (in
$\mathbb{Z}\left[i\right]$) implies $G_{p,q}=PSU_{3}$. For $p\equiv1\left(4\right)$,
we can assume $s\in S_{p}\backslash S_{p}'$ (as these already generate
$\Lambda_{p}$), hence $\det s\sim p\pi$; for $q\equiv1\left(12\right)$
we then have $U_{1}(\mathbb{F}_{q})\cong\mathbb{F}_{q}^{\times}$
and $\det\frac{s}{\pi}\in(\mathbb{F}_{q}^{\times})^{3}$ iff $\left[\frac{\det(s/\pi)}{q}\right]_{3}=\left[\frac{p\pi}{q}\right]_{3}$
is one, and for $q\equiv11\left(12\right)$ we have $\det\frac{s}{\pi}\in U_{1}(\mathbb{F}_{q})^{3}$
iff $\det s\in(\mathbb{F}_{q^{2}}^{\times})^{3}$, namely $\left[\frac{p\pi}{q^{2}}\right]_{3}=1$.

Finally, when $p\equiv1\left(4\right)$, the vertex coloring (\ref{eq:color-def})
corresponds to an epimorphism $\tau\colon\Lambda_{p}\rightarrow\mathbb{Z}/3\mathbb{Z}$
which factors through the determinant map, and in fact $\tau=\ord_{\pi}\circ\det$.
The quotient complex $X_{\Lambda_{p}(q)}$ is tri-partite iff $\Lambda_{p}(q)\subseteq\ker\tau$;
if this happens, then $\tau$ induces an epimorphism from $\Lambda_{p}(q)\backslash\Lambda_{p}\cong G_{p,q}$
onto $\mathbb{Z}/3\mathbb{Z}$ which factor through the determinant,
which implies that necessarily $PSU_{3}(\mathbb{F}_{q})\neq G_{p,q}=PU_{3}(\mathbb{F}_{q})$.
It is thus left to show that $G_{p,q}=PSU_{3}(\mathbb{F}_{q})$ whenever
$X_{\Lambda_{p}(q)}$ is non-tri-partite, and we may assume $q\equiv\pm1\left(12\right)$.
Let $g\in\Lambda_{p}(q)\backslash\ker\tau$. As $g$ is scalar modulo
$q$, we have $\left[\frac{\det g}{q^{\star}}\right]_{3}=1$ where
$\star=\left\{ \!\begin{smallmatrix}1~\ q\equiv\phantom{1}1\left(12\right)\\
2~\ q\equiv11\left(12\right)
\end{smallmatrix}\right.$. Writing $\det g\sim\pi^{t}\overline{\pi}^{s}$, we can assume $t\equiv\tau(g)\equiv2\left(3\right)$
(by taking $g^{-1}$ if needed), and we then have $s\equiv1\left(3\right)$
by (\ref{eq:hpi_hpibar_l}). From this follows that $\left[\frac{p\pi}{q^{\star}}\right]_{3}=1$,
which we have shown to yield $G_{p,q}=PSU_{3}(\mathbb{F}_{q})$.
\end{proof}
\begin{example}
\label{exa:p5_q3}The set $S_{5}'$ (see (\ref{eq:set-Sp}),(\ref{eq:Spprime}))
is of size $5^{2}+5+1=31$, comprising the matrices
\begin{gather*}
\scalebox{0.9}{\text{\ensuremath{\left(\begin{array}{ccc}
2i-1 & 0 & 0\\
0 & 2i-1 & 0\\
0 & 0 & -2i-1
\end{array}\right)},\ensuremath{\left(\begin{array}{ccc}
2i-1 & 0 & 0\\
0 & 1 & 2\\
0 & -2 & 1
\end{array}\right)},\ensuremath{\left(\begin{array}{ccc}
2i-1 & 0 & 0\\
0 & 1 & 2i\\
0 & 2i & 1
\end{array}\right)},\ensuremath{\left(\begin{array}{ccc}
2i-1 & 0 & 0\\
0 & 1 & -2i\\
0 & -2i & 1
\end{array}\right)},}}\\
\scalebox{0.9}{\text{\ensuremath{\left(\begin{array}{ccc}
1 & \pi\  & \pi\\
\pi\  & 1 & -\pi\\
\pi\  & -\pi\  & 1
\end{array}\right)},\ensuremath{\left(\begin{array}{ccc}
1 & \pi\  & -\overline{\pi}\\
\pi\  & 1 & \overline{\pi}\\
\overline{\pi} & -\overline{\pi} & 1
\end{array}\right)},\ensuremath{\left(\begin{array}{ccc}
1 & -\pi\  & -\pi\\
-\pi\  & 1 & -\pi\\
-\pi\  & -\pi\  & 1
\end{array}\right)},\ensuremath{\left(\begin{array}{ccc}
1 & -\overline{\pi} & -\overline{\pi}\\
\overline{\pi} & 1 & -\pi\\
\overline{\pi} & -\pi\  & 1
\end{array}\right)},\ensuremath{\left(\begin{array}{ccc}
1 & -\overline{\pi} & -\pi\\
\overline{\pi} & 1 & \overline{\pi}\\
-\pi\  & -\overline{\pi} & 1
\end{array}\right)}}}
\end{gather*}
(where $\pi=1+i$) and their conjugations by permutation matrices.
Taking $q=3$ we have $G_{5,3}=PSU_{3}\left(\mathbb{F}_{3}\right)$,
and the Cayley complex obtained from the generating sets $S_{5}'$
or $S_{5}=S_{5}'\sqcup\{s^{*}|s\in S_{5}'\}$ is non-tripartite. Geometrically,
$S_{5}'$ corresponds to the colored adjacency operator $A_{1}$ from
(\ref{eq:Hecke_Ai_def}). The spectrum of $Cay\left(G_{5,3},S'_{5}\!\mod3\right)$,
as well as the $L^{2}$-spectrum of $A_{1}$ on the building of $PGL_{3}\left(\mathbb{Q}_{5}\right)$,
are shown in Figure \ref{fig:PGU3_q3_p5}.
\end{example}

Denote $G=PGU_{3}$, and recall that $Y_{q}$ denotes $\mathbb{P}^{2}\mathbb{F}_{q}$
when $q\equiv1\left(4\right)$ and $\left\{ v\in\mathbb{P}^{2}\mathbb{F}_{q}[i]\,\middle|\,v^{*}\!\cdot\!v=0\right\} $
when $q\equiv3\left(4\right)$. In the latter case, fix $\varepsilon\in\mathbb{F}_{q}[i]$
with $N(\varepsilon)=-1$. For any ring $R$ and $H\leq G(R)$, denote
\begin{equation}
H[q]:=\begin{cases}
\left\{ A\in H\,\middle|\,\widetilde{A}\equiv\left(\begin{smallmatrix}* & * & *\\
0 & * & *\\
0 & * & *
\end{smallmatrix}\right)\Mod{q}\right\}  & q\equiv1\Mod{4}\\
\left\{ A\in H\,\middle|\,\left(\begin{smallmatrix}1 & 0 & \overline{\varepsilon}\end{smallmatrix}\right)A\left(\begin{smallmatrix}1\\
0\\
\varepsilon
\end{smallmatrix}\right)\equiv0\Mod{q}\right\}  & q\equiv3\Mod{4},
\end{cases}\label{eq:=00005Bq=00005D-def}
\end{equation}
where we recall that $\widetilde{A}$ is the image of $A$ under $i\mapsto\sqrt{-1}\colon R[i]\rightarrow\mathbb{F}_{q}$.
\begin{lem}
\label{lem:Yq}For any odd $q$, $G(\mathbb{F}_{q})$ acts transitively
on $Y_{q}$ with stabilizer $G(\mathbb{F}_{q})[q]$.
\end{lem}

\begin{proof}
If $q\equiv1\left(4\right)$, then $A\mapsto\widetilde{A}$ induces
$G(\mathbb{F}_{q})\cong PGL_{3}(\mathbb{F}_{q})$, and the latter
acts transitively on $\mathbb{P}^{2}\mathbb{F}_{q}$ with stabilizer
$\scalebox{0.9}{\text{\ensuremath{\left(\begin{smallmatrix}*  &  *  &  *\\
 0  &  *  &  *\\
 0  &  *  &  * 
\end{smallmatrix}\right)}}}$. If $q\equiv3\left(4\right)$, the Hermitian forms $I$ and $J=\scalebox{0.8}{\text{\ensuremath{\left(\begin{smallmatrix}  &   &  1\\
  &  -1\\
 1 
\end{smallmatrix}\right)}}}$ are equivalent over $\mathbb{F}_{q}$; indeed, $B=\left(\begin{smallmatrix}1 &  & 1\\
 & \varepsilon+i\varepsilon\\
\varepsilon &  & -\varepsilon
\end{smallmatrix}\right)$ satisfies $B^{*}B=2J$, so that $g\mapsto B^{-1}gB\colon GU_{3}(\mathbb{F}_{q})\overset{{\scriptscriptstyle \cong}}{\longrightarrow}GU_{3}(\mathbb{F}_{q},J)$.
If $v=\left[\alpha\!:\!\beta\!:\!\gamma\right]\in\mathbb{P}^{2}\mathbb{F}_{q}[i]$
and $0=v^{*}Jv=Tr(\alpha\overline{\gamma})-N(\beta)$ then $(\alpha,\gamma)\neq(0,0)$.
Applying $J$ if needed we can assume $\gamma\neq0$, so that $v=\left[\alpha\!:\!\beta\!:\!1\right]$
with $Tr(\alpha)=N(\beta)$. The last equality implies $\left(\begin{smallmatrix}1 & \overline{\beta} & \alpha\\
 & 1 & \beta\\
 &  & 1
\end{smallmatrix}\right)\in GU_{3}(\mathbb{F}_{q},J)$, showing that $GU_{3}(\mathbb{F}_{q},J)$ acts transitively on $\left\{ v\in\mathbb{P}^{2}\mathbb{F}_{q}[i]\,\middle|\,v^{*}Jv=0\right\} $,
hence so does $G(\mathbb{F}_{q})$ on $Y_{q}$. The stabilizer of
$[1\!:\!0\!:\!0]$ in $GU_{3}(\mathbb{F}_{q},J)$ can be computed
to be 
\begin{equation}
\begin{alignedat}{1}GU_{3}(\mathbb{F}_{q},J)\cap\left\{ \left(\begin{smallmatrix}* & * & *\\
0 & * & *\\
0 & * & *
\end{smallmatrix}\right)\right\}  & =\left\{ \alpha\left(\begin{smallmatrix}1 & \overline{\beta} & N(\beta)/2+xi\\
 & \gamma & \beta\gamma\\
 &  & N(\gamma)
\end{smallmatrix}\right)\,\middle|\,\begin{matrix}\alpha,\gamma\in\mathbb{F}_{q}[i]^{\times}\\
\beta\in\mathbb{F}_{q}[i],x\in\mathbb{F}_{q}
\end{matrix}\right\} \\
 & =GU_{3}(\mathbb{F}_{q},J)\cap\left\{ \left(\begin{smallmatrix}* & * & *\\
0 & * & *\\
0 & 0 & *
\end{smallmatrix}\right)\right\} =GU_{3}(\mathbb{F}_{q},J)\cap\left\{ \left(\begin{smallmatrix}* & * & *\\
* & * & *\\
0 & * & *
\end{smallmatrix}\right)\right\} ,
\end{alignedat}
\label{eq:many_iwahories}
\end{equation}
so the $G(\mathbb{F}_{q})$-stabilizer of $B[1\!:\!0\!:\!0]\!=\![1\!:\!0\!:\!\varepsilon]$
is all $g$ with $\frac{1}{2}\left(\begin{smallmatrix}1\\
0\\
\varepsilon
\end{smallmatrix}\right)^{*}\negmedspace g\left(\begin{smallmatrix}1\\
0\\
\varepsilon
\end{smallmatrix}\right)\!=\!\left(B^{-1}gB\right)_{3,1}\negthickspace=\!0$.
\end{proof}
Let us now define $K_{\ell}:=G(\mathbb{Z}_{\ell})$ for every odd
prime $\ell$, let 
\begin{equation}
\widetilde{K}_{2}:=\left\{ A\in GU_{3}\left(\mathbb{Z}_{2}\right)\,\middle|\,A\equiv\left(\begin{smallmatrix}1 & * & *\\
* & 1 & *\\
* & * & 1
\end{smallmatrix}\right)\Mod{2+2i}\right\} \label{eq:K2-tilde}
\end{equation}
and let $K_{2}$ be the image of $\widetilde{K}_{2}$ in $G(\mathbb{Z}_{2})$.
Let $\boldsymbol{K}=\prod_{\ell}K_{\ell}$ (agreeing with (\ref{eq:BigK})),
and $\boldsymbol{K}^{p}=\prod_{\ell\neq p}K_{\ell}$.
\begin{lem}
\label{lem:K2-split}The set $\widetilde{K}_{2}$ (hence also $K_{2}$)
is a group, and $GU_{3}\left(\mathbb{Z}_{2}\right)=\widetilde{K}_{2}\rtimes GU_{3}\left(\mathbb{Z}\right)$.
\end{lem}

\begin{proof}
If $g\in GU_{3}\left(\mathbb{Z}_{2}\right)$ and $g^{*}g=\lambda I$,
then $\lambda\equiv1\Mod{4}$, since $\lambda\equiv3\left(4\right)$
would give $N_{\mathbb{Q}_{2}[i]/\mathbb{Q}_{2}}\left(\det g\right)=\lambda^{3}\equiv3\left(4\right)$,
which is impossible. From this point, the proof proceeds as that of
Prop.\ \ref{prop:split-GU}, with $\mathbb{Z}_{2}[i]$ replacing
the Gaussian Integers.
\end{proof}
\begin{lem}
\label{lem:Kpq-adel}If $p\neq q$ are odd primes then $G(\mathbb{Q})\boldsymbol{K}^{p}[q]G_{p}=G(\mathbb{A}^{\infty})$
and $\Lambda_{p}\boldsymbol{K}^{p}[q]=\boldsymbol{K}^{p}$.
\end{lem}

\begin{proof}
Denote $K'_{2}=\left\{ g\in U_{3}(\mathbb{Z}_{2})\,\middle|\,g\equiv\left(\begin{smallmatrix}u & * & *\\
* & u & *\\
* & * & u
\end{smallmatrix}\right)\Mod{2+2i},u\in\left\{ \pm1,\pm i\right\} \right\} $ (which is again a group), and let $Q=U_{3}(\mathbb{Q})K{}_{2}'U_{3}(\hat{\mathbb{Z}}^{2,p})[q]G_{p}$.
Strong approximation implies $SU_{3}(\mathbb{A}^{\infty})\subseteq Q$,
and we need to show that $\det(Q)=U_{1}(\mathbb{A}^{\infty})$ to
obtain $Q=U_{3}(\mathbb{A}^{\infty})$; As $K'_{2}$ projects modulo
the center onto $K_{2}$, this would yield $G(\mathbb{Q})\boldsymbol{K}^{p}[q]G_{p}=Q/Z(Q)=PU_{3}(\mathbb{A}^{\infty})=G(\mathbb{A}^{\infty})$
as desired (see Lemma \ref{lem:PGU=00003DPU}). Now, $\alpha\mapsto\diag(1,\alpha,1)$
embeds $U_{1}(\mathbb{Q})$ and $U_{1}(\hat{\mathbb{Z}}^{2})$ in
$Q$ (whether $q\equiv1\left(4\right)$ or $q\equiv3\left(4\right)$),
so that $\det(Q)\supseteq U_{1}(\mathbb{Q})U_{1}(\hat{\mathbb{Z}}^{2})$;
and if $\alpha\in U_{1}(\mathbb{Z}_{2})$ then some $u\in\left\{ \pm1,\pm i\right\} $
satisfies $u\equiv\alpha^{-1}\Mod{2+2i}$, and then $\diag(u^{2}\alpha,u,u)\in K_{2}'$
shows that $U_{1}(\mathbb{Z}_{2})\subseteq\det(Q)$. Together, we
have $\det(Q)\supseteq U_{1}(\mathbb{Q})U_{1}(\hat{\mathbb{Z}})=U_{1}(\mathbb{A}^{\infty})$
by (\ref{eq:U1-classnumber1}). In a similar manner, since $\boldsymbol{K}^{p}$
is open in $G(\mathbb{A}^{p,\infty})$, strong approximation implies
that $\Lambda_{p}=G(\mathbb{Q})\cap\boldsymbol{K}^{p}$ is dense in
$\boldsymbol{K}'^{p}:=PSU_{3}(\hat{\mathbb{Z}}^{p})\cap\boldsymbol{K}^{p}$
(where $\hat{\mathbb{Z}}^{p}=\prod_{\ell\neq p}\mathbb{Z}_{\ell}$).
Thus $\boldsymbol{K}'^{p}\leq\Lambda_{p}\boldsymbol{K}^{p}[q]\leq\boldsymbol{K}^{p}$,
and $\det(\boldsymbol{K}^{p}[q])=\det(\boldsymbol{K}^{p})$ gives
as before $\Lambda_{p}\boldsymbol{K}^{p}[q]=\boldsymbol{K}^{p}$.
\end{proof}
\begin{prop}
\label{prop:Ypq-quotient}Denoting by $Y^{p,q}$ the Schreier graph
$Sch(Y_{q},S_{p,q})$, we have:
\begin{enumerate}
\item If $p\equiv1\left(4\right)$ then $Y^{p,q}$ coincides with the one-skeleton
of $X_{\Lambda_{p}[q]}$.
\item If $p\equiv3\left(4\right)$ then $Y^{p,q}$ is the graph induced
by non-backtracking $2$-walks on the hyperspecial side of $X_{\Lambda_{p}[q]}$.
\end{enumerate}
\end{prop}

\begin{proof}
Lemmas \ref{lem:Kpq-adel} and \ref{lem:Yq} give 
\begin{align*}
\Lambda_{p}[q]\backslash\Lambda_{p} & =\left(\Lambda_{p}\cap\boldsymbol{K}^{p}[q]\right)\backslash\Lambda_{p}\cong\boldsymbol{K}^{p}[q]\backslash\Lambda_{p}\boldsymbol{K}^{p}[q]=\boldsymbol{K}^{p}[q]\backslash\boldsymbol{K}^{p}\\
 & =G(\mathbb{Z}_{q})[q]\backslash G(\mathbb{Z}_{q})\cong G(\mathbb{F}_{q})[q]\backslash G(\mathbb{F}_{q})\cong Y_{q}
\end{align*}
($G(\mathbb{Z}_{q})\rightarrow G(\mathbb{F}_{q})$ is onto by Hensel's
Lemma and the fact that $G$ is unramified over $\mathbb{Z}_{q}$
for odd $q$). We obtain that 
\[
Y^{p,q}\cong Sch\left(\Lambda_{p}[q]\backslash\Lambda_{p},S_{p,q}\right)=\Lambda_{p}[q]\backslash Cay(\Lambda_{p},S_{p}),
\]
which is only a Schreier graph since $\Lambda_{p}[q]$ is not normal
in $\Lambda_{p}$. From this point, the proof continues as that of
Proposition \ref{prop:Xpq_quotient}.
\end{proof}

\section{\label{sec:Representation-theory}Representation theory}

In this section we translate the spectral analysis of complexes and
gates to the language of representation theory. We fix an odd prime
$p$, and denote $G=PGU_{3}$, $G_{\infty}=G(\mathbb{R})=PU(3)$,
$G_{p}=G(\mathbb{Q}_{p})$, $K_{p}=G(\mathbb{Z}_{p})$, $\Gamma_{p}=G(\mathbb{Z}[1/p])$,
and $\Lambda_{p}$ as in Theorem \ref{thm:simp-tran}. We also use
$\Lambda$ to denote any finite index subgroup of $\Gamma_{p}$, so
that in particular $\Lambda$ is a lattice both in $G_{p}$ and in
$G_{\infty}\times G_{p}$.\footnote{The $\Lambda$ we shall eventually consider are $\Gamma_{p}$ for
Theorem \ref{thm:main}(1), $\Lambda_{p}$ for \ref{thm:main-2}(1),
$\Lambda_{p}(q)$ for \ref{thm:main-2}(2,3), and $\Lambda_{p}[q]$
for \ref{thm:main}(2,3).} Finally, we denote by $V^{H}$ the $H$-fixed vectors of a representation
$V$, and recall that a representation $V$ of $G_{p}$ is called
\emph{spherical} if $V^{K_{p}}\neq0$, and \emph{Iwahori-spherical}
if $V^{\mathcal{I}_{p}}\neq0$ for an Iwahori subgroup $\mathcal{I}_{p}\leq G_{p}$.

Let $\mathcal{B}$ be the Bruhat-Tits building of $G_{p}$, and $\mathcal{B}^{sp}$
its hyperspecial vertices, identified as a $G_{p}$-set with $G_{p}/K_{p}$.
Since $K_{p}$ is compact, we can identify
\begin{equation}
L^{2}(X_{\Lambda}^{sp})=L^{2}(\Lambda\backslash\mathcal{B}^{sp})\cong L^{2}(\Lambda\backslash G_{p}/K_{p})\cong L^{2}(\Lambda\backslash G_{p})^{K_{p}},\label{eq:adj-to-rep}
\end{equation}
where $G_{p}$ acts on $L^{2}(\Lambda\backslash G_{p})$ by right
translation. If $\Lambda$ acts transitively on $\mathcal{B}^{sp}$
(for example, any $\Lambda_{p}\leq\Lambda\leq\Gamma_{p}$), then $\Lambda K_{p}=G_{p}$,
which implies 
\begin{equation}
\begin{alignedat}{1}\lqu{(\Lambda\cap G(\mathbb{Z}))}{G_{\infty}} & \cong\lqu{(\Lambda\cap K_{p})}{G_{\infty}}\cong\lqu{\Lambda}{\left(G_{\infty}\times\nicefrac{\Lambda}{(\Lambda\cap K_{p})}\right)}\\
 & \cong\lqu{\Lambda}{\left(G_{\infty}\times\nicefrac{G_{p}}{K_{p}}\right)}\cong\lqu{\Lambda}{\left(G_{\infty}\times G_{p}\right)}/\raisebox{-1pt}{\ensuremath{K_{p}}}
\end{alignedat}
\label{eq:lambda_ginf}
\end{equation}
where $\Lambda$ acts diagonally on $G_{\infty}\times G_{p}$ and
$K_{p}$ acts only on $G_{p}$. Taking $\Lambda=\Lambda_{p},\Gamma_{p}$
gives
\begin{equation}
\begin{aligned}L^{2}\left(PU(3)\right) & \cong L^{2}\left(\Lambda_{p}\backslash\left(G_{\infty}\times G_{p}\right)\right)^{K_{p}}\\
L^{2}\left(G(\mathbb{Z})\backslash PU(3)\right) & \cong L^{2}\left(\Gamma_{p}\backslash\left(G_{\infty}\times G_{p}\right)\right)^{K_{p}}.
\end{aligned}
\label{eq:hec-to-rep}
\end{equation}
We say that $S\subseteq\Gamma_{p}$ is \emph{$K_{p}$-stable} if $K_{p}SK_{p}=SK_{p}$,
or equivalently, $K_{p}Sv_{0}=Sv_{0}$. Such an $S$ defines a geometric
operator $T_{S}$ on $\mathcal{B}^{sp}=G_{p}v_{0}$, and a \emph{Hecke
operator} $T_{S}^{\infty}$ on $L^{2}\left(PU(3)\right)$ (and its
quotients):\emph{ 
\begin{equation}
\begin{aligned}\left(T_{S}f\right)\left(gv_{0}\right) & =\sum\nolimits _{s\in S}f\left(gsv_{0}\right) & \  & \left(f\in L^{2}\left(\mathcal{B}^{sp}\right),g\in G_{p}\right)\\
\left(T_{S}^{\infty}f\right)\left(g\right) & =\sum\nolimits _{s\in S}f\left(s^{-1}g\right) & \  & \left(f\in L^{2}\left(PU(3)\right),g\in PU(3)\right).
\end{aligned}
\label{eq:geom-Hecke-ops}
\end{equation}
}Since $S$ is $K_{p}$-stable, for any $G_{p}$-representation $V$
the space $V^{K_{p}}$ is $\alpha_{S}$-stable, where $\alpha_{S}:=\sum_{s\in S}s\in\mathbb{C}G_{p}$.\footnote{In the language of Hecke algebras, $\alpha_{S}$ corresponds to the
element $\one_{K_{p}SK_{p}}\in C_{c}\left(K_{p}\backslash G_{p}/K_{p}\right)$.} Under the identification in (\ref{eq:adj-to-rep}), $T_{S}$ corresponds
to the action of $\alpha_{S}$ on $L^{2}\left(\Lambda\backslash G_{p}\right)^{K_{p}}$,
so that 
\[
\Spec\left(T_{S}|_{L^{2}\left(X_{\Lambda}^{sp}\right)}\right)=\bigcup\nolimits _{i}\Spec\big(\alpha_{S}|_{V_{\smash{i}}^{K_{p}}}\big),\quad\text{where \quad}L^{2}\left(\Lambda\backslash G_{p}\right)={\textstyle \widehat{\bigoplus}_{i}V_{i}}
\]
is a decomposition of $L^{2}\left(\Gamma\backslash G_{p}\right)$
into irreducible $G_{p}$-representations. The trivial eigenvalues
of $T_{S}$ are the ones which arise from one-dimensional $V_{i}$,
as $G'_{p}=[G_{p},G_{p}]$ acts trivially on such $V_{i}$. Similarly,
for the building we obtain 
\[
\Spec\left(T_{S}|_{L^{2}\left(\mathcal{B}^{sp}\right)}\right)=\bigcup_{V\in\,\supp\mu_{pl}}\negthickspace\negthickspace\negthickspace\negthickspace\Spec\left(\alpha_{S}|_{V^{K_{p}}}\right),\quad\text{where \quad}L^{2}\left(G_{p}\right)=\int_{\widehat{G_{p}}}^{\oplus}V\,d\mu_{pl}\left(V\right),
\]
and $\mu_{pl}$ is the Plancherel measure of $G_{p}$. Recall that
$V$ is called \emph{tempered} if it is weakly contained in $L^{2}\left(G_{p}\right)$,
namely, $V\in\supp\mu_{pl}$. This leads to the following criterion:
\begin{cor}
\label{cor:rep-to-Ram}If every irreducible spherical subrepresentation
of $L^{2}\left(\Lambda\backslash G_{p}\right)$ is either finite-dimensional
or tempered, then $X_{\Lambda}$ is a Ramanujan graph/complex.
\end{cor}

\begin{proof}
As the finite-dimensional representations of $G_{p}$ are one-dimensional,
it follows from the discussion so far that for any geometric operator
$T$ on hyperspecial vertices, the nontrivial spectrum of $T|_{X_{\Lambda}}$
is contained in the spectrum of $T|_{L^{2}\left(\mathcal{B}^{sp}\right)}$.
When $G_{p}\cong PGL_{3}(\mathbb{Q}_{p})$, Theorem \ref{thm:old-ram-def}
shows that $X_{\Lambda}$ is then a Ramanujan complex in the strong
sense; the reason behind this is that operators on other cells can
be encoded in Iwahori-spherical representations (see e.g.\ \cite{Lubetzky2017RandomWalks}),
and every infinite-dimensional irreducible unitary representation
of $PGL_{3}(\mathbb{Q}_{p})$ which is Iwahori-spherical but not spherical
happens to be tempered (see \cite{kang2010zeta} or \cite{Golubev2013triangle}).
When $G_{p}\not\cong PGL_{3}(\mathbb{Q}_{p})$, $X_{\Lambda}$ is
a bipartite quotient of the tree $\mathcal{B}$, and for any bipartite
graph $\mathcal{G}=\left(\mathcal{G}_{L}\sqcup\mathcal{G}_{R},E\right)$
one has $\Spec_{0}A_{\mathcal{G}}\subseteq\{0\}\cup\{\pm\sqrt{\lambda}\,|\,\lambda\in\Spec_{0}A_{\mathcal{G}}^{2}\big|_{\mathcal{G}_{L}}\}$.
In our case, $\left(X_{\Lambda}\right)_{L}=\Lambda\backslash\mathcal{B}^{sp}$
and (\ref{eq:spec-bireg-tree}) show that
\begin{align*}
\Spec_{0}A_{X_{\Lambda}} & \subseteq\left\{ 0\right\} \cup\left\{ \pm\sqrt{\lambda}\,\middle|\,\lambda\in\Spec_{0}A_{X_{\Lambda}^{sp}}^{2}\right\} \\
 & \subseteq\{0\}\cup\left\{ \pm\sqrt{\lambda}\,\middle|\,\lambda\in\Spec A_{\mathcal{B}}^{2}\big|_{L^{2}\left(\mathcal{B}^{sp}\right)}\right\} =\Spec(A_{\mathcal{B}}),
\end{align*}
hence $X_{\Lambda}$ is a Ramanujan graph in the sense of Definition
\ref{def:Ram-graph}.
\end{proof}
\begin{rem}
\label{rem:converse-ram-local}If $G_{p}\cong PGL_{3}(\mathbb{Q}_{p})$,
then the converse to Cor.\ \ref{cor:rep-to-Ram} holds as well, by
\cite[Prop.\ 1.5]{Lubotzky2005a}. In the biregular graph case, we
will show in \cite{Ballantine2018ExplicitCayleyRamanujan} that the
converse holds if one assumes the stronger Definition \ref{def:Ram-complex},
and that Definition \ref{def:Ram-graph} does not suffice.
\end{rem}

We turn to the archimedean place. Assuming now that $S\subseteq\Lambda_{p}$,
the operator $T_{S}^{\infty}$ corresponds under (\ref{eq:hec-to-rep})
to $\alpha_{S}$ acting on $L^{2}\left(\Lambda_{p}\backslash\left(G_{\infty}\times G_{p}\right)\right)^{K_{p}}$,
since $(\alpha_{S}f)(\Lambda_{p}\left(g,1\right))=\sum\nolimits _{s\in S}f(\Lambda_{p}\left(g,s\right))=\sum\nolimits _{s\in S}f(\Lambda_{p}(s^{-1}g,1))$.
We obtain 
\[
\Spec\left(T_{S}^{\infty}|_{L^{2}\left(PU(3)\right)}\right)=\bigcup\nolimits _{i}\Spec\left(\alpha_{S}|_{V_{i}^{K_{p}}}\right),\quad\text{where \quad}L^{2}\left(\Lambda_{p}\backslash\left(G_{\infty}\times G_{p}\right)\right)={\textstyle \widehat{\bigoplus}_{i}V_{i}}
\]
is a decomposition of $L^{2}\left(\Lambda_{p}\backslash\left(G_{\infty}\times G_{p}\right)\right)$
as a $G_{p}$-representation; and more generally, $\Spec\left(T_{S}^{\infty}|_{L^{2}\left((\Lambda\cap G(\mathbb{Z}))\backslash PU(3)\right)}\right)=\bigcup\nolimits _{j}\Spec\Big(\alpha_{S}|_{V_{j}^{K_{p}}}\Big)$
where $L^{2}\left(\Lambda\backslash\left(G_{\infty}\times G_{p}\right)\right)=\widehat{\bigoplus}_{j}V_{j}$.
\begin{cor}
\label{cor:rep-to-Ram-gates}If $\Lambda$ acts transitively on $\mathcal{B}^{sp}$,
$S\subseteq\Lambda$ is $K_{p}$-stable, and every irreducible spherical
$G_{p}$-subrepresentation of $L^{2}\left(\Lambda\backslash\left(G_{\infty}\times G_{p}\right)\right)$
is trivial or tempered, then the nontrivial spectrum of $T_{S}^{\infty}$
on $L^{2}(\left(\Lambda\cap G(\mathbb{Z})\right)\backslash PU(3))$
is contained in $\Spec(T_{S}|_{L^{2}(\mathcal{B}^{sp})})$.
\end{cor}

\subsection{\label{subsec:Explicit-spectral-computations}Explicit spectral computations}

Here we study the operators arising from $S_{p}^{\left(\ell\right)}$
and $S_{p}^{'\left(\ell\right)}$ (see (\ref{eq:set-Sp}), (\ref{eq:Spprime}),
(\ref{eq:S(l)})). In particular, we compute their growth, spectrum
and covering rate under the Ramanujan assumption.
\begin{prop}
\label{prop:T_Sp_action}The sets $S_{p}^{\left(\ell\right)},S_{p}^{'\left(\ell\right)}$
are $K_{p}$-stable, and $T_{S_{p}^{(\ell)}},\smash{T_{S_{p}^{'\left(\ell\right)}}}$
from (\ref{eq:geom-Hecke-ops}) coincide on $\mathcal{B}^{sp}$ with
the spherical sum operators $A^{(\ell)}$ and $\smash{A_{1}^{(\ell)}}$
from Definition \ref{def:Al_A1l_def}, respectively.
\end{prop}

\begin{proof}
Since $G_{p}$ preserves $\mathcal{B}^{sp}$, and $K_{p}$ preserves
vertex colors in $\mathcal{B}$, Prop.\ \ref{prop:ST-action}(\ref{enu:Sp-neighbors})
implies that $S_{p}$ and $S_{p}'$ are $K_{p}$-stable, and\emph{
$K_{p}$}-stability of any $S$ implies that of $S^{\ell}$ and $S^{\left(\ell\right)}$.
The second claim follows from Prop.\ \ref{prop:ST-action}(\ref{enu:Spl-sphere}).
\end{proof}
For any $K_{p}$-stable $S$, denote 
\[
\lambda_{ram}\left(S\right):=\max\left\{ \left|\lambda\right|\,\middle|\,\lambda\in\Spec\left(T_{S}|_{L^{2}\left(\mathcal{B}^{sp}\right)}\right)\right\} ,
\]
and $\lambda_{triv}\left(S\right)=\left|S\right|$. By definition,
$\lambda_{ram}(S)$ bounds the nontrivial eigenvalues of $T_{S}$
on Ramanujan complexes, and if the assumptions in Cor.\ \ref{cor:rep-to-Ram-gates}
hold, also of $T_{S}^{\infty}$ on $L^{2}\left(PU(3)\right)$.
\begin{prop}
\label{prop:Spl-spec}For $p\equiv1\Mod{4}$:
\begin{align*}
\lambda_{triv}\big(S_{p}^{\left(\ell\right)}\big) & =\left(\ell+1\right)p^{2\ell}+2\ell p^{2\ell-1}+2\ell p^{2\ell-2}+\left(\ell-1\right)p^{2\ell-3}\\
\lambda_{ram}\big(S_{p}^{\left(\ell\right)}\big) & ={\textstyle {\ell+3 \choose 3}\frac{\ell+2}{2}p^{\ell}-{\ell+1 \choose 3}\frac{3\ell+8}{2}p^{\ell-1}+{\ell+1 \choose 3}\frac{3\ell-8}{2}p^{\ell-2}-{\ell-1 \choose 3}\frac{\ell-2}{2}p^{\ell-3}}\\
\lambda_{triv}\big(S_{p}^{'\left(\ell\right)}\big) & =\left(p^{2\ell}+p^{2\ell-1}+p^{2\ell-2}\right)\left(1+{\textstyle \sum_{i=2}^{\ell}}\,p^{-i}\right)\\
\lambda_{ram}\big(S_{p}^{'\left(\ell\right)}\big) & ={\textstyle {\ell+2 \choose 2}p^{\ell}-{\ell-1 \choose 2}p^{\ell-3}}.
\end{align*}
For $p\equiv3\Mod{4}$:
\begin{align*}
\lambda_{triv}\big(S_{p}^{\left(\ell\right)}\big) & =p^{4\ell}+p^{4\ell-3}\\
\lambda_{ram}\big(S_{p}^{\left(\ell\right)}\big) & =\left(\ell+1\right)p^{2\ell}+\ell p^{2\ell-3}\left(p^{2}-p-1\right)+p^{2\ell-3}.\ \ \ \ \ \ \ \ \ \ \ \ \ \ \ \ \ \ \ \ \ \ \ 
\end{align*}
\end{prop}

\begin{proof}
The representation of $G_{p}$ which gives rise to $\lambda_{ram}$
is the unitary induction of the trivial representation of the Borel
group, and its matrix coefficient (w.r.t.\ its unique $K_{p}$-spherical
vector) is Harish-Chandra's $\Xi$-function \cite{Haagerup1988}.
For $p\equiv3\Mod{4}$, the result can be obtained by evaluating the
modular function of the Borel group on each level in the tree. For
$p\equiv1\Mod{4}$, $\lambda_{ram}(S)=\Spec(\alpha_{S}|_{\mathrm{Ind}_{B}^{G}\mathbf{1}})$
can be computed as in \cite[§V.3]{macdonald1979symmetric}, by summing
up Hall-Littlewood polynomials: 
\[
\lambda_{ram}\big(S_{p}^{\left(\ell\right)}\big)=\sum\nolimits _{i=0}^{\ell}P_{\left(\ell,i,0\right)}(1,1,1;\tfrac{1}{p}),\qquad\lambda_{ram}\big(S_{p}^{'\left(\ell\right)}\big)=\sum\nolimits _{i=0}^{\ell/2}P_{\left(\ell-i,i,0\right)}(1,1,1;\tfrac{1}{p})
\]
(For more details, see \cite[§3.2]{Kaufman2019Freeflagsover}). As
for $\lambda_{triv}$, when $p\equiv3\Mod{4}$ it is simply the size
of the $2\ell$-sphere in a $\left(p^{3}+1,p+1\right)$-biregular
tree, centered at a vertex of degree $p^{3}+1$. For $p\equiv1\Mod{4}$
one can use the same Hall-Littlewood polynomials, evaluated at $(p,1,\nicefrac{1}{p};\nicefrac{1}{p})$
(these yield the character of the Borel whose unitary induction contains
the trivial representation), or argue directly: the Iwahori projection
takes the $\ell$-sphere $\bigcup_{0\leq r\leq\ell}K_{p}\diag(1,p^{r},p^{\ell})v_{0}$
in $\mathcal{B}\left(PGL_{3}(\mathbb{Q}_{p})\right)$ to the $\ell$-sphere
$\bigcup_{\minmax\left(a,b,c\right)=\left\{ 0,\ell\right\} }\diag(p^{a},p^{b},p^{c})v_{0}$
in the fundamental apartment. The fiber size of each vertex in this
projection is then given by explicit computation of Weyl lengths (cf.\ \cite[§4]{Kaufman2019Freeflagsover}).
Similarly, the $\ell$-sphere with respect to $\mathrm{dist}_{1}$
projects onto $\bigcup_{\min\left(a,b,\ell-a-b\right)=0}\diag(p^{a},p^{b},p^{\ell-a-b})v_{0}$.
\end{proof}
From these computations we obtain the following conclusion:
\begin{prop}
\label{prop:Sp-growth-and-AOAC}In terms of Definition \ref{def:golden-gate},
\begin{enumerate}
\item \label{enu:growth}The gate sets $S_{p},S_{p}'$ satisfy the exponential
growth condition.
\item If every spherical $G_{p}$-subrepresentation of $L^{2}\left(\Lambda\backslash\left(G_{\infty}\times G_{p}\right)\right)$
is either trivial or tempered, then $S_{p}$ and $S_{p}^{'}$ satisfy
the a.o.a.c.\ property on $\left(\Lambda\cap G(\mathbb{Z})\right)\backslash PU(3)$.
\end{enumerate}
\end{prop}

\begin{proof}
\emph{(1)} follows from Prop.\ \ref{prop:Spl-spec} since $|S^{\left(\ell\right)}|=\lambda_{triv}(S^{\left(\ell\right)})$,
and we move to \emph{(2)}. By \cite[Prop.\ 3.1]{Parzanchevski2018SuperGoldenGates},
if a finite set $S\subset G_{\infty}$ satisfies 
\begin{equation}
\left\Vert T_{S^{\left(\ell\right)}}^{\infty}\big|_{L_{0}^{2}\left(G_{\infty}\right)}\right\Vert =O\Big(\ell^{c}\sqrt{\lambda_{triv}(S^{\smash{\left(\ell\right)}})}\Big)\label{eq:ao-spec}
\end{equation}
for some $c>0$ (where $L_{0}^{2}\left(G_{\infty}\right)=\left\{ f\in L^{2}\left(G_{\infty}\right)\,\middle|\,f\bot\one\right\} $),
then $S$ is a.o.a.c.. Furthermore, if (\ref{eq:ao-spec}) holds with
$L_{0}^{2}(G_{\infty})$ replaced by $L_{0}^{2}(G_{\infty})^{C}\cong L_{0}^{2}(C\backslash G_{\infty})$
for some finite $C\leq G_{\infty}$, the same proof shows that $S$
is a.o.a.c.\ on $C\backslash G_{\infty}$, and we fix now $C=\Lambda\cap G(\mathbb{Z})$.
Next we observe that if $S$ is $K_{p}$-stable then $(T_{S}^{\infty})^{*}=T_{S^{-1}}^{\infty}$,
and since all Hecke operators commute with each other (see e.g.\ \cite{Lubotzky2005a}),
this implies that $T_{S}^{\infty}$ is normal, so that $\Vert T_{S}^{\infty}|_{L_{0}^{2}(C\backslash G_{\infty})}\Vert=\max\{\left|\lambda\right|\,|\,\lambda\in\Spec T_{S}^{\infty}|_{L_{0}^{2}(C\backslash G_{\infty})}\}$.
It follows from Cor.\ \ref{cor:rep-to-Ram-gates} that $\Vert T_{S}^{\infty}|_{L_{0}^{2}(C\backslash G_{\infty})}\Vert\leq\lambda_{ram}(S)$
(the trivial representation of $G_{\infty}$ corresponds to $\one$).
Taking now $S=S_{p}^{(\ell)}$ or $S_{p}'^{(\ell)}$ (which are $K_{p}$-stable
by Prop.\ \ref{prop:T_Sp_action}) we obtain from Prop.\ \ref{prop:Spl-spec}
that $\left\Vert T_{S^{\left(\ell\right)}}^{\infty}\big|_{L_{0}^{2}\left(C\backslash G_{\infty}\right)}\right\Vert =O\Big(\ell^{c}\sqrt{\lambda_{triv}(S^{\smash{\left(\ell\right)}})}\Big)$,
hence $S_{p}$ and $S_{p}'$ are a.o.a.c.\ on $\left(\Lambda\cap G(\mathbb{Z})\right)\backslash PU(3)$.
\end{proof}

\subsection{\label{subsec:Automorphic-representations}Automorphic representations}

The description of geometric and Hecke operators in terms of representations
of $G_{p}$ is already useful, as seen in §\ref{subsec:Explicit-spectral-computations}.
However, to use the methods currently available for proving a Ramanujan
result we must recast the analysis in terms of automorphic representations.
Recall that an automorphic representation of $G$ over $\mathbb{Q}$
is an irreducible representation $\pi$ of $G(\mathbb{A})$ which
is weakly contained in $L^{2}(G(\mathbb{Q})\backslash G(\mathbb{A}))$
\cite{borel1979automorphic}. Such $\pi$ decomposes as a restricted
tensor product $\pi=\otimes_{v}^{'}\pi_{v}$, where $\pi_{v}$, called
the \emph{local $v$-factor }of $\pi$, is an irreducible $G_{v}$-representation
(cf.\ \cite{flath1979decomposition}). We say that $\pi$ is \emph{spherical,
Iwahori-spherical }or \emph{tempered at $p$} if $\pi_{p}$ is of
the corresponding class. If $\pi^{K}\neq0$ for a compact open $K\leq G(\hat{\mathbb{Z}})$,
then $\pi$ is said to be \emph{of level }$K$. We now fix an odd
$p$ and let $G,G_{p},\Gamma_{p},K_{p},\Lambda_{p},\Lambda$ be as
in §\ref{sec:Representation-theory}. We denote $G_{\ell}=G(\mathbb{Q}_{\ell})$
for all $\ell$, and $G(\hat{\mathbb{Z}}^{p}):=\prod_{\ell\neq p}G(\mathbb{Z}_{\ell})$.

\begin{prop}
\label{prop:autrep-Ram}Let $K=K^{p}\times K_{p}$ be a compact open
subgroup of $G(\hat{\mathbb{Z}})$ with $K^{p}\leq G(\hat{\mathbb{Z}}^{p})$,
such that $G(\mathbb{Q})\cap K^{p}\subseteq\Lambda$.\footnote{For example, $K=\{g\in G(\hat{\mathbb{Z}})|g\equiv I\left(N\right)\}$
with $p\nmid N$, for which $G(\mathbb{Q})\cap K^{p})=\Gamma_{p}(N)$.}
\begin{enumerate}
\item \label{enu:Ram}If every automorphic representation of $G/\mathbb{Q}$
with trivial $\infty$-factor and of level $K$, is either finite-dimensional
or tempered at $p$, then $X_{\Lambda}$ is a Ramanujan complex.
\item If $\Lambda$ acts transitively on $\mathcal{B}^{sp}$, and every
automorphic representation of $G/\mathbb{Q}$ of level $K$ is either
trivial or tempered at $p$, then for any $K_{p}$-stable $S\subseteq\Lambda$,
the nontrivial spectrum of $T_{S}^{\infty}$ on $L^{2}\left((\Lambda\cap G(\mathbb{Z}))\backslash PU(3)\right)$
is contained in the spectrum of $T_{S}$ on $L^{2}\left(\mathcal{B}^{sp}\right)$.
\item \label{enu:GG}If the assumptions in (2) hold and furthermore $\Lambda$
contains $S_{p}$ (see (\ref{eq:set-Sp})), then $S_{p}$ and $S_{p}^{'}$
satisfy a.o.a.c.\ on $(\Lambda\cap G(\mathbb{Z}))\backslash PU(3)$.
\end{enumerate}
\end{prop}

\begin{proof}
\emph{(1) }Note that $\Lambda_{K}:=G(\mathbb{Q})\cap K^{p}$ is of
finite index in $\Lambda$, since both are lattices in $G_{p}$. As
a finite quotient of a Ramanujan complex is Ramanujan, it is enough
to prove that $X_{\Lambda_{K}}$ is Ramanujan. Denote $G_{0}:=G(\mathbb{Q})K^{p}G_{p}$,
and note that $PSU_{3}(\mathbb{A})\leq G_{0}G_{\infty}\leq G(\mathbb{A})$
by strong approximation. Furthermore, $G_{0}G_{\infty}$ is of finite
index in $G(\mathbb{A})$ since $[G(\mathbb{Q}_{\ell}):PSU_{3}(\mathbb{Q}_{\ell})]\leq3$
for all $\ell$, and for almost all $\ell$ we have $G(\mathbb{Z}_{\ell})\subseteq K^{p}$,
implying $G(\mathbb{Q}_{\ell})\subseteq G_{0}G_{\infty}$. We obtain
an inclusion
\begin{align}
X_{\Lambda_{K}}^{sp} & \cong(G(\mathbb{Q})\cap K^{p})\backslash G_{p}/K_{p}\cong G(\mathbb{Q})\backslash\left(\nicefrac{G(\mathbb{Q})}{G(\mathbb{Q})\cap K^{p}}\times\nicefrac{G_{p}}{K_{p}}\right)\nonumber \\
 & \cong G(\mathbb{Q})\backslash\left(\nicefrac{G(\mathbb{Q})K^{p}}{K^{p}}\times\nicefrac{G_{p}}{K_{p}}\right)=G(\mathbb{Q})\backslash G(\mathbb{Q})K^{p}G_{p}/K\label{eq:st_app}\\
 & =G(\mathbb{Q})\backslash G_{0}/K=G(\mathbb{Q})\backslash G_{0}G_{\infty}/KG_{\infty}\subseteq G(\mathbb{Q})\backslash G(\mathbb{A})/KG_{\infty}.\nonumber 
\end{align}
This gives an embedding of $L^{2}(X_{\Lambda_{K}}^{sp})\cong L^{2}(\Lambda_{K}\backslash G_{p})^{K_{p}}$
into $L^{2}\left(G(\mathbb{Q})\backslash G(\mathbb{A})\right)^{KG_{\infty}}$,
under which $T_{S}$ acts as $\alpha_{S}\in\mathbb{C}G_{p}\subseteq\mathbb{C}G(\mathbb{A})$
(see §\ref{sec:Representation-theory}). The decomposition $L^{2}\left(G(\mathbb{Q})\backslash G(\mathbb{A})\right)=\int^{\oplus}\pi$,
where $\pi$ runs over all automorphic representations of $G/\mathbb{Q}$,
yields $L^{2}\left(X_{\Lambda_{K}}^{sp}\right)\subseteq\int^{\oplus}\pi^{KG_{\infty}}$.
Let $\rho$ be an irreducible, infinite-dimensional, spherical subrepresentation
of $L^{2}(\Lambda_{K}\backslash G_{p})$. It has a $K_{p}$-fixed
vector $0\ne f\in L^{2}(\Lambda_{K}\backslash G_{p})^{K_{p}}$, which
determines the representation $\rho$. By the inclusion above, we
may view this vector as a $KG_{\infty}$-fixed vector $f\in L^{2}(G(\mathbb{Q})\backslash G(\mathbb{A}))^{KG_{\infty}}$,
which determines an infinite-dimensional automorphic representation
$\pi$ of $G$, with trivial $\infty$-factor, of level $K$, whose
local $p$-factor is precisely $\rho$ \cite{flath1979decomposition}.
Thus, $\rho$ is tempered, and we can apply Cor.\ \ref{cor:rep-to-Ram}.

\emph{(2,3)} Again it is enough to look at $T_{S}^{\infty}$ acting
on $\Lambda_{K}\backslash\left(G_{\infty}\times G_{p}\right)/K_{p}$,
which is a finite cover of $(\Lambda\cap G(\mathbb{Z}))\backslash PU(3)$
(see (\ref{eq:lambda_ginf})). This time,
\[
\Lambda_{K}\backslash\left(G_{\infty}\times G_{p}\right)/K_{p}\cong G(\mathbb{Q})\backslash G_{0}G_{\infty}/K\subseteq G(\mathbb{Q})\backslash G(\mathbb{A})/K,
\]
and we see that every irreducible spherical $G_{p}$-subrepresentation
of $L^{2}\left(\Lambda\backslash\left(G_{\infty}\times G_{p}\right)\right)\leq L^{2}\left(\Lambda_{K}\backslash\left(G_{\infty}\times G_{p}\right)\right)$
is trivial or tempered by the same analysis as in \emph{(1),} save
that $\pi_{\infty}$ is no longer trivial (in fact, it corresponds
to a constituent in the regular representation $L^{2}\left(G_{\infty}\right)$).
Now \emph{(2)} follows from Cor.\ \ref{cor:rep-to-Ram-gates}, and
\emph{(3)} follows from Prop.\ \ref{prop:Sp-growth-and-AOAC}(2).
\end{proof}

\section{\label{sec:Ramanujan-type-theorems}Ramanujan-type theorems}

In this section we prove a Ramanujan type result for $PGU_{3}^{E}$
(Theorem \ref{thm:ram-gen}). Let us begin with a few notations. The
reader should consult \cite{Rogawski1990Automorphicrepresentationsunitary}
for more information. Throughout this section fix the imaginary quadratic
field $E/\mathbb{Q}$, and denote by $U_{3}$, $GU_{3}$ and $PGU_{3}$
a the corresponding unitary groups over $E/\mathbb{Q}$ with respect
to a definite hermitian form. Let $U_{2,1}=U_{2,1}^{E}$ be the \emph{quasi-split
}unitary group, obtained from the hermitian form $J=\left(\begin{smallmatrix} &  & 1\\
 & -1\\
1
\end{smallmatrix}\right)$, $U_{1,1}$ the unitary group of the hermitian form $\left(\begin{smallmatrix} & i\\
-i
\end{smallmatrix}\right)$, and $U_{0,1}=U_{1}$. We call a representation $\pi$ of $G(\mathbb{Q}_{p})$
\emph{unramified} if it is spherical ($\pi^{G(\mathbb{Z}_{p})}\neq0$)
and in addition $p$ is unramified in $E$. Accordingly, an automorphic
representation of $PGU_{3}^{E}/\mathbb{Q}$ is \emph{unramified at
$p$}, if its local $p$-factor is unramified.
\begin{thm*}[\ref{thm:ram-gen}]
Let $\pi$ be an infinite-dimensional automorphic representation
of $PGU_{3}^{E}$ over $\mathbb{Q}$, which is Iwahori-spherical at
some prime which ramifies in $E$. Then the unramified local factors
of $\pi$ are tempered.
\end{thm*}
The proof of Theorem \ref{thm:ram-gen} makes use of several deep
results:
\begin{enumerate}
\item Rogawski's classification of the automorphic spectrum of unitary groups
in three variables \cite{Rogawski1990Automorphicrepresentationsunitary}.
\item Shin's proof of the Ramanujan-Petersson conjecture for cohomological
self-dual representations of $GL_{n}$ over CM fields \cite{Shin2011Galoisrepresentationsarising},
which itself uses the Fundamental Lemma \cite{Ngo2010Lelemmefondamental},
and previous works by Harris-Taylor, Clozel, Kottwitz and especially
Deligne. For an account of the proof see \cite{shin2020construction}.
\item The following instances of Langlands functoriality, proved for $U_{3}$
in \cite{Rogawski1990Automorphicrepresentationsunitary}:

\textbf{Base change.} Local base change associates to every representation
$\pi_{v}$ of $U_{i,1}(\mathbb{Q}_{v})$ a representation $\Pi_{v}=BC_{v}(\pi_{v})$
of $GL_{i+1}(E_{v})$. Global base change states that for any automorphic
representation $\pi$ of $U_{i,1}/\mathbb{Q}$, there exists an automorphic
representation $\Pi=BC(\pi)$ of $GL_{i+1}/E$, called the global
base change of $\pi$, such that $BC\left(\pi\right)_{p}=BC_{p}\left(\pi_{p}\right)$
for every $p$ which is unramified in $E$.\footnote{Conjecturally, this also holds at the ramified primes - see Remark
\ref{rem:ramified-tempered}.} This is proved for all $i$, cf.\ \cite{Shin2011Galoisrepresentationsarising}.

\textbf{Endoscopy.} Local endoscopic transfer associates to every
representation $\rho_{v}$ of $(U_{i-1,1}\times U_{1})(\mathbb{Q}_{v})$
a representation $\pi_{v}=ET_{v}(\rho_{v})$ of $U_{i,1}(\mathbb{Q}_{v})$.
Global endoscopic transfer associates with an automorphic representation
$\rho$ of $(U_{i-1,1}\times U_{1})/\mathbb{Q}$ an automorphic representation
$\pi=ET(\rho)$ of $U_{i,1}/\mathbb{Q}$, called the global endoscopic
transfer of $\rho$, such that $ET\left(\rho\right)_{p}=ET_{p}\left(\rho_{p}\right)$
for every unramified $p$. This is conjectured to hold for all $i$,
and was proved in \cite[§12,13]{Rogawski1990Automorphicrepresentationsunitary}
for $i=1,2$, which suffice for us. See \cite{Mok2015Endoscopicclassificationrepresentations}
for related results.

\textbf{Inner form transfer.} The group $U_{3}$ is an inner form
of the quasi-split group $U_{2,1}$, and the two groups are isomorphic
over all nonarchimedean local fields (see \cite{Jacobowitz1962Hermitianformsover}).
To every discrete automorphic representation $\pi$ of $U_{3}$, \cite[§14]{Rogawski1990Automorphicrepresentationsunitary}
attaches a cohomological, discrete automorphic representation $\pi'$
of $U_{2,1}$, such that $\pi_{p}\cong\pi_{p}'$ at all primes $p$.
\end{enumerate}
Recall that an automorphic representation $\pi$ of a reductive group
$G$ over a global field $F$ is called \emph{discrete} if it is isomorphic
to a subrepresentation of $L^{2}(G(F)\backslash G(\mathbb{A}_{F}))$,
\emph{cohomological} if $\pi_{\infty}$ appears in the cohomology
of the Lie algebra of $G_{\infty}$ (for a classification of such
representations of $U_{2,1}$ see \cite[§3]{Rogawski1992Analyticexpressionnumber}),
and \emph{$F$-cuspidal} if it does not embed in a parabolically induced
representation from a proper $F$-parabolic subgroup. An automorphic
representation $\pi$ of $U_{i,1}$ over $\mathbb{Q}$ is said to
be \emph{stable} if its global base change $\Pi=BC(\pi)$ is an $E$-cuspidal
automorphic representation of $GL_{i+1}$ over $E$.
\begin{rem}
We shall make use of the fact that local base-change and endoscopic
transfer both preserve temperedness for spherical representations:
they preserve Satake parameters \cite{Minguez2011Unramifiedrepresentationsunitary},
and a spherical representation is tempered if and only if its Satake
parameters lie in a compact subgroups of the Langland dual \cite{gross1998satake}
(for a more detailed treatment of quasi-split unitary groups, see
\cite[Thm.\ 2.5.1(2)]{Mok2015Endoscopicclassificationrepresentations}
and \cite[§2.6]{Minguez2011Unramifiedrepresentationsunitary}).
\end{rem}

\begin{thm}[\cite{Rogawski1990Automorphicrepresentationsunitary}]
\label{thm:rep-types}Every discrete automorphic representation of
$U_{2,1}$ over $\mathbb{Q}$ is of one of the following five types:
\begin{enumerate}
\item Stable cuspidal representations of $U_{2,1}$.
\item One-dimensional representations of $U_{2,1}$.
\item Endoscopic transfer of stable cuspidal representations of $U_{1,1}\times U_{1}$.
\item Endoscopic transfer of one-dimensional representations of $U_{1,1}\times U_{1}$.
\item Endoscopic transfer of stable cuspidal representations of $U_{1}\times U_{1}\times U_{1}$
(obtained via endoscopic transfer from $\left(U_{0,1}\times U_{1}\right)\times U_{1}$
to $U_{1,1}\times U_{1}$, and from the latter to $U_{2,1}$).
\end{enumerate}
\end{thm}

\begin{proof}
This follows from Rogawski's classification of the discrete automorphic
representations of $U_{2,1}$ \cite[§11, §13]{Rogawski1990Automorphicrepresentationsunitary}.
There, Rogawski decomposes the discrete automorphic spectrum of $U_{2,1}$
into three disjoint sets of packets, $\Pi_{a}$, $\Pi_{e}$ and $\Pi_{s}$.
Types \emph{(1)} and \emph{(2)} above comprise what Rogawski refers
to as the \emph{stable }discrete automorphic spectrum, denoted $\Pi_{s}$.
Types \emph{(3)} and \emph{(5)} comprise Rogawski's \emph{endoscopic
}discrete automorphic spectrum, denoted $\Pi_{e}$. Type \emph{(4)}
comprises what Rogawski calls the \emph{A-packets}, and denotes by
$\Pi_{a}$.
\end{proof}
\begin{prop}
\label{prop:tempered-types}The unramified local factors of a cohomological
discrete automorphic representation $\pi$ of type (1), (3) or (5)
above, are tempered.
\end{prop}

\begin{proof}
The representation $\pi$ is either stable, or an endoscopic lift
of a stable cuspidal representation $\sigma$ of $U_{1,1}\times U_{1}$
or $U_{1}\times U_{1}\times U_{1}$. We first show that if $\rho$
is an irreducible, discrete, stable cohomological representation of
$U_{i,1}$ ($i\geq0$), then the unramified local factors of $\rho$
are tempered. Denote by $BC(\rho)$ the automorphic base change of
$\rho$ from $U_{i,1}/\mathbb{Q}$ to $GL_{i+1}/E$. By \cite{Shin2011Galoisrepresentationsarising},
$BC(\rho)$ is a cohomological, conjugate self dual, discrete irreducible
representation. Moreover, since $\rho$ is stable, $BC(\rho)$ is
cuspidal. Hence by the Ramanujan-Petersson result of Shin \cite[Cor.\ 1.3]{Shin2011Galoisrepresentationsarising},
$BC(\rho)$ is tempered at all primes. Since the global base change
is compatible with the local base change at unramified primes, and
the latter preserves temperedness of unramified representations, the
unramified local factors of $\rho$ are tempered. Next, if $\sigma$
is an irreducible discrete stable cohomological representation of
$U_{1,1}\times U_{1}$, then it factors as $\rho_{1}\otimes\rho_{2}$
with $\rho_{i}$ of the same type, and since the unramified local
factors of each factor are tempered, the same holds for their tensor
product. Similar considerations apply to $U_{1}\times U_{1}\times U_{1}$,
and we obtain that the unramified local factors of $\sigma$ are tempered.
Since the local endoscopic lift preserves temperedness of unramified
representations, the same holds for $\pi$.
\end{proof}
Our next goal (Prop.\ \ref{prop:not-type-4}) is to show that representations
of type \emph{(4)} cannot be Iwahori-spherical at a ramified prime,
and we start with some general preparatory results. Let $\mathbb{E}/\mathbb{F}$
be \emph{any} quadratic Galois extension of nonarchimedean local fields,
let $\mathcal{O}=\mathcal{O}_{\mathbb{E}}$ and fix a uniformizer
$\varpi=\varpi_{\mathbb{E}}$. Denote $\tilde{G}=GL_{3}(\mathbb{E})$,
$\tilde{K}=\mathbb{E}^{\times}GL_{3}(\mathcal{O})$ and $G=U_{2,1}(\mathbb{E}/\mathbb{F})$
(which is isomorphic to $U_{3}(\mathbb{E}/\mathbb{F})$). As in Remark
\ref{rem:B-fixed-section}, we observe the involution $g^{\#}=J(g^{*})^{-1}J^{-1}$
(where $J=\left(\begin{smallmatrix} &  & 1\\
 & -1\\
1
\end{smallmatrix}\right)$) so that $G=\{g\in\tilde{G}\,|\,g^{\#}=g\}$. Denote by $\mathcal{B}=\mathcal{B}(G)$
the Bruhat-Tits tree of $G$, and by $\tilde{\mathcal{B}}=\mathcal{B}_{red}(\tilde{G})$
the reduced Bruhat-Tits building of $\tilde{G}=GL_{3}(\mathbb{E})$,
which is equivalent to the building of $PGL_{3}(\mathbb{E})$, and
whose vertices correspond to $\tilde{G}/\tilde{K}$.
\begin{lem}
\label{lem:Iwahori-Kottwitz}The Iwahori subgroups of $G$ are the
stabilizers of edges in $\mathcal{B}$.
\end{lem}

\begin{proof}
By \cite[Prop.\ 3]{Haines2008Parahoric}, each Iwahori subgroup of
$G$ is the intersection of a stabilizer of a maximal face in the
Bruhat-Tits building of $G$ with the kernel of the Kottwitz homomorphism
$\kappa\,:\,G\rightarrow X^{*}(\hat{Z}^{I})$. Recall that $X^{*}(\hat{Z}^{I})$
is the group of algebraic characters of $\hat{Z}^{I}$, where $\hat{Z}$
is the center of the dual group $\hat{G}=GL_{3}(\mathbb{C})$ of $G$
and $I=\mbox{Gal}(\overline{\mathbb{F}}/\mathbb{F})$ is the absolute
Galois group whose action on $\hat{G}$ and $\hat{Z}$ factors through
$\mbox{Gal}(\mathbb{E}/\mathbb{F})=\{1,\sigma\}$ with $\sigma$ acting
via $g\mapsto J(g^{t})^{-1}J^{-1}$. Since $\sigma$ acts on $\hat{Z}$
via $z\mapsto z^{-1}$, we get that $\hat{Z}^{I}=\left\{ \pm I\right\} $,
which is a compact (in fact finite) group hence admits no non-trivial
algebraic characters, i.e.\ $X^{*}(\hat{Z}^{I})=\left\{ 1\right\} $,
which implies $\ker\kappa=G$.
\end{proof}
We fix a fundamental chamber in $\tilde{\mathcal{B}}$:
\[
C_{0}=\left\{ v_{0}=\tilde{K},\;v_{1}=\diag(1,1,\varpi)\tilde{K},\;v_{2}=\diag(1,\varpi,\varpi)\tilde{K}\right\} .
\]
Recall from Remark \ref{rem:B-fixed-section} that $\mathcal{B}$
is embedded as a metric space in the fixed-section $\mathcal{B}^{\#}$
in $\tilde{\mathcal{B}}$. Since $v_{0}^{\#}=v_{0}$, $v_{1}^{\#}=v_{2}$
and $v_{2}^{\#}=v_{1}$, the median $C_{0}^{\#}$ is an edge in $\mathcal{B}^{\#}$,
which we denote by $e_{0}$.
\begin{lem}
\label{lem:e0-in-tree}The edge $e_{0}=C_{0}^{\#}$ of $\mathcal{B}^{\#}$
also forms an edge in the Bruhat-Tits tree $\mathcal{B}$.
\end{lem}

\begin{proof}
Let $\tilde{S}=\{\mbox{diag}(\alpha,\beta,\gamma)\,:\,\alpha,\beta,\gamma\in\mathbb{E}^{\times}\}$
be the standard maximal torus of $\tilde{G}$ and $S=\left\{ \mbox{diag}(\alpha,1,\alpha^{-1})\,:\,\alpha\in\mathbb{F}^{\times}\right\} $
the standard maximal split torus of $G$. Let $\tilde{A}=A_{red}(\tilde{G},\tilde{S},\mathbb{E})\subset\tilde{\mathcal{B}}$
and $A=A(G,S,\mathbb{F})\subset\mathcal{B}$ be the apartments associated
with $\tilde{S}$ and $S$, respectively (cf.\ \cite{Tits1979Reductivegroupsover}),
and note that $\tilde{S}$ and $\tilde{A}$ are stable under $\#$.
By \cite[§1.10]{Tits1979Reductivegroupsover} the apartment $A$
coincides as a metric space with $\tilde{A}\cap\mathcal{B}^{\#}$,
the $\#$-fixed section in $\tilde{A}$. Furthermore, the edge $e_{0}$
of $\B^{\#}$, connecting $v_{0}$ and the median of $v_{1}$ and
$v_{2}$, is also an edge (chamber) of $\B$ by \cite[§1.15, following equation (8)]{Tits1979Reductivegroupsover}
(in fact, this shows that $\tilde{A}\cap\mathcal{B}^{\#}$ coincides
with $A$ as a simplicial complex).
\end{proof}
The point-wise stabilizer of the fundamental chamber $C_{0}$ in $\tilde{G}=GL_{3}(\mathbb{E})$
is 
\[
\tilde{I}:=\mbox{Stab}_{\tilde{G}}\left(C_{0}\right)=\left\{ \left(\begin{array}{ccc}
\mathcal{O}^{\times} & \mathcal{O} & \mathcal{O}\\
\varpi\mathcal{O} & \mathcal{O}^{\times} & \mathcal{O}\\
\varpi\mathcal{O} & \varpi\mathcal{O} & \mathcal{O}^{\times}
\end{array}\right)\right\} ,
\]
and it admits the so-called \emph{Iwahori factorization} $\tilde{I}=\bar{\tilde{N}}^{1}\cdot\tilde{T}^{0}\cdot\tilde{N}^{0}$
with respect to the subgroups
\[
\bar{\tilde{N}}^{1}=\left\{ \left(\begin{smallmatrix}1\\
\varpi\mathcal{O} & 1\\
\varpi\mathcal{O} & \varpi\mathcal{O} & 1
\end{smallmatrix}\right)\right\} ,\ \tilde{T}^{0}=\left\{ \left(\begin{smallmatrix}\mathcal{O}^{\times}\\
 & \mathcal{O}^{\times}\\
 &  & \mathcal{O}^{\times}
\end{smallmatrix}\right)\right\} ,\ \tilde{N}^{0}=\left\{ \left(\begin{smallmatrix}1 & \mathcal{O} & \mathcal{O}\\
 & 1 & \mathcal{O}\\
 &  & 1
\end{smallmatrix}\right)\right\} .
\]

\begin{lem}
\label{lem:Iwahori}The group $I=\tilde{I}\cap G$ is an Iwahori subgroup
of $G$, admitting the Iwahori factorization $I=\bar{N}^{1}\cdot T^{0}\cdot N^{0}$,
where $\bar{N}^{1}=\bar{\tilde{N}}^{1}\cap G$, $T^{0}=\tilde{T}^{0}\cap G$,
and $N^{0}=\tilde{N}^{0}\cap G.$
\end{lem}

\begin{proof}
Observe that by color considerations, an element in $\tilde{G}$ fixes
both $v_{0}$ and $\left\{ v_{1},v_{2}\right\} $ (as a set) if and
only if it fixes $C_{0}$ point-wise. It follows that $I$ equals
$\mathrm{Stab}_{G}\left(e_{0}\right)$ (both point-wise and set-wise),
hence by Lemmas \ref{lem:Iwahori-Kottwitz} and \ref{lem:e0-in-tree}
it is an Iwahori subgroup. The Iwahori factorization for $I$ follows
from that of $\tilde{I}$: If $g\in I$ we can write $g=n_{1}tn_{2}$
with $n_{1}\in\bar{\tilde{N}}^{1}$, $t\in\tilde{T}^{0}$ and $n_{2}\in\tilde{N}^{0}$.
As $g\in G$ we have
\[
n_{1}^{\#}t^{\#}n_{2}^{\#}=g^{\#}=g=n_{1}tn_{2}\quad\Rightarrow\quad(n_{1}^{\#})^{-1}n_{1}=tn_{2}(n_{2}^{\#})^{-1}(t^{\#})^{-1},
\]
and $(n_{1}^{\#})^{-1}n_{1}$ is unipotent lower-triangular, since
$\#$ preserves these properties. Similarly, $tn_{2}(n_{2}^{\#})^{-1}(t^{\#})^{-1}$
is upper-triangular, forcing $(n_{1}^{\#})^{-1}n_{1}=1$, i.e.\ $n_{1}\in\bar{N}^{1}$.
This implies $t^{-1}t^{\#}=n_{2}(n_{2}^{\#})^{-1}$, and as the l.h.s.\ is
diagonal and the r.h.s.\ unipotent upper-triangular, this leads to
$t\in T^{0}$ and $n_{2}\in N^{0}$ as well.
\end{proof}
We are now in a position to prove our main claim:
\begin{prop}
\label{prop:not-type-4}Let $\pi$ be a discrete automorphic representation
of $U_{2,1}$ over $\mathbb{Q}$. If $\pi$ is Iwahori-spherical at
a prime $p_{0}$ that ramifies in $E$, then $\pi$ is not of type
(4).
\begin{proof}
Any automorphic representation of type \emph{(4)} belongs to an A-packet
of the form $\Pi(\xi)=\otimes_{v}\Pi(\xi_{v})$, where $\xi$ is an
automorphic character of $(U_{1,1}\times U_{1})/\mathbb{Q}$ (see
\cite[following Prop.\ 13.1.3]{Rogawski1990Automorphicrepresentationsunitary}).
For each prime $p$, the local packet $\Pi\left(\xi_{p}\right)$ contains
a non-tempered representation $\pi_{p}^{n}\left(\xi\right)$, and
when $p$ does not split in $E$, also a $\mathbb{Q}_{p}$-cuspidal
representation $\pi_{p}^{s}\left(\xi\right)$. To establish the proposition
we shall prove that for a prime $p_{0}$ that ramifies in $E$, neither
$\pi_{p_{0}}^{n}\left(\xi\right)$ nor $\pi_{p_{0}}^{s}\left(\xi\right)$
are Iwahori-spherical. For any $p$, let $G_{p}=U_{2,1}(\mathbb{Q}_{p})$,
let $B_{p}$ be the Borel subgroup of upper triangular matrices and
let $T_{p}$ be the maximal torus
\[
T_{p}=\left\{ d(\alpha,\beta)=\left(\begin{smallmatrix}\alpha\\
 & \beta\\
 &  & \nicefrac{1}{\overline{\alpha}}
\end{smallmatrix}\right)\,\middle|\,\alpha\in E_{p}^{\times},\beta\in U_{1}(\mathbb{Q}_{p})\right\} .
\]
Let $X(T_{p})$ be the group of characters of $T_{p}$, and let $X^{un}(T_{p})$
be the subgroup of unramified characters, namely those trivial on
$T_{p}^{0}=\{d(\alpha,\beta)\in T_{p}\,|\,\alpha,\beta\in\mathcal{O}_{E_{p}}^{\times}\}$.
Let $\Delta_{B_{p}}$ be the modular character of $B_{p}$, which
factors through $B_{p}\twoheadrightarrow T_{p}$, and satisfies $\Delta_{B_{p}}(d(\alpha,\beta))=|\alpha|_{p}^{4}$.
Denote by $I_{B_{p}}^{G_{p}}(\chi)$ the unitary parabolic induction
of $\chi\in X(T_{p})$, i.e.
\[
I_{B_{p}}^{G_{p}}(\chi):=\Ind_{B_{p}}^{G_{p}}\left(\Delta_{B_{p}}^{1/2}\otimes\mathrm{Inf}_{T_{p}}^{B_{p}}\chi\right).
\]
The proposition follows from the following two claims:
\begin{enumerate}
\item For any prime $p$, every Iwahori-spherical irreducible representation
of $G_{p}$ embeds in $I_{B_{p}}^{G_{p}}(\chi)$ for some $\chi\in X^{un}(T_{p})$.
\item For $p_{0}$ that ramifies in $E$, neither $\pi_{p_{0}}^{n}\left(\xi\right)$
nor $\pi_{p_{0}}^{s}\left(\xi\right)$ embed in $I_{B_{p_{0}}}^{G_{p_{0}}}(\chi)$
for any $\chi\in X^{un}(T_{p_{0}})$.
\end{enumerate}
\emph{Proof of (1):} This is Proposition 2.6 in \cite{casselman1980unramified},
in which the group is assumed to be split or unramified, but the proof
applies in our settings as well. We explain the argument: Let $I_{p}=\bar{N_{p}^{1}}T_{p}^{0}N_{p}^{0}$
be the Iwahori subgroup of $G_{p}$ and its Iwahori factorization
from Lemma \ref{lem:Iwahori}. Let $\left(\pi,V\right)$ be an Iwahori-spherical
irreducible representation of $G_{p}$, hence by assumption $V^{I_{p}}\ne0$.
Denoting $V(N_{p})=\langle\{\pi_{p}(n)v-v\,|\,v\in V,n\in N_{p}\}\rangle$,
let $J_{B_{p}}^{G_{p}}(\pi)=V/V(N_{p})$ be the Jacquet module of
$V$, considered as a representation of $T_{p}$. The map $V^{I_{p}}\rightarrow(V/V(N_{p}))^{T_{p}^{0}}$
is an isomorphism (see e.g.\ \cite[Lem.\ 4.7]{borel1976admissible}),
hence $J_{B_{p}}^{G_{p}}(\pi)$ contains a non-zero $T_{p}^{0}$-fixed
vector. As $T_{p}$ is abelian, this implies that $J_{B_{p}}^{G_{p}}(\pi)$
maps onto some $\chi\in X^{un}(T_{p})$. Frobenius reciprocity then
gives $\Hom_{G_{p}}(\pi,I_{B_{p}}^{G_{p}}(\chi))\cong\Hom_{T_{p}}(J_{B_{p}}^{G_{p}}(\pi),\chi)\neq0$,
so being irreducible $\pi$ embeds in $I_{B_{p}}^{G_{p}}(\chi)$.\\
\emph{Proof of (2):} Since $\vphantom{\big|}\pi_{p_{0}}^{s}(\xi)$
is supercuspidal, it does not embed in any induction from a proper
parabolic subgroup of $G_{p}$, and in particular in $I_{B_{p_{0}}}^{G_{p_{0}}}(\chi)$
for any $\chi\in X(T_{p_{0}})$. Let us assume then that $\pi_{p_{0}}^{n}\left(\xi\right)$
embeds in $I_{B_{p_{0}}}^{G_{p_{0}}}(\chi_{0})$ for some $\chi_{0}\in X^{un}(T_{p_{0}})$.
It is shown in \cite[§12.2]{Rogawski1990Automorphicrepresentationsunitary}
that $\pi_{p}^{n}(\xi)$, for any $p$, is a quotient of the parabolic
induction $I_{B_{p}}^{G_{p}}(\chi_{_{\xi,p}})$, where $\chi_{_{\xi,p}}\in X(T_{p})$
is defined by
\[
\chi_{_{\xi,p}}(d(\alpha,\beta)):=\mu_{p}(\alpha)\cdot|\alpha|_{p}^{1/2}\cdot\xi_{p}^{1}(\nicefrac{\alpha}{\overline{\alpha}})\cdot\xi_{p}^{2}((\nicefrac{\alpha}{\overline{\alpha}})\cdot\beta);
\]
Here $\xi_{p}$ is the $p$-factor of $\xi$, which is a character
of $H_{p}=U_{1,1}(\mathbb{Q}_{p})\times U_{1}(\mathbb{Q}_{p})$ and
decomposes as $\xi_{p}(h,\lambda)=\xi_{p}^{1}(\det(h))\xi_{p}^{2}(\det(h)\lambda)$
for some characters $\xi_{p}^{1},\xi_{p}^{2}$ of $U_{1}(\mathbb{Q}_{p})$,
and $\mu_{p}$ is the $p$-factor of $\mu$, where $\mu$ is an automorphic
character of $E$ whose restriction to $\mathbb{Q}$ equals $w_{E/\mathbb{Q}}$,
the character associated by class field theory to $E/\mathbb{Q}$.
The important fact is that $w_{E/\mathbb{Q}}$ is ramified at every
prime which ramifies in $E$. Taking $\alpha\in\mathbb{Z}_{p_{0}}^{\times}$
for which $w_{E/\mathbb{Q}}(\alpha)\neq1$, we obtain $\chi_{_{\xi,p_{0}}}(d(\alpha,1))\neq1$
and thus $\chi_{_{\xi,p_{0}}}\notin X^{un}(T_{p_{0}})$. Altogether,
we obtain a nonzero map $I_{B_{p_{0}}}^{G_{p_{0}}}(\chi_{_{\xi,p_{0}}})\twoheadrightarrow\pi_{p_{0}}^{n}(\xi)\hookrightarrow I_{B_{p_{0}}}^{G_{p_{0}}}(\chi_{0})$
and by \cite[Thm.\ 2.9 for $M=N=T_{p_{0}}$]{Bernstein1977InducedrepresentationsreductiveI}
it follows that $\chi_{_{\xi,p_{0}}}$ is a twist of $\chi_{0}$ by
an element of the Weyl group $W_{G}$. Namely, $\chi_{\xi,p_{0}}(t)=\chi_{0}(\sigma t\sigma^{-1})$
for some $\sigma\in W_{G}=\big\{ I,\scalebox{0.8}{\text{\ensuremath{\left(\begin{smallmatrix}  &   &  1\\
  &  1\\
 1 
\end{smallmatrix}\right)}}}\big\}$, but this is impossible as $\chi_{_{\xi,p_{0}}}$ is ramified and
$\chi_{0}$ is not.
\end{proof}
\end{prop}

We can now prove Theorem \ref{thm:ram-gen}.
\begin{proof}[Proof of Theorem \ref{thm:ram-gen}]
Let $\pi$ be as in the Theorem, and using Lemma \ref{lem:PGU=00003DPU}
consider it as an automorphic representation of $U_{3}$ with a trivial
central character. Letting $\pi'$ denote the inner form transfer
of $\pi$ from $U_{3}$ to $U_{2,1}$, it suffices to prove that the
unramified local factors of $\pi'$ are tempered. Since $\pi$ is
not one-dimensional, neither is $\pi'$, hence it is of type \emph{(1)},\emph{
(3)},\emph{ (4) or (5)} from Theorem \ref{thm:rep-types}. By Prop.\ \ref{prop:not-type-4},
$\pi'$ is not of type \emph{(4)}, since $\pi$, and hence $\pi'$,
are spherical at a prime which ramifies in $E$. Thus, by Prop.\ \ref{prop:tempered-types},
the unramified local factors of $\pi'$ are tempered.
\end{proof}
\begin{rem}
\label{rem:ramified-tempered}Unfortunately, the global base change
as stated in \cite{Shin2011Galoisrepresentationsarising} gives no
information as to what happens at the ramified primes, so that we
cannot extend Theorem \ref{thm:ram-gen} to guarantee temperedness
there. However, we expect that local base change and endoscopic transfer
do preserve temperedness also for ramified representations, and furthermore
that global base change agrees with the local one also at the ramified
primes. Thus, we conjecture that Theorem \ref{thm:ram-gen} holds
without the restriction to unramified factors. It is also widely believed
that an automorphic representation which is tempered at almost all
places is tempered at all places, which gives further support to our
suspicion (an even stronger conjecture appears in \cite[Conj.\ 4]{Clozel2007Spectral}).
\end{rem}

\subsection{\label{subsec:Ramanujan-Lambdap}Ramanujan-type theorems for $\Lambda_{p}$}

Throughout this section we prove Theorem \ref{thm:ram-K2}. We denote
$G=PGU_{3}$, $G_{\ell}=G(\mathbb{Q}_{\ell})$ for all primes $\ell$,
$K_{\ell}=G(\mathbb{Z}_{\ell})$ for odd $\ell$, 
\[
K_{2}:=\left\{ A\in G\left(\mathbb{Z}_{2}\right)\,\middle|\,A\equiv\left(\begin{smallmatrix}1 & * & *\\
* & 1 & *\\
* & * & 1
\end{smallmatrix}\right)\Mod{2+2i}\right\} 
\]
(which is a group by Lemma \ref{lem:K2-split}), $\boldsymbol{K}=\prod_{\ell}K_{\ell}$
(as in (\ref{eq:BigK})) and $\boldsymbol{K}^{p}=\prod_{\ell\neq p}K_{\ell}$.
For any odd $p$ we have 
\begin{equation}
G(\mathbb{Q})\cap\boldsymbol{K}^{p}G_{p}G_{\infty}=G(\mathbb{Q}\cap\hat{\mathbb{Z}}\mathbb{Q}_{p}\mathbb{R})\cap K_{2}=G\left(\mathbb{Z}[1/p]\right)\cap K_{2}=\Lambda_{p},\label{eq:Lambdap_global}
\end{equation}
which is claim \emph{(1)} of Theorem \ref{thm:ram-K2}. To prove the
other claims of Theorem \ref{thm:ram-K2}, along the lines of the
proof of Theorem \ref{thm:ram-gen}, we need to produce a substitute
for Prop.\ \ref{prop:not-type-4}.
\begin{prop}
\label{prop:K2-Iwhori}Let $\pi$ be an automorphic representation
of $G$. If $\pi_{\infty}$ is trivial, $\pi_{2}$ has a $K_{2}$-fixed
vector, $\pi_{q}$ has an Iwahori-fixed vector for some odd prime
$q$, and $\pi_{\ell}$ is spherical for any $\ell\notin\{2,q\}$,
then $\pi'$, the inner form transfer of $\pi$ to $U_{2,1}$, is
not of type (4).
\end{prop}

\begin{proof}
Assume that $\pi'$ is of type \emph{(4)}. As in the proof of Prop.\ \ref{prop:not-type-4},
we have $\pi'\in\Pi(\xi)=\otimes_{v}\Pi(\xi_{v})$ for some automorphic
character $\xi$ of $U_{1,1}\times U_{1}$, and $\pi_{p}'\in\Pi\left(\xi_{p}\right)$
for every prime $p$. Since $\pi_{q}'\cong\pi_{q}$ has an Iwahori-fixed
vector it is not supercuspidal, hence $\pi'_{q}=\pi^{n}\left(\xi_{q}\right)$.
The latter is the Langlands quotient of $I_{B_{p}}^{G_{p}}(\chi_{_{\xi}})$,
with $\chi_{\xi}$ a character of the Borel group of $U_{2,1}(\mathbb{Q}_{q})$.
Furthermore, since $\pi_{q}'$ is Iwahori-spherical, $\chi_{\xi}$
must be unramified \cite{casselman1980unramified} (see also the proof
of claim (2) inside Prop.\ \ref{prop:not-type-4}), which implies
that $\pi_{q}\cong\pi'_{q}\cong\pi^{n}(\xi_{q})$ is unramified as
well (see \cite[pp.\ 451]{Gelbart1991LfunctionsFourier}). As we already
assumed that $\pi_{2}$ is of level $K_{2}$, $\pi_{\ell}$ is unramified
at $\ell\ne2,q$, and $\pi_{\infty}$ is trivial, we see that $\pi$
has level $\boldsymbol{K}G_{\infty}=G_{\infty}\prod_{p}K_{p}$. As
$G(\mathbb{Z})$ embeds in $K_{\ell}$ for every odd $\ell$, and
$G(\mathbb{Z})K_{2}=G(\mathbb{Z}_{2})$ by Lemma \ref{lem:K2-split},
we have $G(\mathbb{Q})\boldsymbol{K}G_{\infty}=G(\mathbb{Q})G(\hat{\mathbb{Z}})G_{\infty}=G(\mathbb{A})$
by (\ref{eq:clnum1}) and Lemma \ref{lem:PGU=00003DPU}. It follows
that the only automorphic representation of $G$ of level $\boldsymbol{K}G_{\infty}$
is the trivial one, which is of type \emph{(2)} rather than \emph{(4)}.
\end{proof}
Claim (\ref{enu:ram-K2-sch}) of Theorem \ref{thm:ram-K2} follows:
As $\boldsymbol{K}\{q\}=K_{2}\mathcal{I}_{q}G(\hat{\mathbb{Z}}^{2,q})$,
Proposition \ref{prop:K2-Iwhori} shows that the inner form transfer
of $\pi$ to $U_{2,1}$ is not of type \emph{(4)}, and from this point,
the proof continues as that of Theorem \ref{thm:ram-gen}. To prove
Claim (\ref{enu:ram-K2-cayley}) of Theorem \ref{thm:ram-K2} we shall
need to introduce a conjecture, formulated together with Brooke Feigon
while working on \cite{Ballantine2018ExplicitCayleyRamanujan}, regarding
the levels of the representations in the A-packet associated with
type \emph{(4)}.
\begin{conjecture}
\label{conj:level-A-packets}Let $\xi=\otimes_{v}\xi_{v}$ be an automorphic
character of $(U_{1,1}\times U_{1})/\mathbb{Q}$, and $\Pi(\xi)=\otimes_{v}\Pi(\xi_{v})$
its associated A-packet of $U_{2,1}$. For a non-split prime $p$,
$\pi_{p}^{n}\left(\xi_{p}\right)$ has the shallowest level in its
local packet, namely, for any $K_{p}'\leq G(\mathbb{Z}_{p})$, 
\[
\pi_{p}^{s}\left(\xi_{p}\right)^{K_{p}'}\ne0\quad\Longrightarrow\quad\pi_{p}^{n}\left(\xi_{p}\right)^{K_{p}'}\ne0.
\]
\end{conjecture}

\begin{prop}
\label{prop:K2-no-type-4} If $\pi$ is an automorphic representation
of $G$ such that $\pi_{2}^{K_{2}}\neq0$, then Conjecture \ref{conj:level-A-packets}
implies that the inner form transfer of $\pi$ to $U_{2,1}$ is not
of type (4).
\end{prop}

\begin{proof}
Let $\pi$ be an automorphic representation of $G$ whose transfer
$\pi'$ to $U_{2,1}$ is of type \emph{(4)}, and continue with the
notations from the proof of Prop.\ \ref{prop:not-type-4}, in particular
$\pi'\in\Pi(\xi)=\otimes_{v}\Pi(\xi_{v})$, and $\pi_{p}\cong\pi_{p}'\in\Pi\left(\xi_{p}\right)$
for any prime $p$. By Conjecture \ref{conj:level-A-packets} it suffices
to prove that $\pi_{2}^{n}\left(\xi_{2}\right)^{K_{2}}=0$ to obtain
that $\pi_{2}$ has no $K_{2}$-fixed vector. Recall that $\pi_{2}^{n}\left(\xi_{2}\right)$
is a composition factor of $I_{B_{2}}^{G_{2}}(\chi_{_{\xi}})=\Ind_{B_{2}}^{G_{2}}(\widetilde{\chi})$,
where $\widetilde{\chi}=\Delta_{B_{2}}^{1/2}\otimes\mathrm{Inf}_{T_{2}}^{B_{2}}\chi_{_{\xi}}$
and $\chi_{\xi}$ is a certain character of a maximal torus $T_{2}$
contained in a Borel group $B_{2}$. Assume first that $\pi_{2}\leq I_{B_{2}}^{G_{2}}(\chi_{_{\xi}})$.
By Mackey theory, 
\begin{gather*}
\pi_{2}^{K_{2}}\leq Hom_{K_{2}}\left(\mathbb{C},Res_{K_{2}}^{G}Ind_{B_{2}}^{G}\widetilde{\chi}\right)\cong Hom_{K_{2}}\Bigg(\mathbb{C},\bigoplus_{g\in K_{2}\backslash G/B_{2}}Ind_{K_{2}\cap{\vphantom{|}}^{g}\!B_{2}}^{K_{2}}Res_{K_{2}\cap{\vphantom{|}}^{g}\!B_{2}}^{{\vphantom{|}}^{g}\!B_{2}}{}^{g}\widetilde{\chi}\Bigg)\\
\cong\bigoplus_{g\in K_{2}\backslash G/B_{2}}Hom_{K_{2}\cap{\vphantom{|}}^{g}\!B_{2}}\left(\mathbb{C},Res_{K_{2}\cap{\vphantom{|}}^{g}\!B_{2}}^{{\vphantom{|}}^{g}\!B_{2}}{}^{g}\widetilde{\chi}\right)\cong\bigoplus_{g\in K_{2}\backslash G/B_{2}}Hom_{K_{2}^{g}\cap B_{2}}\left(\mathbb{C},Res_{K_{2}^{g}\cap B_{2}}^{B_{2}}\widetilde{\chi}\right).
\end{gather*}
Denote by $K_{2}^{st}=G(\mathbb{Z}_{2})$ the standard maximal compact
subgroup of $G_{2}$. As $G_{2}=B_{2}\cdot K_{2}^{st}$ by Iwasawa
decomposition, and $K_{2}^{st}=G(\mathbb{Z})\ltimes K_{2}$ by Lemma
\ref{lem:K2-split}, 
\[
B_{2}\backslash G_{2}/K_{2}=B_{2}\backslash B_{2}K_{2}^{st}/K_{2}\cong\left(K_{2}^{st}\cap B_{2}\right)\backslash K_{2}^{st}/K_{2}=K_{2}\left(K_{2}^{st}\cap B_{2}\right)\backslash K_{2}^{st},
\]
which implies (using $K_{2}\trianglelefteq K_{2}^{st}$) that 
\begin{gather*}
\bigoplus_{g\in K_{2}\backslash G/B_{2}}Hom_{K_{2}^{g}\cap B_{2}}\left(\mathbb{C},Res_{K_{2}^{g}\cap B_{2}}^{B_{2}}\widetilde{\chi}\right)\leq\bigoplus_{g\in K_{2}\backslash K_{2}^{st}}Hom_{K_{2}^{g}\cap B_{2}}\left(\mathbb{C},Res_{K_{2}^{g}\cap B_{2}}^{B_{2}}\widetilde{\chi}\right)\\
=Hom_{K_{2}\cap B_{2}}\left(\mathbb{C},Res_{K_{2}\cap B_{2}}^{B_{2}}\widetilde{\chi}\right)^{\oplus\left|G\left(\mathbb{Z}\right)\right|}=\left(\widetilde{\chi}^{K_{2}\cap B_{2}}\right)^{\oplus\left|G\left(\mathbb{Z}\right)\right|}.
\end{gather*}
Since $\widetilde{\chi}$ is one-dimensional, it is left to find $g\in B_{2}\cap K_{2}$
such that $\widetilde{\chi}\left(g\right)\neq1$ to obtain $\pi_{2}^{K_{2}}=0$.
Choosing $\sqrt{-7}\in\mathbb{Z}_{2}$ with $\sqrt{-7}\equiv1\Mod{4}$,
and denoting
\[
A:=\frac{1}{2}\left(\begin{array}{ccc}
\sqrt{-7}+2i & 1+i & 1\\
3+\sqrt{-7}+(1-\sqrt{-7})i & 2 & 3-\sqrt{-7}+(1+\sqrt{-7})i\\
2 & 2+2i & -2\sqrt{-7}+4i
\end{array}\right)\in GL_{3}\left(\mathbb{Q}_{2}\left[i\right]\right),
\]
we have $AA^{*}=\left(\begin{smallmatrix} &  & 1\\
 & -1\\
1
\end{smallmatrix}\right)$. This allows us to describe an explicit Borel subgroup of $G_{2}$:
\begin{align*}
B_{2} & =\left\{ A^{-1}\left(\begin{array}{ccc}
\alpha & \alpha x & \alpha y\\
 & \beta & \beta\overline{x}\\
 &  & \nicefrac{1}{\overline{\alpha}}
\end{array}\right)A\in G_{2}\,\middle|\,{\alpha\in\mathbb{Q}_{2}[i]^{\times},\;\beta\in U_{1}(\mathbb{Q}_{2})\atop x,y\in\mathbb{Q}_{2}[i],\;y+\overline{y}=x\overline{x}}\right\} ,\\
T_{2} & =\left\{ A^{-1}\left(\begin{smallmatrix}\alpha\\
 & \beta\\
 &  & \nicefrac{1}{\overline{\alpha}}
\end{smallmatrix}\right)A\,\middle|\,{\alpha\in\mathbb{Q}_{2}[i]^{\times}\atop \beta\in U_{1}(\mathbb{Q}_{2})}\right\} .
\end{align*}
We note that $\Delta_{B_{2}}\big|_{K_{2}\cap B_{2}}\equiv1$, and
that for any $x,y\in\mathbb{Q}_{2}\left[i\right]$,
\[
\widetilde{\chi}\left(A^{-1}\left(\begin{smallmatrix}3 & 3x & 3y\\
 & 1 & \overline{x}\\
 &  & 1/3
\end{smallmatrix}\right)A\right)=w_{\mathbb{Q}\left[i\right]/\mathbb{Q}}(3)\cdot|3|_{2}^{1/2}\cdot\xi_{2}^{1}(\nicefrac{3}{3})\cdot\xi_{2}^{2}((\nicefrac{3}{3})\cdot1)=-1,
\]
so if there is an element of the form $A^{-1}\left(\begin{smallmatrix}3 & 3x & 3y\\
 & 1 & \bar{x}\\
 &  & 1/3
\end{smallmatrix}\right)A$ (with $y+\overline{y}=x\overline{x}$) inside $K_{2}$, we are done.
A direct computation shows that 
\begin{align*}
g & =A^{-1}\left(\begin{array}{ccc}
3 & 6+6\,i & 12+6\,i\\
0 & 1 & 2-2\,i\\
0 & 0 & 1/3
\end{array}\right)A\\
 & =\frac{1}{3}\left(\begin{smallmatrix}51+74i+14\sqrt{-7} & 19+83i-\left(7+7i\right)\sqrt{-7} & -92+96i-\left(48+46i\right)\sqrt{-7}\\
107+13i+\left(21-21i\right)\sqrt{-7} & 103+36i & 79+109i-\left(65-29i\right)\sqrt{-7}\\
260-96i-\left(48+102i\right)\sqrt{-7} & 131-73i-\left(33+69i\right)\sqrt{-7} & -141-110i-14\sqrt{-7}
\end{smallmatrix}\right)
\end{align*}
is in $K_{2}^{st}$ and satisfies $g\equiv\left(\begin{smallmatrix}3+2i & 0 & 2\,i\\
0 & 1 & 2+2i\\
2i & 2+2i & 3+2i
\end{smallmatrix}\right)\Mod{4}$, hence $g\in K_{2}$ as desired. If $\pi_{2}$ is a quotient of $I_{B_{2}}^{G_{2}}(\chi_{_{\xi}})$,
rather than a subgroup, then $\pi_{2}\leq I_{B_{2}}^{G_{2}}(\sigma.\chi_{_{\xi}})$
for the nontrivial element $\sigma\in W_{G}$ (see claim 2 in Proposition
\ref{prop:not-type-4} and \cite[Thm.\ 2.5]{Bernstein1977InducedrepresentationsreductiveI}),
and the same arguments apply with $g^{\sigma}$ replacing $g$.
\end{proof}
Claim (\ref{enu:ram-K2-cayley}) of Theorem \ref{thm:ram-K2} now
follows along the lines of Theorem \ref{thm:ram-gen}. For small enough
$N$ we can verify Claim (\ref{enu:ram-K2-cayley}) directly, lendi
ng some evidence towards Conjecture \ref{conj:level-A-packets}.
\begin{prop}
\label{prop:cases35}For $N=3,5$, Claim (\ref{enu:ram-K2-cayley})
of Theorem \ref{thm:ram-K2} holds unconditionally.
\end{prop}

\begin{proof}
Fix $q=N=3$ or $5$, and let $\pi$ be an automorphic representation
of $G$ of level $\boldsymbol{K}(q)$. Assume in contradiction that
$\pi$ is infinite-dimensional and that $\pi_{\ell}$ is non-tempered
for some $\ell\nmid2q$. By Theorem \ref{thm:rep-types} and Proposition
\ref{prop:not-type-4}, $\pi$ must be of type (4), hence $\pi\in\Pi(\xi)$
for some automorphic character $\xi$ of $U_{1,1}\times U_{1}$. If
$p\nmid2q$ is another prime, then $\pi_{p}$ is spherical as $K_{p}(q)=G(\mathbb{Z}_{p})$,
hence $\pi_{p}=\pi_{p}^{n}(\xi)$. But since $\pi_{p}^{n}(\xi)$ is
non-tempered for any $\xi$, we find that $\pi_{p}$ is non-tempered.

Our goal now is to give a converse to Proposition \ref{prop:autrep-Ram}(1)
(with $K=\boldsymbol{K}(q)$ and $\Lambda=\Lambda_{K}=\Lambda_{p}(q)$),
showing that this implies that $X^{p,q}$ is not a Ramanujan complex.
However, the complexes $X^{5,3}$ and $X^{13,5}$ are small enough
that we can compute their $A_{1}$-spectrum and verify that they are
Ramanujan, as shown in Figure \ref{fig:PGU3_q3_p5}. We arrive at
a contradiction, hence $\pi_{\ell}$ is tempered for \emph{any} $\ell\nmid2q$,
which completes the proof.

We turn to the missing converse: denoting $G_{0}=G(\mathbb{Q})\boldsymbol{K}(q)^{p}G_{p}$,
one shows by similar arguments to the proof of Lemma \ref{lem:Kpq-adel}
that either $G_{0}G_{\infty}=G(\mathbb{A})$ or $G_{0}G_{\infty}=G(\mathbb{A}^{q})PSU(\mathbb{Q}_{q})$.
In the former case, (\ref{eq:st_app}) becomes 
\[
L^{2}(\Lambda_{p}(q)\backslash G_{p})^{K_{p}}\cong L^{2}(X_{\Lambda_{p}(q)}^{sp})\cong L^{2}\left(G(\mathbb{Q})\backslash G(\mathbb{A})\right)^{\boldsymbol{K}(q)G_{\infty}},
\]
and as $\pi_{p}$ appears in the l.h.s.\ Remark \ref{rem:converse-ram-local}
shows that $X^{p,q}$ is not Ramanujan when $p\equiv1\left(4\right)$.\footnote{The same holds for $p\equiv3(4)$ if one takes Definition \ref{def:Ram-complex}
for graphs as well.} Assume now that $G_{0}G_{\infty}=G(\mathbb{A}^{q})PSU(\mathbb{Q}_{q})$,
which is a normal subgroup of $G(\mathbb{A})$ of index three. We
have that $\rho=\mathrm{Res}_{PSU_{3}(\mathbb{Q}_{q})}^{G(\mathbb{Q}_{q})}\pi_{q}$
is either irreducible, in which case so is $\tau:=\mathrm{Res}_{G_{0}G_{\infty}}^{G(\mathbb{A})}\pi$,
or $\rho\cong\tau_{q}\oplus\chi\tau_{q}\oplus\chi^{2}\tau_{q}$ where
$\tau_{q}$ is irreducible and $\chi$ is a nontrivial linear character
of $G(\mathbb{Q}_{q})$, in which case $\mathrm{Res}_{G_{0}G_{\infty}}^{G(\mathbb{A})}\pi\cong\tau\oplus\chi\tau\oplus\chi^{2}\tau$
with $\tau:=\otimes_{\ell}\begin{cases}
\pi_{\ell} & \ell\neq q\\
\tau_{q} & \ell=q
\end{cases}$. In both cases $\tau$ has a $\boldsymbol{K}(q)$-fixed vector, which
contributes to 
\[
L^{2}(\Lambda\backslash G_{p})^{K_{p}}\cong L^{2}(X_{\Lambda}^{sp})\cong L^{2}\left(G(\mathbb{Q})\backslash G_{0}G_{\infty}\right)^{\boldsymbol{K}(q)G_{\infty}},
\]
and since $\tau_{p}=\pi_{p}$ this shows again that $X^{p,q}$ is
not Ramanujan when $p\equiv1\left(4\right)$.
\end{proof}
\begin{rem}
In the first paragraph of the proof, we essentially proved the following
claim: Let $G/\mathbb{Q}$ be a definite unitary algebraic group in
three variables. Let $\pi=\otimes_{v}\pi_{v}$ be an automorphic representation
of $G$ and let $S_{\pi}$ be the set of places at which $\pi_{p}$
is ramified (including the $p$ which ramify in $E$). Then 
\[
\exists p\not\in S_{\pi}\;:\;\pi_{p}\mbox{ is tempered}\qquad\Longrightarrow\qquad\forall\ell\not\in S_{\pi}\;:\;\pi_{\ell}\mbox{ is tempered}.
\]
Moreover, if the conjecture in Remark \ref{rem:ramified-tempered}
holds then we get the stronger claim
\[
\exists p\not\in S_{\pi}\;:\;\pi_{p}\mbox{ is tempered}\qquad\Longrightarrow\qquad\forall\ell\;:\;\pi_{\ell}\mbox{ is tempered},
\]
verifying a conjecture of Clozel \cite[Conj.\ 4(1)]{Clozel2007Spectral}
for definite $U_{3}$. Similarly, we obtain \cite[Conj.\ 4(2)]{Clozel2007Spectral},
which states that the \emph{weight} of $\pi_{p}$ (see loc.\ cit.)
is independent of $p$ for any $p\not\in S_{\pi}$ for any automorphic
representation $\pi$. Indeed, for any $\xi$, $\pi_{p}^{n}(\xi)$\textcolor{red}{{}
}is of weight $(1,0,-1)$ (equivalently, its Satake parameters are
of absolute value $p^{1/2},1,p^{-1/2}$). More generally, if $\pi$
is of type (4), (2), or (1/3/5), the weight of any unramified local
component of $\pi$ is $(1,0,-1)$, $(2,0,-2)$, or $(0,0,0)$ respectively.
\end{rem}

\section{\label{sec:Golden-gates-and}Golden gates and Ramanujan complexes}

In this section we collect the results from previous sections to prove
Theorems \ref{thm:main} and \ref{thm:main-2}. The Ramanujan property
is given by the various claims in Theorem \ref{thm:ram-K2}, whose
proofs appear in §\ref{subsec:Ramanujan-Lambdap}. Afterwards, we
present an explicit candidate for a super golden gate set. We keep
the notations from §\ref{subsec:Ramanujan-Lambdap} (including $K_{2},\boldsymbol{K}$),
and recall the sets $S_{p}$ and $S_{p,q}=S_{p}\Mod{q}$ from (\ref{eq:set-Sp}).

\begin{thm*}[Theorem \ref{thm:main-2}]
Conjecture \ref{conj:level-A-packets} implies:
\begin{enumerate}
\item The set $S_{p}$ (and also $S_{p}'$) is a golden gate set for $PU(3)$.
\item If $p\equiv1\Mod{4}$ then the Cayley graph $X^{p,q}:=Cay(G_{p,q},S_{p,q})$
is the one-skeleton of a two-dimensional Ramanujan complex.\footnote{The group $G_{p,q}$ is explicitly determined in Table \ref{tab:Gpq},
and the size of $S_{p}$ is in Prop.\ \ref{prop:ST-action}. Some
examples are demonstrated in Example \ref{exa:p5_q3} and Figure \ref{fig:PGU3_q3_p5}.}
\item If $p\equiv3\Mod{4}$ then $X^{p,q}=Cay(G_{p,q},S_{p,q})$ corresponds
to a non-backtracking $2$-walk operator on the left side of a $(p^{3}+1,p+1)$-biregular
Ramanujan graph.
\end{enumerate}
\end{thm*}
\begin{proof}
\emph{(1)} The sets $S_{p}$ and $S_{p}'$ are golden gate sets (assuming
Conjecture \ref{conj:level-A-packets}): they have a.o.a.c.\ on $PU(3)$
by Theorem \ref{thm:ram-K2}(\ref{enu:ram-K2-cayley}) (with $N=1$)
and Prop.\ \ref{prop:autrep-Ram}(\ref{enu:GG}) (with $\Lambda=\Lambda_{p}$),
efficient navigation by §\ref{subsec:Navigation-in-Lambdap}, approximability
by Remark \ref{rem:AOAC-approx}(\ref{enu:approximability}), and
exponential growth by Prop.\ \ref{prop:Sp-growth-and-AOAC}(\ref{enu:growth}).

\emph{(2,3)} For odd primes $p\neq q$, the quotient $X_{\Lambda_{p}\left(q\right)}$
is a Ramanujan graph or complex by Theorem \ref{thm:ram-K2}(\ref{enu:ram-K2-cayley})
(with $N=q$) and Prop.\ \ref{prop:autrep-Ram}(\ref{enu:Ram}) (with
$K=\boldsymbol{K}(q)$ and $\Lambda=\Lambda_{p}(q)=\Lambda_{K}=G(\mathbb{Q})\cap K^{p}$).
On the other hand, by Prop.\ \ref{prop:Xpq_quotient} $X^{p,q}$
coincides with the one-skeleton of $X_{\Lambda_{p}(q)}$ for $p\equiv1\left(4\right)$,
and with $2$-walks on the left side of $X_{\Lambda_{p}(q)}$ for
$p\equiv3\left(4\right)$.
\end{proof}
\begin{thm*}[Theorem \ref{thm:main}]
\begin{enumerate}
\item The set $S_{p}$ (and also $S_{p}'$) is a golden gate set for $PU_{3}(\mathbb{Z})\backslash PU(3)$.
\item If $p\equiv1\Mod{4}$ then the Schreier graph $Y^{p,q}:=Sch(Y_{q},S_{p,q})$
is the one-skeleton of a two-dimensional Ramanujan complex.
\item If $p\equiv3\Mod{4}$ then $Y^{p,q}=Sch(Y_{q},S_{p,q})$ corresponds
to a non-backtracking $2$-walk operator on the left side of a $(p^{3}+1,p+1)$-biregular
Ramanujan graph.
\end{enumerate}
\end{thm*}
\begin{proof}
\emph{(1)} is proved as in Theorem \ref{thm:main-2}, replacing $\Lambda_{p}$
with $\Gamma_{p}$ and Theorem \ref{thm:ram-K2}(\ref{enu:ram-K2-cayley})
with Theorem \ref{thm:ram-gen} (whose proof appears in §\ref{sec:Ramanujan-type-theorems}).

\emph{(2,3)} Since $\Lambda_{p}[q]=G(\mathbb{Q})\cap\boldsymbol{K}^{p}[q]$,
we have $X_{\Lambda_{p}[q]}^{sp}\cong G(\mathbb{Q})\backslash G(\mathbb{A})/G(\mathbb{R})\boldsymbol{K}[q]$
by taking $K=\boldsymbol{K}[q]$ in (\ref{eq:st_app}) (in which the
last inclusion becomes an equality by Lemma \ref{lem:Kpq-adel}).
For $q\equiv1\left(4\right)$, let $\mathcal{I}_{q}=\left\{ g\in K_{q}\,\middle|\,\widetilde{g}\equiv\left(\begin{smallmatrix}* & * & *\\
0 & * & *\\
0 & 0 & *
\end{smallmatrix}\right)\Mod{q}\right\} $, which is the image of the standard Iwahori of $PGL_{3}(\mathbb{Q}_{q})$
under its chosen isomorphism with $G_{q}$. For $q\equiv3\left(4\right)$,
let $\mathcal{I}_{q}=K_{q}[q]$, which is indeed Iwahori: let $\mathcal{I}'_{q}$
be the $PU_{2,1}(\mathbb{Z}_{q})$-preimage of $\left(\begin{smallmatrix}* & * & *\\
0 & * & *\\
0 & 0 & *
\end{smallmatrix}\right)\leq PU_{2,1}(\mathbb{F}_{q})$, which is an Iwahori subgroup of $PU_{2,1}(\mathbb{Q}_{q})$ by Lemma
\ref{lem:Iwahori}; by the equalities in (\ref{eq:many_iwahories}),
$\mathcal{I}'_{q}$ is also the preimage of $\left(\begin{smallmatrix}* & * & *\\
* & * & *\\
0 & * & *
\end{smallmatrix}\right)\leq PU_{2,1}(\mathbb{F}_{q})$. If we choose $\tilde{\varepsilon}\in\mathbb{Z}_{q}[i]$ with $N(\tilde{\varepsilon})=-1$
and take $\tilde{B}=\left(\begin{smallmatrix}1 &  & 1\\
 & \tilde{\varepsilon}+i\tilde{\varepsilon}\\
\tilde{\varepsilon} &  & -\tilde{\varepsilon}
\end{smallmatrix}\right)$, then $g\mapsto\tilde{B}g\tilde{B}^{-1}$ gives an isomorphism from
$PU_{2,1}(\mathbb{Q}_{q})$ to $G(\mathbb{Q}_{q})$, which takes $\mathcal{I}'_{q}$
to $\mathcal{I}_{q}$. In summary, for any odd $q$ we have $\boldsymbol{K}\{q\}=K_{2}\mathcal{I}_{q}G(\hat{\mathbb{Z}}^{2,q})\subseteq\boldsymbol{K}[q]$,
and we conclude from Theorem \ref{thm:ram-K2}(\ref{enu:ram-K2-sch})
and Prop.\ \ref{prop:autrep-Ram}(\ref{enu:Ram}) (with $K=\boldsymbol{K}[q]$
and $\Lambda=\Lambda_{p}[q]$) that $X_{\Lambda_{p}[q]}$ is Ramanujan.
Finally, Prop.\ \ref{prop:Ypq-quotient} identifies $Y^{p,q}$ with
the one-skeleton of $X_{\Lambda_{p}[q]}$ or with $2$-walks on its
left side, as before.
\end{proof}

\subsection{\label{subsec:super-golden-gates}Super golden gates}

In applications to quantum computing, implementation of error correction
schemes requires every fundamental quantum gate to be an operator
of finite order. In \cite{Parzanchevski2018SuperGoldenGates}, golden
gates with this additional property were constructed using arithmetic
lattices which act simply-transitively on the \emph{edges }of a Bruhat-Tits
tree, rather than its vertices. Here we present such a lattice for
$PU(3)$.
\begin{prop}
Let $\sigma=\scalebox{0.9}{\text{\ensuremath{\left(\begin{smallmatrix}  &  1\\
  &   &  1\\
 1 
\end{smallmatrix}\right)}}},\tau=\scalebox{0.9}{\text{\ensuremath{\left(\begin{smallmatrix}-1  &  1\\
 i  &  i\\
  &   &  1-i 
\end{smallmatrix}\right)}}}\in PGU_{3}\left(\mathbb{Z}\left[1/2\right]\right)$. The lattice $\Lambda_{2}:=\left\langle \sigma,\tau\right\rangle $
acts simply-transitively on the edges of the $3$-regular Bruhat-Tits
tree of $PGU_{3}\left(\mathbb{Q}_{2}\right)$, and satisfies $\Lambda_{2}=\left\langle \sigma,\tau\,\middle|\,\sigma^{3},\tau^{3}\right\rangle \cong\nicefrac{\mathbb{Z}}{3\mathbb{Z}}*\nicefrac{\mathbb{Z}}{3\mathbb{Z}}$.
\end{prop}

\begin{proof}
While the tree $\mathcal{B}$ of $G_{2}=PGU_{3}(\mathbb{Q}_{2})$
is regular, the action of $G_{2}$ on $\mathcal{B}^{0}$ has two orbits,
forming a bi-partition of the tree. Explicit computation shows that
$\left\langle \sigma\right\rangle $ fixes the vertex $v_{0}=PGU_{3}(\mathbb{Z}_{2})$,
and acts simply-transitively on its neighbors. The matrix $\tau$
fixes one of these neighbors, which we denote by $v_{1}$, and which
corresponds to the midpoint of the edge $\scalebox{0.9}{\text{\ensuremath{\left(\begin{smallmatrix}1+i  &  1\\
  &  1\\
  &   &  1 
\end{smallmatrix}\right)}}}\negmedspace-\negthickspace\negthickspace-\scalebox{0.9}{\text{\ensuremath{\left(\begin{smallmatrix}1+i  &  1\\
  &  1\\
  &   &  1+i 
\end{smallmatrix}\right)}}}$ in the building $\widetilde{\mathcal{B}}$ of $PGL_{3}(\mathbb{Q}_{2}[i])$.
Furthermore, $\left\langle \tau\right\rangle $ acts simply-transitively
on the neighbors of $v_{1}$ (which are $\left\{ I,\scalebox{0.9}{\text{\ensuremath{\left(\begin{smallmatrix}2  &  2+i\\
  &  1\\
  &   &  1+i 
\end{smallmatrix}\right)}}},\scalebox{0.9}{\text{\ensuremath{\left(\begin{smallmatrix}2  &  1\\
  &  1\\
  &   &  1+i 
\end{smallmatrix}\right)}}}\right\} $ in $\widetilde{\mathcal{B}}$). It follows from Bass-Serre theory
(cf.\ \cite{serre1980trees}) that $\Lambda_{2}$ is a free amalgamation
of the two stabilizers, and that it acts simply-transitively on the
(undirected) edges $\mathcal{B}^{1}$.
\end{proof}
Similarly to the analysis for $\Lambda_{p}$, this shows that $\Lambda_{2}$
has exponential growth, and can be efficiently navigated as in §\ref{subsec:Navigation-in-Lambdap}
(navigation by the action on edges is explained in \cite{Parzanchevski2018SuperGoldenGates}).
However, the spectral analysis of Hecke operators arising from $\Lambda_{2}$
corresponds to the $2$-factor of the associated automorphic representations,
and currently Theorem \ref{thm:ram-gen} does not give temperedness
at ramified primes. As explained in Remark \ref{rem:ramified-tempered},
we expect that the ramified local factors are tempered as well, and
if this is the case, then any Hecke operator which arises from a $K_{p}$-stable
set in $\Lambda_{2}$ would have Ramanujan spectrum on $L^{2}\left(G(\mathbb{Z})\backslash PU(3)\right)$,
and will give us a super golden gate set.
\begin{rem}
The tree associated with $G_{2}$ can also help at other places: For
odd $p$, we have seen in §\ref{sec:Golden-gates-and} that $S_{p}$
(and similarly $S_{p}'$) is a golden gate set for $PU_{3}(\mathbb{Z})\backslash PU(3)$,
which is a $96$-fold quotient of $PU(3)$. Assuming Conjecture \ref{conj:level-A-packets}
we showed that it is in fact golden for $PU(3)$ itself, but even
unconditionally we can do slightly better, showing it is golden for
the $32$-fold quotient $C\backslash PU(3)$, where $C=\{g\in PU_{3}(\mathbb{Z})\,|\,g_{3,3}\neq0\}$.
Indeed, let $\mathcal{I}_{2}$ be the Iwahori subgroup of $G_{2}$
stabilizing both vertices $v_{0}$ and $v_{1}$ above, $K=\mathcal{I}_{2}G(\hat{\mathbb{Z}}^{2})$,
and $\Lambda=G(\mathbb{Q})\cap K^{p}$. One can then verify that $C=\mathcal{I}_{2}\cap G(\mathbb{Z})=\Lambda\cap G(\mathbb{Z})$,
and from (\ref{eq:Omega_p_mod2}) we also see that $S_{p}$ stabilizes
(point-wise) the $1$-ball around $v_{0}$, hence in particular $S_{p}\subseteq\Lambda$.
Thus, Prop.\ \ref{prop:autrep-Ram}(\ref{enu:GG}) can be used to
obtain a.o.a.c., once we show that every automorphic representation
$\pi$ of $G/\mathbb{Q}$ of level $K$ is infinite-dimensional or
tempered; this follows from Theorem \ref{thm:ram-gen}, as $\pi$
is Iwahori-spherical at the ramified place $2$.
\end{rem}

\bibliographystyle{amsalpha}
\bibliography{mybib}

\noindent \begin{flushleft}
\noun{School of Mathematics, Institute for Advanced Study, Princeton,
USA.}\texttt{}~\\
\texttt{shai.evra@gmail.com}\noun{\medskip{}
}\texttt{}~\\
\noun{Einstein Institute of Mathematics, Hebrew University of Jerusalem,
Israel.}\texttt{}~\\
\texttt{parzan@math.huji.ac.il}
\par\end{flushleft}
\end{document}